\documentclass[11pt, a4paper,leqno]{amsart}
\usepackage{amsfonts}
\usepackage{amsmath}
\usepackage{amssymb}
\usepackage{mathrsfs}
\usepackage{graphicx}
\usepackage{amsthm}
\usepackage[unicode=true]{hyperref}%
\providecommand{\U}[1]{\protect\rule{.1in}{.1in}}

\calclayout
\allowdisplaybreaks

\newtheorem{theorem}{Theorem}
\theoremstyle{plain}

\newtheorem{corollary}{Corollary}

\newtheorem{lemma}{Lemma}

\newtheorem{proposition}{Proposition}

\numberwithin{equation}{section}

\theoremstyle{definition}
\newtheorem{definition}{Definition}
\newtheorem{example}{Example}

\theoremstyle{remark}
\newtheorem{remark}{Remark}

\def\R{\mathbb{R}}
\def\T{\mathbb{T}}

\begin{document}
\title[A sharp stability criterion for Euler equations via sparseness]{A sharp stability criterion for Euler equations via sparseness}

\author{\'{O}scar Dom\'{\i}nguez}
\address{O. Dom\'inguez, Departamento de An\'alisis Matem\'atico y Matem\'atica
Aplicada, Facultad de Matem\'aticas, Universidad Complutense de Madrid, Plaza
de Ciencias 3, 28040 Madrid, Spain}
\email{oscar.dominguez@ucm.es}

\author{Mario Milman}
\address{M. Milman, Instituto Argentino de Matematica\\
Buenos Aires\\
Argentina}
\email{mario.milman@icloud.com}
\urladdr{https://sites.google.com/site/mariomilman/}

\thanks{Part of this research was carried out while the first-named author was a postdoc at the Institut Camille Jordan, Lyon, supported by the LABEX
MILYON (ANR-10-LABX-0070) of Universit\'{e} de Lyon, within the program
\textquotedblleft Investissement d'Avenir" (ANR-11-IDEX-0007) operated by the
French National Research Agency (ANR)}
\subjclass{76B03, 42B35, 42B37, 46B70}
\keywords{$H^{-1}$-stability, Euler equations, energy conservation, sparseness, Morrey spaces, approximate solutions, physical solutions,  extrapolation spaces.}

\begin{abstract}
We introduce sparse versions of function spaces that are relevant to
characterize the solutions of Euler equations without concentration. The standard 
Sobolev space $H^{-1}$ is given a sparse structure that
allows to measure the degree of compactness of embeddings into $H^{-1}$ and
provides new quantitative general criteria for $H^{-1}$-stability. 
Indices of sparseness are defined, and function spaces whose indices have
prescribed decay are constructed, resulting in an improvement of the classical
$H^{-1}$-stability results: sparse stability. The analysis relies on the introduction of
sparse Riesz-Morrey-Tadmor spaces, that are characterized via maximal
operators and new sparse domination theorems, together with extrapolation techniques. Our methods also yield
improvements on recent results on the conservation of energy of physically
realizable solutions of $2$D-Euler. 

\end{abstract}
\maketitle
\tableofcontents

\section{Preamble\label{sec:preambulo}}

The classical Euler equations for incompressible fluid flow are given by%

\begin{equation}
\left\{
\begin{array}
[c]{c}%
u_{t}+u\cdot\nabla u=-\nabla p,\\
\text{div }u=0,\\
\text{initial and boundary conditions,}%
\end{array}
\right.  \label{Euler}%
\end{equation}
where $u=(u_{1},\ldots,u_{n})$ is the \emph{velocity} field and $p$ is the
(scalar) \emph{pressure}. Although the Euler equations have been studied for more than two and half
centuries, many important open problems remain unanswered. In particular, while it is
easy to see that smooth solutions of \eqref{Euler} conserve kinetic energy, the existence of
weak solutions that conserve energy or the uniqueness of weak solutions are
more subtle issues. 


\subsection{$H^{-1}$-stability for approximate solutions of Euler equations}\label{Section1Intro}

Research on conservation of energy has been considerably influenced by the work of
DiPerna--Majda \cite{DiPernaMajda, DiPernaMajda1*, DiPernaMajda2}. These authors 
introduced the key concept of \emph{approximate solutions} $\{u^\varepsilon\}_{\varepsilon > 0}$ (cf. Definition \ref{def:approxsol}) that weakly converge to $u$. If $u^\varepsilon \to u$ strongly in $L^2$, then $u$ is a weak solution to  \eqref{Euler}. Otherwise the energy concentrates on sets.  Despite this, $u$ may be still an Euler solution due to the presence of subtle cancellations. This is the so-called \emph{concentration-cancellation phenomenon}.



In their foundational papers, Lopes Filho--Nussenzveig Lopes--Tadmor
\cite{LNT} and Tadmor \cite{Tadmor} develop $H^{-1}$-stability\footnote{More precisely,  $H^{-1}_{\text{loc}}$-stability.} (cf. Definition \ref{def:h-1}) into a
very powerful unifying framework to study lack of concentrations in approximate solutions. To be more precise, 
these authors obtained the following result.    


\begin{theorem}
[\cite{LNT}]\label{teoh-1} Suppose that $\{u^{\varepsilon}\}_{\varepsilon>0}$
is an $H^{-1}$-stable approximate family of Euler solutions. Then $\{u^{\varepsilon}\}_{\varepsilon>0}$  converges strongly (possibly passing to a subfamily)
to a weak solution of the Euler equation $u$ in $L^{\infty
}([0,T];L_{\emph{loc}}^{2}(\mathbb{R}^{n};\mathbb{R}^{n}))$.
\end{theorem}



The implementation of $H^{-1}$-stability depends on having at one's
disposal sharp criteria to characterize the compact sets of $H^{-1}$ and, in
particular, leverage this knowledge to decide which function spaces, among
those relevant in the description of physical phenomena connected with the
Euler equations, embed compactly into $H^{-1}.$ In this direction, the
$H^{-1}$-criteria, as it applies to rearrangement invariant spaces, was
extensively developed in \cite{LNT}, recovering and extending earlier results from \cite{DiPernaMajda} and  \cite{Lio96}.




%

As shown in \cite{DiPernaMajda}, solutions to $2$D Euler equations when the initial vorticity is supported in a curve play a central role in fluid dynamics.  These solutions are called \emph{vortex sheets} and their regularity can be naturally  measured in terms of the Morrey spaces $M^{p,\alpha}$ (cf. \eqref{aparecio} below). The $H^{-1}$-stability theory for Morrey spaces is also
treated in \cite{LNT}, and relies on the following compactness
result\footnote{The same statement holds, mutatis mutandi, for the Morrey
space of measures \cite[Theorem 4.3]{LNT}.} \footnote{ We use standard notation: If $X$ is a function space on $\R^n$, we let
$X_{c}$ be the subspace of compactly supported functions; and we let $X_{\text{loc}}$ be
the set of functions $f$ such that $f \mathbf{1}_{Q_0}\in X,$ for all cubes 
$Q_0.$ We write%
\[
X_{c}\hookrightarrow H_{\text{loc}}^{-1}(\R^n) \qquad \text{ (resp. }X_{c}%
\overset{compactly}{\hookrightarrow}H_{\text{loc}}^{-1}(\R^n))
\]
if for all $Q_0$
\[
X(Q_0)\hookrightarrow H^{-1} (\R^n) \qquad \text{ (resp. }X(Q_0
)\overset{compactly}{\hookrightarrow}H^{-1}(\R^n)),
\]
where $X(Q_0) = \{f \in X : \text{supp } f \subset Q_0\}$ with $\|f\|_{X(Q_0)} = \|f \mathbf{1}_{Q_0}\|_{X}$.
} due independently to DeVore and Tao (cf. \cite[Theorems 4.2]{LNT}):
\begin{equation}
M^{p,\alpha}_c(\R^n)\overset{compactly}{\hookrightarrow}H^{-1}_{\text{loc}}(\R^n
)\label{comp}%
\end{equation}
provided that one of the following conditions is satisfied:
\begin{enumerate}
\item[(a)] $p>\frac{n}{2}$,
\item[(b)] $p=\frac{n}{2}$ \quad and \quad  $\alpha>1.$
\end{enumerate}
Once in possession of these statements, Theorem \ref{teoh-1} can be applied to establish that for families of approximate
solutions, with uniformly bounded vorticities in $M^{p,\alpha}(\R^n),$ one
can extract convergent subsequences to a solution of the Euler equation
(\ref{Euler}), without concentration. In the special case $n=2, p=1$ and $\alpha > 1$, this result\footnote{The original statement from \cite{DiPernaMajda} involves a certain additional  assumption on weak decay at infinity of vorticities.} was first obtained by DiPerna--Majda \cite[Theorem 3.1]{DiPernaMajda} using tools from elliptic theory. On the other hand, the case $n=3$ and $p = 3/2$ is connected with the work of Giga and Miyakawa \cite{Giga} on well-posedness of $3$D Navier--Stokes equations with initial singular data such as vortex filaments.

At present time, the picture is not
completely understood for all the values of the parameters involved in
(\ref{comp}). Specially in $2$D, it is known that for $p=1$ and $\alpha=1/2,$
(\ref{comp}) does not hold (cf. \cite{DiPernaMajda, M93}). To the best of our knowledge, it
remains an open problem to decide whether (\ref{comp}) with $p=1$ still holds  for $\alpha$ $\in(\frac{1}%
{2},1]$ (the so-called \textquotedblleft gap problem\textquotedblright), leaving open the existence of solutions without concentrations in $M^{1, \alpha}$. 
Similar types of gaps also appear when dealing with higher dimensions.

In an effort to understand the nature of these gaps, and their impact on the
convergence of approximate solutions of the Euler equations, Tadmor
\cite{Tadmor} introduced the finer scale of RMT spaces, $R_{p,q}\log^{\alpha},$ that sharpen (\ref{comp}); cf. Definition \ref{DefinitionRMT}. 

\subsection{Tadmor's refinement of $H^{-1}$-stability\label{sec:tad}}

It is shown in \cite{Tadmor} that RMT spaces
\textquotedblleft interpolate the compactness gap\textquotedblright\ in the
sense that%
\begin{equation}
R_{p,2}\log^{\alpha} (\R^n)_c\overset{compactly}{\hookrightarrow}H_{\text{loc}}%
^{-1}(\R^n)\label{comp1}%
\end{equation}
provided that one of the following conditions is satisfied:
\begin{enumerate}
\item[(a)] $p>\frac{2n}{n+2},$
\item[(b)] $p=\frac{2n}{n+2}$\quad  and (crucially) \quad $\alpha>\frac{1}{2}.$
\end{enumerate}

The $R_{p,2}\log^{\alpha}$ scale is sharp, with respect to the $H^{-1}$-
stability, in the sense that for approximate solutions, with vorticities
uniformly bounded in $R_{\frac{2n}{n+2},2}\log^{\alpha},$ $\alpha>1/2,$
we can extract solutions without concentration, while for $\alpha\in(0,1/2]$
there is a weak limit solution (i.e. a concentration-cancellation effect). On
the other hand, the original gap problem for Morrey spaces $M^{1,\alpha},$ is
apparently not resolved in this fashion, since\footnote{In other words, Tadmor's scale requires a stronger
regularity condition than Morrey regularity on the set of vorticities to
achieve compactness.} (cf. \cite[page 519 and the discussion after 
(3.5)]{Tadmor})
\begin{equation*}
R_{1,2}\log^{\alpha}(\mathbb{R}^{2})\subset M^{1,\alpha}(\mathbb{R}%
^{2}).
\end{equation*}

\subsection{Paving the way to sparseness}\label{SectionPaving}

The presence of logarithms in (\ref{comp1}) (and \eqref{comp}) is very natural and is connected
with some implicit \emph{extrapolation} constructions that are needed since
$R_{1,2}(\mathbb{R}^{2})$ (or more generally,  $R_{\frac{2n}{n+2}%
,2}(\mathbb{R}^{n})$) is not suitable for $H^{-1}$-stability. In
fact, we have (cf. (\ref{agregada}))
\begin{equation}
R_{\frac{2n}{n+2},2}(\mathbb{R}^{n})\nsubseteq H^{-1}(\mathbb{R}%
^{n}).\label{barnard}%
\end{equation}
To overcome this obstacle, in this paper we propose a different methodology
based on \emph{sparseness}.


\subsection{A new framework for $H^{-1}$-stability: Sparse stability}

The main goal of this paper is to reformulate $H^{-1}$-stability applying  the theory of 
sparse spaces, that we recently introduced in \cite{DMComptes}. In a nutshell, we show that the ``defect" of RMT spaces exhibited  by  \eqref{barnard} can be overcome if the geometry of testing cubes in the definition of these spaces is changed. More precisely, let $SR_{\frac{2n}{n+2},2}$ the space that is obtained by replacing  pairwise disjoint cubes in $R_{\frac{2n}{n+2},2}$   by sparse\footnote{Loosely speaking, sparse cubes are not necessarily disjoint but possible overlappings can be controlled in a sharp fashion; cf. Definition \ref{DefnSparse}.} families of cubes (cf.  Definition \ref{DefSRMT} below). Then, somewhat informally, the following surprising formula holds (cf. Theorem \ref{CorSparSob} for the precise statement)
\begin{equation}\label{SNS}
	SR_{\frac{2n}{n+2},2}(\R^n) = H^{-1}(\R^n). 
\end{equation}

%

Armed with formula \eqref{SNS} we provide a sparse structure to $H^{-1},$ which we
exploit to develop new methods to characterize  compact sets 
in $H^{-1}.$ In particular, we introduce \emph{indices of sparseness}, associated to function spaces, that measure the degree of compactness  into
$H^{-1}$. Conversely, given a
\emph{decay} $\Psi$, i.e. a positive decreasing function on $[0, \infty)$ satisfying
\begin{equation}
\lim_{t\rightarrow\infty}\Psi(t)=0,\label{psi1}%
\end{equation}
 we
construct \emph{sparse spaces} $S_{\Psi}$, whose sparse indices have the prescribed decay $\Psi$. This leads to the introduction of the concept of \emph{$\Psi$-sparse stability} for approximate solutions; cf. Definition \ref{DefSpS}. 
As a consequence, we create a refined scale that exhausts the classical $H^{-1}$-stability in the following sense (compare with Theorem \ref{teoh-1}.)


\begin{theorem}
\label{teo:exhaust} Let $\{u^{\varepsilon
}\}_{\varepsilon>0}$ be a family of approximate solutions. The
following are equivalent:
\begin{enumerate}
\item[(i)] $\{u^{\varepsilon}\}_{\varepsilon>0}$ is $H^{-1}$-stable,
\item[(ii)] $\{u^\varepsilon\}_{\varepsilon > 0}$ is sparse stable.
\end{enumerate}
As a consequence, if \emph{(ii)} holds then (possibly passing to a subfamily) $u^\varepsilon \to u$ strongly in $L^2$, where $u$ is a solution of \eqref{Euler}. 
\end{theorem}

We show that sparse stability not only provides a simplified approach to all previously known existence results from  \cite{DiPernaMajda}, \cite{LNT} and \cite{Tadmor} but, more importantly, it leads to the sharpening of the classical results.

We next detour to present in detail the construction of
sparse spaces (including their connection with negative Sobolev spaces; cf. \eqref{SNS}), sparse indices and sparse stability. 



\subsection{Negative Sobolev spaces via sparse RMT spaces\label{sec:functionsbackground}}

The norms of many familiar spaces in analysis are defined in terms of
coverings by disjoint cubes or \textquotedblleft packings\textquotedblright%
\ (e.g. spaces like BMO, John-Nirenberg spaces, Morrey spaces,
Campanato spaces, Brudnyi spaces, Lipschitz spaces, Garsia-Rodemich spaces, ...). In \cite{DMComptes}, we initiated the analysis of \textquotedblleft
sparse versions\textquotedblright\ of classical spaces,  obtained
modifying the requirements on the coverings: We replaced the usual packings of
cubes by the slightly bigger class of \textquotedblleft sparse
coverings\textquotedblright. We briefly recall the definition of sparse family of cubes.

Given 
$Q_{0}$ a fixed cube\footnote{In this paper all cubes are assumed to have sides
parallel to the coordinate axes.}, let $\mathcal{D}(Q_{0})$ be the
dyadic cubes in $Q_{0}$.

\begin{definition}[Sparse cubes]
\label{DefnSparse} A  (countable) family $(Q_{i})_{i\in
I} \subset \mathcal{D(}Q_{0})$ is called \emph{$\eta
$-sparse\footnote{In this paper the parameter $\eta$ will not play a role,
so in what follows we shall let $\eta=1/2$.}}, $\eta\in(0,1)$, if for every
$Q_{i}$ there exists a measurable subset $E_{Q_{i}}$ such that:
\begin{enumerate}
\item[(i)] the sets $E_{Q_{i}}$ are pairwise disjoint,
\item[(ii)] $\eta |Q_{i}|\leq|E_{Q_{i}}|$.
\end{enumerate}
We let $S(Q_{0})$ be the collection of all sparse families of
dyadic cubes in $Q_{0}.$ Analogously, one can introduce $S(\R^n)$, the set formed by all sparse  families of dyadic cubes in $\R^n$. 
\end{definition}

In particular, the sparse spaces $SR_{p,q}\log^{\alpha}(Q_{0})$ are
constructed modifying standard RMT spaces (cf. (\ref{aparece})) by means of replacing families of
packings, \textquotedblleft$(Q_{i})_{i\in I}\in\Pi(Q_{0})",$ by
sparse families, \textquotedblleft$(Q_{i})_{i\in I}\in S(Q_{0})"$.

\begin{definition}[Sparse RMT spaces]\label{DefSRMT}
Let $1 \leq p, q \leq \infty$ and $\alpha \in \R$. The \emph{sparse RMT space} $SR_{p,q}\log^{\alpha}(Q_{0})$ is formed by all those $f \in L^1(Q_0)$ such that
\begin{equation}
\Vert f\Vert_{SR_{p,q}\log^{\alpha}(Q_{0})}=\sup_{(Q_{i})_{i\in I}\in
S(Q_{0})}\left\{\sum_{i\in I}\bigg[\frac{(1-(\log|Q_{i}|)_{-})^{\alpha}%
}{|Q_{i}|^{\frac{1}{p^{\prime}}}}\,\int_{Q_{i}}|f|\bigg]^{q}\right\}^{\frac
{1}{q}}<\infty. \label{SRpq}%
\end{equation}
In particular, we let $SR_{p, q}(Q_0) = SR_{p, q} \log^0(Q_0)$. The corresponding spaces on $\R^n$ are introduced analogously. 
\end{definition}

\begin{remark}\label{RemarkMeasures}
	The definition of $SR_{p, q} \log^\alpha$ is well adapted to work with signed measures $\omega \in BM^+$. Indeed, we simply replace $\int_{Q_i} |f|$ in \eqref{SRpq}  by $\omega(Q_i)$. 
\end{remark}

Since we trivially have $\Pi(Q_{0})\subset S(Q_{0}),$ it follows that
\[
SR_{p,q}\log^{\alpha}(Q_{0})\subset R_{p,q}\log^{\alpha}(Q_{0}).
\]
However, in general, the sparse spaces are different\footnote{Note that $SR_{p, \infty} \log^\alpha (Q_0) = R_{p, \infty} \log^\alpha(Q_0) = M^{p, \alpha}(Q_0)$; cf. Remark \ref{Remark3}.} from their parent spaces. In our
context, the differences manifest themselves through the behavior of the maximal operators $M_{\lambda, Q_0}$ (cf. \eqref{maxlam}). The theory developed in Section
\ref{sec:sparsedommax} will play a crucial role in our analysis.


Suppose that $1\leq p<q<\infty,$ then a special case of Proposition
\ref{SparseDomination2} below shows that there exists 
$\{Q_{i}\}_{i\in I} \in S(Q_0)$ and a constant $c$ depending only on $p$ and $q$ such that
\begin{equation}
M_{n(\frac{1}{p}-\frac{1}{q}),Q_{0}}f(x)\leq c \,  \sum_{i\in I}\bigg(\frac
{1}{|Q_{i}|^{\frac{1}{p^{\prime}}+\frac{1}{q}}}\int_{Q_{i}}%
|f(y)|\,dy\bigg)\mathbf{1}_{Q_{i}}(x).\label{gar1}%
\end{equation}
The process of constructing such
coverings is referred as  \emph{sparse
domination} in the literature; cf. \cite{LN19}, \cite{H21} and the references given therein.  
From \eqref{gar1} and more or less standard arguments, we obtain the following  remarkable
result connecting $SR_{p, q}(\R^n)$ and classical Sobolev spaces $H^{-\lambda}_q(\R^n), \, \lambda \in (0, n)$ (cf. \eqref{rpot0}).

\begin{theorem}\label{CorSparSob}
Let $1\leq p<q<\infty$. If  $f \in L^{1}_{\emph{loc}}
(\mathbb{R}^{n}), \, f \geq0$ a.e.,  then
\[
\Vert f\Vert_{SR_{p,q}(\mathbb{R}^{n})}\approx\Vert f\Vert_{H_{q}^{-n(\frac
{1}{p}-\frac{1}{q})}(\mathbb{R}^{n})}.
\]
In general 
\[
SR_{p, q}(\R^{n}) \hookrightarrow H_{q}^{-n(\frac{1}{p}-\frac{1}{q}%
)}(\mathbb{R}^{n}).
\]
In particular, the canonical choice of parameters  $p=\frac{2n}{n+2}, \, n \geq 2,$ and $q=2$ gives\footnote{The role of
sparseness is crucial here. In particular, for the classical space
$R_{\frac{2n}{n+2},2},$ this approach fails dramatically (cf. (\ref{barnard}%
)).} 
\begin{equation}\label{SobIde}
	\Vert f\Vert_{SR_{\frac{2n}{n+2},2}(\mathbb{R}^{n})}\approx\Vert f\Vert_{H^{-1}(\mathbb{R}^{n})} \qquad \text{if} \qquad f \geq 0.
\end{equation}
Consequently
\begin{equation}\label{SobIde2}
	SR_{\frac{2n}{n+2},2}(\mathbb{R}^{n}) \hookrightarrow H^{-1}(\R^n). 
\end{equation}
\end{theorem}

\subsection{Sparse indices\label{sec:despues}}

Since we deal with local problems, most of the analysis will be carried out on
cubes, but similar constructions are also possible in $\R^n$. Let $Q_{0} \subset \R^n, \, n \geq 2,$ be a fixed cube and let $\mathcal{Q} \in S(Q_{0}).$ For
$k\in\mathbb{N}_{0}:=\mathbb{N}\cup\{0\}$, let $\mathbb{D}_{\leq k;Q_{0}}:=\{Q:Q \in \mathcal{D}(Q_{0})$ with sidelength $\ell(Q)\leq 2^{-k}\ell(Q_{0})\}$,
$\mathbb{D}_{\leq k; Q_0}(\mathcal{Q)}:=\mathbb{D}_{\leq k;Q_{0}}\cap\mathcal{Q}$. When there is no danger of confusion,  we use the simplified notation $\mathbb{D}_{\leq k}$  and $\mathbb{D}_{\leq k}(\mathcal{Q)}$.

\begin{definition}[Sparse indices] 
\begin{enumerate}
\item[(i)] 
The \emph{sparse indices of $f \in L^{1}(Q_{0})$} are defined by\footnote{Sparse indices may
depend on the given cube $Q_{0}.$ However, since this dependance will not play
a role in our arguments, it will be safely omitted in the corresponding notation.}
\begin{equation}\label{spinf}
s_{N}(f)= \sup_{\mathcal{Q} \in S(Q_0)} \left[\sum_{Q \in \mathbb{D}_{\leq N-1} (\mathcal{Q})}\bigg(|Q|^{\frac{1}%
{n}-\frac{1}{2}}\int_{Q}|f|\bigg)^{2}\right]^{\frac{1}{2}}, \qquad N \in \mathbb{N}.
\end{equation}
\item[(ii)] Let $X$ be a function space, $X\subset L_{\text{loc}}^{1}(\mathbb{R}^{n})$. The \emph{sparse indices of $X(Q_0)$} are defined by
\begin{equation}
 s_{N}(X) =\sup_{\|f\|_{X(Q_0)} \leq 1} s_N(f). \label{delacompa1}%
\end{equation}
\end{enumerate}
\end{definition}

\begin{remark}
	The  definitions above can be extended in a natural way  to the setting of measures with distinguished sign.
\end{remark}

Compactness of embeddings into $H^{-1}$ can be characterized in terms of sparse indices. Specifically, we have the following result.

\begin{theorem}\label{teo:compacto}
Let $X$ be a function space\footnote{As usual, $ L^1_{\text{loc}, +}(\R^n) = \{f \in  L^1_{\text{loc}}(\R^n) : f \geq 0 \text{ a.e.}\}$. The additional assumption $f \geq 0$ is not restrictive  since $s_N(f) = s_N(|f|)$.} $X\subset L_{\emph{loc}, +}%
^{1}(\mathbb{R}^{n}), \, n \geq 2,$ (or more generally, $X \subset BM^+_c$). Then\footnote{Note that $s_1(X) = \sup_{N \in \mathbb{N}} s_N(X)$.}
\begin{enumerate}
\item[(i)] $s_1(X)<\infty\iff X_{c}\hookrightarrow H_{\emph{loc}}^{-1}(\R^n).$
\item[(ii)] $
\lim_{N\rightarrow\infty}s_{N}(X)=0 \iff X_{c}\overset{compactly}{\hookrightarrow}H_{\emph{loc}}^{-1}(\R^n).$
\end{enumerate}
\end{theorem}

The proof of this result is given in Section \ref{Section4.1}. 

Sparse indices provide a very satisfactory criteria for $H^{-1}$-stability (cf. Theorem \ref{teoh-1}). 

\begin{corollary}\label{Corollary1}
	Let $X$ be a function space $X\subset L_{\emph{loc}, +}%
^{1}(\mathbb{R}^{n}), \, n \geq 2,$ (or more generally, $X \subset BM^+_c$) such that
\begin{equation}\label{LimitZero}
	\lim_{N\rightarrow\infty}s_{N}(X)=0. 
\end{equation}
 Suppose that $\{u^{\varepsilon}\}_{\varepsilon>0}$ is an approximate family of Euler solutions with related set of vorticities $\{\omega^\varepsilon\}_{\varepsilon > 0}$ uniformly bounded in $C((0, T); X)$. Then (passing to a subfamily if necessary) $u^{\varepsilon} \to u$ strongly in $L^{\infty
}([0,T];L_{\emph{loc}}^{2}(\mathbb{R}^{n};\mathbb{R}^{n})),$  where  $u$ is a solution to \eqref{Euler}.
\end{corollary}

The  new indices pose a challenge: Can we compute them? In  Section \ref{Section4.2} we provide explicit calculation of sparse indices for classical spaces like Lebesgue, Morrey, and RMT spaces. The results are in Table \ref{Tabla1} below. These computations combined with Theorem \ref{teo:compacto} give a unified
proof of (\ref{comp}) and (\ref{comp1}).

\begin{table}[h]
\centering
\begin{tabular}
[c]{|c|c|}\hline
$X$ & Upper estimate for $s_N(X)$\\\hline
& \\
$L^p, \quad p > \frac{2 n}{n+2}$ & $2^{-N (\frac{2+n}{2n}-\frac{1}{\min\{2,p\}}) n}$\\
& \\\hline
& \\
$M^{p, \alpha}, \quad p > \frac{n}{2}, \quad \alpha \in \R$ &
$2^{-N(\frac{2}{n} -\frac{1}{p}) \frac{n}{2}} N^{-\frac{\alpha}{2}}$\\
& \\\hline
& \\
$M^{\frac{n}{2}, \alpha}, \quad \alpha > 1$ & $N^{\frac{1-\alpha}{2}}$\\
& \\\hline
& \\
$R_{p, 2} \log^\alpha, \quad p > \frac{2 n}{n+2}, \quad \alpha \in \R$ & $2^{-N (\frac{2+n}{2n}-\frac{1}{p}) n} N^{-\alpha}$ \\
& \\\hline
& \\
$R_{\frac{2 n}{n+2}, 2} \log^\alpha, \quad \alpha > \frac{1}{2}$ & $N^{ \frac{1}{2} -\alpha}$ \\
&  \\\hline
\end{tabular}
\caption{Sparse indices for some classical function spaces}%
\label{Tabla1}%
\end{table}

 Furthermore,
understanding the rates of decay of the sparse indices will allow us to measure the degree of $H^{-1}$-compactness, and allow us to
extend the known results as we now explain.

\subsection{Sparse spaces\label{sec:preambledecay}}

So far, given a function space $X$, we analysed the decay of its sparse indices $s_N(X)$ (cf. \eqref{delacompa1}) in order to guarantee $H^{-1}$-stability, cf. Corollary \ref{Corollary1}. However, note that the definition of sparse indices $s_N(f)$ (cf. \eqref{spinf}) is independent of any space $X$. This simple observation leads to the following question: Can we use ``reverse engineering" to create new function spaces whose sparse indices have prescribed decay? A natural construction associated with this idea can
be described as follows.




\begin{definition}[Sparse spaces]\label{DefSpS}
Let $\Psi$ be a decay  (cf. \eqref{psi1}). The \emph{sparse space} $S_\Psi(Q_0)$ is formed by all $f \in L^1(Q_0)$ such that
\begin{equation}
\left\Vert f\right\Vert
_{S_{\Psi}(Q_{0})}=\sup_{N\in\mathbb{N}} \frac{s_N(f)}{\Psi(N)} <\infty.\label{psi2}%
\end{equation}
The counterparts on $\R^n$ and  for positive measures can be introduced analogously.
\end{definition}

Note, parentetically, the superficial similarity with the constructions of
Yudovich spaces and Extrapolation spaces in \cite{DMEuler}. As we shall soon
see this connection goes deeper and, moreover, some concrete calculations can
be effected which lead to the introduction of new Euler relevant function spaces.

From \eqref{psi2}, we clearly have%
\[
s_{N}(S_{\Psi}(Q_0)) \leq \Psi(N)
\]
therefore, by Theorem \ref{teo:compacto},
\begin{equation}\label{116}
S_{\Psi} (\R^n)_{c}\overset{compactly}{\hookrightarrow}H_{\text{loc}}^{-1}(\R^n).
\end{equation}
In particular, using sparse embeddings we can formulate a new $H^{-1}$-criteria.

\begin{theorem}[$H^{-1}$-stability via sparse embeddings]
\label{CorFinal} Suppose that
\begin{equation}\label{117}
	X_c \hookrightarrow S_\Psi(\R^n)_c
	\end{equation}
 holds for some decay $\Psi$, then $X$ is $H^{-1}$-stable,  in the sense that
\[
X_{c}\overset{compactly}{\hookrightarrow}H_{\emph{loc}}^{-1}(\mathbb{R}^{n}).
\]
In particular, if $\{u^{\varepsilon}\}_{\varepsilon > 0}$ is a family of approximate solutions of
the Euler equations, such that the related set of vorticities $\{\omega
^{\varepsilon}\}_{\varepsilon > 0}$ is uniformly bounded in $X$. Then there exists a subfamily
of $\{u^{\varepsilon}\}_{\varepsilon > 0}$ which converges strongly to a weak Euler solution in
$L^{\infty}([0,T];L_{\emph{loc}}^{2}(\mathbb{R}^{n}))$.
\end{theorem}

The proof of this result is an immediate consequence of \eqref{116} and Theorem \ref{teoh-1}. 

Next we go a step further  and show that assumption \eqref{117} in Theorem \ref{CorFinal} is in fact necessary to establish $H^{-1}$-stability.  To do this, we need to introduce the natural generalization\footnote{Recall that $H^{-1}$-stability does not involve any function space $X$, but only approximate solutions; cf. Definition \ref{def:h-1}.} (say, function space-free) of \eqref{117} to approximate solutions: \emph{sparse stability}. As already anticipated in Theorem \ref{teo:exhaust} (cf. Section \ref{sec:construcciones} for its proof),  this new concept provides us with  a remarkable characterization of $H^{-1}$-stability in terms of sparseness.

\begin{definition}[Sparse stability]
We
say that a family $\{u^{\varepsilon}\}_{\varepsilon>0}$ of approximate
solutions of the Euler equation is \emph{sparse stable} if there exists a decay $\Psi$ such that the corresponding
set of vorticities $\{\omega^{\varepsilon}\}_{\varepsilon > 0}$ is uniformly bounded
in\footnote{Let $\mathbb{A}^{n}$ be the set of all anti-symmetric matrices of
order $n$ with real entries. In what follows, we will use the simplified notation $S_\Psi(\R^n)$ rather than $S_\Psi(\mathbb{R}
^{n};\mathbb{A}^{n})$.} $C(0,T;S_\Psi(\mathbb{R}%
^{n};\mathbb{A}^{n}))$. In particular,  $\{u^\varepsilon\}_{\varepsilon > 0}$ satisfies the \emph{admissible sparse stability} property if  $\Psi$ is an admissible\footnote{cf. Definition \ref{DefAdm}(i).} decay. 
\end{definition}

%

Applying sparse stability, it is thus possible, at least
theoretically, to improve all the known $H^{-1}$-stability results of
\cite{DiPernaMajda}, \cite{LNT}, \cite{Tadmor} (cf. Sections \ref{Section1Intro} and \ref{sec:tad}). However, to make the implied
extensions meaningful, we need to exhibit concrete instantiations.  In fact, we obtain significant improvements on the
classical results  and we show that our constructions
unexpectedly connect with the theory of Yudovich spaces \cite{Y95}, Vishik
spaces \cite{V99} and more specifically with the extrapolation spaces of
\cite{JM91} and \cite{DMEuler}.

\subsection{New extrapolation spaces guaranteeing strong convergence to Euler solutions}
As already mentioned in Section \ref{SectionPaving}, extrapolation constructions seem to be implicit in \eqref{comp}- (\ref{comp1}).  In Sections \ref{SectionDMaj} and \ref{Section4} we confirm this belief and show how the extrapolation theory of Jawerth--Milman \cite{JM91} (more precisely,  the updated account given recently  in  \cite{DMEuler}) can be successfully applied to construct concrete
examples of function spaces with prescribed sparse decay. In particular, these new spaces 
strictly contain the limiting spaces involved in (\ref{comp}) and
(\ref{comp1}), but are still $H^{-1}$-stable. As a consequence, we are able to extend the existence results for vortex sheets of \cite{DiPernaMajda} and \cite{Tadmor} to larger classes of vorticities.

\subsubsection{Sharpening Morrey regularity of DiPerna--Majda}\label{SecDM}

We introduce the distributional space $V_\Psi(\R^n)$ given by (cf. Definition \ref{DefV})
$$
\sum_{j=N}^{\infty}2^{-2j}\Vert\Delta_{j}f\Vert_{L^{\infty}(\mathbb{R}^{n}%
)}\lesssim\Psi(N)^{2}, \qquad N \in \mathbb{N}_0. 
$$
These spaces may be considered as ``dual" counterparts of  classical Vishik spaces proposed in \cite{V99} in connection with uniqueness issues for Euler flows; cf. Remark \ref{Remark1} for further explanations.  Applying the set of techniques explained in previous sections, we show sparse stability\footnote{In fact, sparse numbers of $V_\Psi$ behave like $\Psi$.} and non-concentration phenomenon in $V_\Psi$; cf. Theorem \ref{ThmVPsi}. A crucial point in our arguments is that $V_\Psi$ can be characterized as an extrapolation space of classical Besov spaces (cf. Theorem \ref{PropInterpol}). In particular, for the special decay $\Psi(t) = t^{\frac{1-\alpha}{2}}, \, \alpha > 1$, we have  (cf. Theorem \ref{Thm11})
$$M^{\frac{n}{2}, \alpha}(\R^n) \hookrightarrow V_\Psi(\R^n).$$
Furthermore, this result is sharp\footnote{However, it is unclear for us whether the Morrey gap problem can be solved using these techniques.}, i.e.,  we give a constructive method to produce functions  in $V_\Psi(\R^n)$ but not in $M^{\frac{n}{2}, \alpha}(\R^n)$. As a byproduct, we get a non-trivial improvement of \eqref{comp}.

\subsubsection{Sharpening Tadmor regularity} The results stated in Section \ref{SecDM} for Morrey spaces admit counterparts for RMT spaces. In this setting, the role of $V_\Psi$ is played by the new space  $T_\Psi(\R^n)$ (cf. Definition \ref{DefTPsi}), which admits the following nice characterization in terms of Fourier integrals: 
$$
	\int_{|\xi| > 2^N} (1+|\xi|^{2})^{-1} |\widehat{f}(\xi)|^2 \, d \xi \lesssim \Psi(N)^2, \qquad N \in \mathbb{N}_0.
$$
Then we establish sparse stability and non-concentration phenomenon in $T_\Psi$; cf. Theorem \ref{ThmTPsi}.  In particular, for the special decay $\Psi(t) = t^{\frac{1}{2}-\alpha}, \, \alpha > \frac{1}{2}$, we have (cf. Theorem \ref{Theorem4.8})
$$
	R_{\frac{2 n}{n+2}, 2} \log^\alpha(\R^n) \hookrightarrow T_\Psi(\R^n). 
$$
Again, this result is sharp. As a consequence, we improve Tadmor's embedding \eqref{comp1}.

%

\subsection{Energy conservation for physically realizable solutions via sparse stability}

In Section \ref{sec:physically} we show that our methods are sufficiently
robust to provide criteria for the preservation of energy by \emph{physically
realizable solutions}\footnote{Roughly speaking, physically realizable solutions are weak solutions of Euler equations that can be obtained as weak limits of vanishing viscosity; cf. Definition \ref{DefVV}.} of $2$D Euler equations  on the two-dimensional torus $\T^2$. Indeed, extending $L^p (\T^2)$-results\footnote{For general $L^p$ solutions with $p \geq 3/2$, conservation of energy can be derived from the well-known Besov-type criterion of Cheskidov, Constantin, Friedlander and Shvydkoy \cite{Ches2}; cf. also \cite[Theorem 1]{Ches} for an alternative and streamlined proof.}, $p > 1$, 
obtained by Cheskidov, Lopes-Filho, Nussenzveig-Lopes and Shvydkoy \cite{Ches} (cf. also \cite{Crippa} for the  case on the whole plane $\R^2$),  we show that our framework can be used
to provide conditions for physically realizable solutions  to
conserve energy. 

\begin{theorem}\label{ThmCon}
Let $u$ be a physically realizable weak solution of the $2$\emph{D} Euler equations with a physical realization $\{u^{\varepsilon}\}_{\varepsilon>0}$ satisfying admissible sparse stability. Then $u$ is conservative, i.e., $\|u(t)\|_{L^2(\mathbb{T}^2)} = \|u_0\|_{L^2(\mathbb{T}^2)}$. 
\end{theorem}

Very recently, Lanthaler, Mishra and Par\'es-Pulido \cite{Lan} proposed an interesting approach to energy conservation based on the so-called structure functions (i.e., the $L^2$-modulus of smoothness) of $\{u^\varepsilon\}_{\varepsilon > 0}$. On the other hand, Theorem \ref{ThmCon}    relies on the sparse indices of $\{\omega^\varepsilon\}_{\varepsilon > 0}$. Switching from $u^\varepsilon$ to $\omega^\varepsilon$ has important  advantages from the point of view of applications, as  it is illustrated in the following result. 

\begin{corollary}\label{CorCons}
		Let $X$ be a function space $X\subset L
^{1}(\mathbb{T}^{2})$ (or more generally, $X \subset BM^+$) with sparse indices $s_N(X)$ satisfying \eqref{LimitZero} and the admissibility condition (cf. Definition \ref{DefAdm}). Let $\{u^\varepsilon\}_{\varepsilon > 0}$ be a physical realization of the Euler solution $u$. If $\{\omega^\varepsilon\}_{\varepsilon > 0}$ is uniformly bounded in $X$, then $u$ is conservative. 
\end{corollary}

This result follows immediately from Theorem \ref{ThmCon} since $X \hookrightarrow S_\Psi(\T^2)$, where $\Psi(N) = s_N(X)$. In particular, Corollary \ref{CorCons} can be applied to all classical function spaces  exhibited in Table \ref{Tabla1}, as well as the new spaces $X = V_\Psi$ and $X = T_\Psi$. 

\subsection{Brief interlude: some references\label{sec:approximate}}
Existence and uniqueness of weak solutions for the $2$D Euler equations are well
established. In particular, we mention the concentration-cancellation result by Delort \cite{Delort}
(resp. Vecchi and Wu \cite{Vecc}) proving existence of weak solutions for
initial vorticities in $BM_{c}^{+}\cap H^{-1}$ (resp. in $L_{c}^{1}\cap
H^{-1});$ existence and uniqueness results of weak solutions for initial
bounded vorticities (resp. vorticities in Yudovich spaces) were established by
Yudovich \cite{Y63} (resp. \cite{Y95}), and the corresponding results
for vorticities in BMO obtained by Vishik \cite{V99}.
Still, important questions remain open. Notably, uniqueness of
weak solutions for vorticities in $L^{p}, \, p>1,$ remains an open problem,
although there are now indications, after the recent work of Vishik \cite{V2018}  (cf. also \cite{Adele}), that the answer is probably negative.

We close this introduction with the belief that,
given the central role of negative Sobolev spaces in PDEs, our methods could
find applications elsewhere.

\vspace{4mm}
\emph{Notation.} Given two normed spaces $X$ and $Y$, the symbol $X \hookrightarrow Y$ means that the identity operator  from $X$ into $Y$ is continuous. Given two positive quantities $A$ and $B$, we write $A \lesssim B$ if there is a constant $C > 0$ such that $A \leq C B$.  We also use $A \approx B$ if $A \lesssim B$ and $A \gtrsim B$. For $a\in\mathbb{R}$,  
$a_{-}=\min\{a,0\}$, $p^{\prime}$ denotes the dual
exponent of $p$ given by $\frac{1}{p}+\frac{1}{p^{\prime}}=1$ and $|E|$ is the
(Lebesgue) measure of a measurable set $E$.

\section{Background}

\subsection{Approximate solutions and $H^{-1}$-stability} For convenience of the reader, we  recall the well-known concepts of approximate solutions of Euler equations and their $H^{-1}$-stability, as introduced in \cite{DiPernaMajda} and \cite{LNT}, respectively.  

\begin{definition}
\label{def:approxsol}A family of velocity vector
fields $\{u^{\varepsilon}(\cdot,t)\}_{\varepsilon>0}$, $t\in\lbrack0,T],$ defines
an \emph{approximate solution} of (\ref{Euler}) if for some $L>1,$ it is
uniformly bounded in
$L^{\infty}([0,T];L_{c}^{2}(\mathbb{R}^{n};\mathbb{R}^{n}))\cap\text{Lip}%
((0,T);H_{\text{loc}}^{-L}(\mathbb{R}^{n};\mathbb{R}^{n})),$\footnote{The
uniform bound in $\text{Lip}((0,T); H_{\text{loc}}^{-L}(\mathbb{R}%
^{n};\mathbb{R}^{n}))$ is a technical assumption in order to guarantee that
initial vector fields $u^{\varepsilon}(\cdot,0)$ are well-defined. In
practice, this follows easily from the uniform energy bound $L^{\infty
}([0,T];L_{c}^{2}(\mathbb{R}^{n};\mathbb{R}^{n}))$, cf. \cite{DiPernaMajda}.}
with $\operatorname{div}u^{\varepsilon}=0$ (in the distributional sense), and
is weak consistent with \eqref{Euler}, in the sense that\footnote{The weak
formulation related to domains with boundary is analogous to that of
$\mathbb{R}^{n}$, taking into account the additional boundary condition
$u^{\varepsilon}\cdot n=0$ (in the trace sense).}
\begin{equation*}
\int_{0}^{T}\int_{\mathbb{R}^{n}}\varphi_{t}\cdot u^{\varepsilon}%
+(D\varphi\,u^{\varepsilon})\cdot u^{\varepsilon}\,dx\,dt+\int_{\mathbb{R}%
^{n}}\varphi(x,0)\cdot u^{\varepsilon}(x,0)\,dx\rightarrow0\qquad
\text{as}\quad\varepsilon\rightarrow0 
\end{equation*}
for every test field $\varphi\in C_{c}^{\infty}([0,T)\times\mathbb{R}%
^{n};\mathbb{R}^{n})$ with $\operatorname{div}\varphi=0$. Here, $D\varphi$ is
the Jacobian matrix of $\varphi$.
\end{definition}

\begin{remark}
If the family is constant, $u^{\varepsilon}\equiv u$ for all
$\varepsilon>0,$ then $u$ is in fact a classical \emph{weak solution} to \eqref{Euler}.
\end{remark}

\begin{remark}
There are standard methodologies to construct approximation solution families,
e.g. through mollification of initial data, Navier-Stokes approximate
solutions (also known as vanishing viscosity method), vortex blob approximations, discrete methods, ...
\end{remark}

\begin{definition}[$H^{-1}$-stability]\label{def:h-1} We
say that a family $\{u^{\varepsilon}\}_{\varepsilon>0}$ of approximate
solutions of the Euler equation is \emph{$H^{-1}$-stable} if the corresponding
set of vorticities $\{\omega^{\varepsilon}= \operatorname{curl}u^{\varepsilon
}\}_{\varepsilon > 0}$ (i.e., $\omega_{i,j}^{\varepsilon}=(u_{i}^{\varepsilon})_{x_{j}}%
-(u_{j}^{\varepsilon})_{x_{i}}$ for $i,j=1,\ldots,n$) is precompact
in $C((0,T);H_{\text{loc}}^{-1}(\mathbb{R}%
^{n};\mathbb{A}^{n})).$
\end{definition}

\subsection{Riesz-Morrey-Tadmor spaces}
 Let $\Pi(Q_{0})$ be the set of
families of packings\footnote{Families of pairwise disjoint cubes.}
$(Q_{i})_{i\in I},$ with $Q_{i}\in\mathcal{D}(Q_0)$.

\begin{definition}[\cite{Tadmor}]\label{DefinitionRMT}
The   \emph{Riesz-Morrey-Tadmor spaces}\footnote{Our
notation differs from \cite{Tadmor}, where the space $R_{p,q}\log^{\alpha
}(Q_{0})$ is instead denoted by $V^{pq}(\log V)^{\alpha}(Q_{0})$ (or simply by
$V^{pq,\alpha}(Q_{0})$). The reason behind this change of notation comes from
the Riesz theorem (cf. \eqref{rieszclasico}).}
(RMT spaces, in short) $R_{p,q}\log^{\alpha}(Q_{0}),$ $1\leq p,q\leq\infty
,\alpha\in\mathbb{R}$, are defined through the condition
\begin{equation}
\Vert f\Vert_{R_{p,q}\log^{\alpha}(Q_{0})}=\sup_{(Q_{i})_{i\in I}\in\Pi
(Q_{0})}\left\{\sum_{i\in I}\bigg[\frac{(1-(\log|Q_{i}|)_{-})^{\alpha}}%
{|Q_{i}|^{\frac{1}{p^{\prime}}}}\,\int_{Q_{i}}|f|\bigg]^{q}\right\}^{\frac
{1}{q}}<\infty.\label{aparece}%
\end{equation}
The corresponding spaces on $\R^n$ are defined analogously. 
\end{definition}

\begin{remark}\label{Remark3}
	In particular, \emph{Morrey spaces }are part of this
scale. Let $1\leq p\leq\infty,\alpha\in\mathbb{R},$ the Morrey space
$M^{p,\alpha}(Q_{0})$ is defined by
\begin{equation}
\Vert f\Vert_{M^{p,\alpha}(Q_0)}=\sup_{Q\in\mathcal{D(}Q_{0})}\frac{(1-(\log
|Q|)_{-})^{\alpha}}{|Q|^{\frac{1}{p^{\prime}}}}\,\int_{Q}|f|<\infty
.\label{aparecio}%
\end{equation}
Consequently, $M^{p,\alpha}(Q_{0})=R_{p,\infty}\log^{\alpha}(Q_{0}).$
\end{remark}

\begin{remark}
A similar comment to Remark \ref{RemarkMeasures} also applies to  $R_{p, q} \log^\alpha$ and $M^{p, \alpha}$.  
\end{remark}

\section{Characterization of sparse RMT spaces via maximal
operators\label{sec:sparsedommax}}

Let
$0\leq\lambda<n$ and $\alpha \in \R$. For $f\in L^{1}(Q_{0}),$  consider the maximal operator
\begin{equation}
M_{\lambda,\alpha,Q_{0}}f(x)=\sup_{\substack{Q\in\mathcal{D}(Q_{0})\\x\in
Q}} |Q|^{\frac{\lambda}{n}-1}  (1-(\log|Q|)_{-})^{\alpha}\int_{Q}|f(y)| \,dy,\qquad x\in Q_{0}%
.\label{maxlam}%
\end{equation}
In the absence of logarithmic weight (i.e., $\alpha=0$), we simply  write
$M_{\lambda,Q_{0}}$. In addition, if $\lambda=0$ then one recovers the classical
(dyadic) maximal operator $M_{Q_{0}}$.

In this section we show that, under some natural conditions, the sparse
$SR_{p,q}\log^{\alpha}$ spaces (cf. Definition \ref{DefSRMT}) admit simple characterizations in
terms of maximal operators \eqref{maxlam}. This is in sharp contrast with the parent spaces
$R_{p,q}\log^{\alpha}$. 

\begin{theorem}
\label{ThmSparseRieszMaximal} Suppose that $p,q,\alpha$ satisfy
\begin{equation}
1\leq p\leq q<\infty\qquad\text{and}\qquad\alpha\in\mathbb{R}\quad(\alpha
\leq0\text{ if }p=q). \label{AssumptionMon}%
\end{equation}
Then 
\[
SR_{p,q}\log^{\alpha}(Q_{0})=M_{n(\frac{1}{p}-\frac{1}{q}),\alpha,Q_{0}}%
L^{q}(Q_{0}).
\]
More precisely,
\begin{equation*}
\Vert f\Vert_{SR_{p,q}\log^{\alpha}(Q_{0})}\approx\Vert M_{n(\frac{1}{p}%
-\frac{1}{q}),\alpha,Q_{0}}f\Vert_{L^{q}(Q_{0})},
\end{equation*}
where the hidden constants of equivalence are independent of $f$ and $Q_{0}$.
\end{theorem}

\begin{remark}
When $p=q,$ the restriction $\alpha\leq0$ is necessary to avoid trivial cases.
To be more precise, if $\alpha>0$ and $p=q$, then
\[
\Vert M_{n(\frac{1}{p}-\frac{1}{q}),\alpha,Q_{0}}f\Vert_{L^{q}(Q_{0})}%
<\infty\implies f=0\quad\text{a.e. on}\quad Q_{0}.
\]
Indeed, since $M_{n(\frac{1}{p}-\frac{1}{q}),\alpha,Q_{0}}f(x)=M_{0,\alpha
,Q_{0}}f(x)<\infty$ \ a.e. $x\in Q_{0}$, we have
\begin{equation}
\frac{1}{|Q|}\int_{Q}|f|\leq(1-(\log|Q|)_{-})^{-\alpha}M_{0,\alpha,Q_{0}%
}f(x)\label{delotro}%
\end{equation}
for every $Q\in\mathcal{D}(Q_{0}),$ with $x\in Q,$ and $|Q|$ sufficiently
small. Taking limits on both sides of (\ref{delotro}) as $|Q|\rightarrow0$,
and applying the Lebesgue differentiation theorem, we conclude that $f(x)=0$
for every Lebesgue point $x$.
\end{remark}

Sparse domination principles underlie the characterizations of sparse spaces
via maximal functions.

\begin{proposition}
\label{SparseDomination2} Suppose that $p,q$ and $\alpha$ satisfy
\eqref{AssumptionMon}, and let $f\in L^{1}(Q_{0})$. Then there exists a family
$(Q_{i})_{i\in I}\in S(Q_{0})$ (depending on $f$ and the parameters $p,q$ and
$\alpha$) such that
\begin{equation}
M_{n(\frac{1}{p}-\frac{1}{q}),\alpha,Q_{0}}f(x)\leq2 \max\bigg\{1,e^{\frac
{1}{p}-\frac{1}{q}-\alpha}\bigg(\frac{pq\alpha}{q-p}\bigg)^{\alpha
}\bigg\}\,\sum_{i\in I}\bigg(\frac{\big(1-(\log|Q_{i}|)_{-}\big)^{\alpha}%
}{|Q_{i}|^{\frac{1}{p^{\prime}}+\frac{1}{q}}}\int_{Q_{i}}%
|f(y)|\,dy\bigg)\mathbf{1}_{Q_{i}}(x)\label{SparseDominationEq2}%
\end{equation}
for almost every $x\in Q_{0}$.
\end{proposition}

\begin{remark}
As usual, if $p=q$ the constant $\max\{1,e^{\frac{1}{p}-\frac{1}{q}-\alpha
}(\frac{pq\alpha}{q-p})^{\alpha}\}$ in \eqref{SparseDominationEq2} should be
interpreted to be equal to $1$.
\end{remark}

\begin{proof}
[Proof of Proposition \ref{SparseDomination2}]The desired decomposition will
be obtained by a standard process of exhaustion, whereby for each cube of the
starting decomposition we shall apply the process again and again. We set up
the selection process by letting $\mathcal{Q}_{f}=\mathcal{Q}_{f,p,q,\alpha}$,
be the collection of  $Q\in \mathcal{D}( Q_{0})$ such that the
following condition is satisfied%
\begin{equation}
\frac{\big(1-(\log|Q|)_{-}\big)^{\alpha}}{|Q|^{\frac{1}{p^{\prime}}+\frac
{1}{q}}}\int_{Q}|f(y)|\,dy\geq2\max\bigg\{1,e^{\lambda-\alpha}\bigg(\frac
{\alpha}{\lambda}\bigg)^{\alpha}\bigg\}\,\frac{\big(1-(\log|Q_{0}%
|)_{-}\big)^{\alpha}}{|Q_{0}|^{\frac{1}{p^{\prime}}+\frac{1}{q}}}\int_{Q_{0}%
}|f(y)|\,dy,\label{ProofSparseDomination1}%
\end{equation}
where $\lambda:=\frac{1}{p}-\frac{1}{q}$. If the collection $\mathcal{Q}_{f}$
is empty then we let $E_{Q_{0}}=\{Q_{0}\},$ and we readily verify that
(\ref{SparseDominationEq2}) holds. Otherwise we continue the process selecting
$(Q_{i})_{i\in\mathbb{N}},$ the family of maximal dyadic cubes in
$\mathcal{Q}_{f}$. By construction, the selected family $(Q_{i})_{i\in
\mathbb{N}}$ is pairwise disjoint and, therefore, for almost every $x\in
Q_{0}$,%
\begin{align*}
M_{n(\frac{1}{p}-\frac{1}{q}),\alpha,Q_{0}}f(x) &  =M_{n(\frac{1}{p}-\frac
{1}{q}),\alpha,Q_{0}}f(x)\,\mathbf{1}_{Q_{0}\backslash\cup_{i=1}^{\infty}%
Q_{i}}(x) +\sum_{i=1}^{\infty}M_{n(\frac{1}{p}-\frac{1}{q}),\alpha,Q_{0}%
}f(x)\,\mathbf{1}_{Q_{i}}(x)  \nonumber\\
&  =:(A)+(B).
\end{align*}
Next we estimate each of the terms $(A)$ and $(B)$ separately.

\vspace{2mm} \textbf{Estimate (A)}: We claim that, for $x\not \in \cup
_{i=1}^{\infty}Q_{i}$,
\begin{equation}
M_{n(\frac{1}{p}-\frac{1}{q}),\alpha,Q_{0}}f(x) \leq2\, \max\bigg\{1,
e^{\lambda- \alpha} \bigg(\frac{\alpha}{\lambda} \bigg)^{\alpha}\bigg\} \,
\frac{\big( 1 - (\log|Q_{0}|)_{-} \big)^{\alpha}}{|Q_{0}|^{\frac{1}{p^{\prime
}}+\frac{1}{q}}}\int_{Q_{0}}|f(y)|\,dy. \label{Aux1}%
\end{equation}
Indeed, suppose, to the contrary, that for some $x\not \in \cup_{i=1}^{\infty
}Q_{i},$ \eqref{Aux1} does not hold. Then, by the definition of $M_{n(\frac
{1}{p}-\frac{1}{q}),\alpha,Q_{0}}$ (cf. \eqref{maxlam}), there exists a dyadic
cube $Q\subset Q_{0},$ such that $x\in Q,$ and $Q\in\mathcal{Q}_{f}.$
Consequently, there exists a maximal cube $Q_{i}$ such that $Q\subset Q_{i},$
but this leads to a contradiction since $x\notin Q_{i}.$ Therefore, for
$x\not \in \cup_{i=1}^{\infty}Q_{i}$, we have
\begin{equation*}
(A)\leq2\, \max\bigg\{1, e^{\lambda- \alpha} \bigg(\frac{\alpha}{\lambda}
\bigg)^{\alpha}\bigg\} \, \frac{\big( 1 - (\log|Q_{0}|)_{-} \big)^{\alpha}%
}{|Q_{0}|^{\frac{1}{p^{\prime}}+\frac{1}{q}}}\int_{Q_{0}}|f(y)|\,dy.
\end{equation*}
Moreover, from $(Q_{i})_{i\in\mathbb{N}}$ $\subset\mathcal{Q}_{f}$ (cf.
\eqref{ProofSparseDomination1}) we see that%
\begin{equation}
\label{hshash}\varphi(|Q_{i}|)|Q_{i}|^{-1}\int_{Q_{i}}|f(y)|\,dy\geq2\,
\max\bigg\{1, e^{\lambda- \alpha} \bigg(\frac{\alpha}{\lambda} \bigg)^{\alpha
}\bigg\} \, \varphi(|Q_{0}|)|Q_{0}|^{-1}\int_{Q_{0}}|f(y)|\,dy,
\end{equation}
where
\[
\varphi(t) := t^{\frac{1}{p} -\frac{1}{q}} \, \big(1-(\log t)_{-}
\big)^{\alpha}, \qquad t > 0.
\]

We distinguish two possible cases.

(I) Suppose first that $\alpha\leq\lambda$. Routine computations show, under
this assumption, that $\varphi$ is a non-decreasing function. It follows from
\eqref{hshash} that
\begin{align*}
\sum_{i=1}^{\infty}|Q_{i}|  &  \leq\frac{1}{2\,\max \{1,e^{\lambda-\alpha
} (\frac{\alpha}{\lambda})^{\alpha} \}\,}\frac{|Q_{0}|}%
{\varphi(|Q_{0}|)\int_{Q_{0}}|f(y)|\,dy}\,\sum_{i=1}^{\infty}\varphi
(|Q_{i}|)\int_{Q_{i}}|f(y)|\,dy\\
&  \leq\frac{1}{2\,\max \{1,e^{\lambda-\alpha} (\frac{\alpha}{\lambda
})^{\alpha} \}\,}\frac{|Q_{0}|}{\int_{Q_{0}}|f(y)|\,dy}\,\sum_{i=1}^{\infty}\int_{Q_{i}}|f(y)|\,dy  \leq\frac{|Q_{0}|}{2}.
\end{align*}
Note that, in the first step of above computations, we assume that $f$ is not
identically zero (almost everywhere) on $Q_{0}$; otherwise the desired result
\eqref{SparseDominationEq2} holds trivially. Therefore, if we assign to the
cube $Q_{0}$ the set%
\begin{equation}
E_{Q_{0}}:=Q_{0}\backslash\bigcup_{i=1}^{\infty}Q_{i}, \label{EQDef}%
\end{equation}
then%
\begin{equation}
\frac{\left\vert Q_{0}\right\vert }{2}\leq|E_{Q_{0}}|, \label{laitera2}%
\end{equation}
i.e., the sparseness condition given in Definition \ref{DefnSparse}(ii) holds.

(II) Suppose now that $\alpha>\lambda$. Let
\[
\psi(t):=\left\{
\begin{array}
[c]{ll}%
t^{\lambda}\,\big(1-\log t\big)^{\alpha}, & \text{if}\qquad t\in
(0,e^{1-\frac{\alpha}{\lambda}}),\\
& \\
e^{\lambda-\alpha}\,\big(\frac{\alpha}{\lambda}\big)^{\alpha}, &
\text{if}\qquad t\in\lbrack e^{1-\frac{\alpha}{\lambda}},\infty).
\end{array}
\right.
\]
It is plain that $\psi$ is a non-decreasing function such that, moreover,
$\varphi(t)\leq\psi(t)$ for $t>0$. By \eqref{hshash}, we have
\begin{align*}
\sum_{i=1}^{\infty}|Q_{i}|  &  \leq\frac{1}{2\,\max\{1,e^{\lambda-\alpha
}(\frac{\alpha}{\lambda})^{\alpha}\}\,}\frac{|Q_{0}|}{\varphi(|Q_{0}%
|)\int_{Q_{0}}|f(y)|\,dy}\,\sum_{i=1}^{\infty}\varphi(|Q_{i}|)\int_{Q_{i}%
}|f(y)|\,dy\\
&  \leq\frac{1}{2\,\max\{1,e^{\lambda-\alpha}(\frac{\alpha}{\lambda})^{\alpha
}\}\,}\frac{|Q_{0}|}{\varphi(|Q_{0}|)\int_{Q_{0}}|f(y)|\,dy}\,\sum
_{i=1}^{\infty}\psi(|Q_{i}|)\int_{Q_{i}}|f(y)|\,dy\\
&  \leq\frac{1}{2\,\max\{1,e^{\lambda-\alpha}(\frac{\alpha}{\lambda})^{\alpha
}\}}\frac{|Q_{0}|\psi(|Q_{0}|)}{\varphi(|Q_{0}|)\int_{Q_{0}}|f(y)|\,dy}%
\,\sum_{i=1}^{\infty}\int_{Q_{i}}|f(y)|\,dy\\
&  \leq\frac{|Q_{0}|}{2\max\{1,e^{\lambda-\alpha}(\frac{\alpha}{\lambda
})^{\alpha}\}}\frac{\psi(|Q_{0}|)}{\varphi(|Q_{0}|)}.
\end{align*}
Furthermore, using the estimate
\[
\frac{\psi(|Q_{0}|)}{\varphi(|Q_{0}|)}\leq\max\bigg\{1,e^{\lambda-\alpha
}\,\bigg(\frac{\alpha}{\lambda}\bigg)^{\alpha}\bigg\},
\]
we obtain
\[
\sum_{i=1}^{\infty}|Q_{i}|\leq\frac{|Q_{0}|}{2}.
\]
Hence the set $E_{Q_{0}}$ defined by \eqref{EQDef} satisfies the required
sparseness condition \eqref{laitera2}.

\vspace{2mm} \textbf{Estimate (B)}: We will show that the procedure used to
estimate $(A)$ can be iterated to estimate each term of the sum $(B)$. Fix
$i\in\mathbb{N}$. Observe that\textbf{ }for $x\in Q_{i},$ the maximality of
the $Q_{i}$'s and the nesting property of dyadic cubes, yield
\begin{equation}
M_{n(\frac{1}{p}-\frac{1}{q}),\alpha,Q_{0}}f(x)=M_{n(\frac{1}{p}-\frac{1}%
{q}),\alpha,Q_{i}}f(x). \label{igualdad}%
\end{equation}
Indeed, we only need to prove that the right-hand side is $\geq$ the
left-hand side. Consider $Q$ a generic dyadic cube such that $x\in Q\subset
Q_{0}.$ In particular, $Q\cap Q_{i}\neq\emptyset$. Now there are two possible
situations. Firstly, if $Q\subseteq Q_{i},$ the cube $Q$ enters in the
competition for computing both, the left- and right-hand sides of
(\ref{igualdad}), which is consistent with what we wish to prove. Assume now
that $Q_{i}$ $\subset Q.$ In this case, since $Q_{i}$ is a maximal element of
$\mathcal{Q}_{f}$, we must have that $Q\notin\mathcal{Q}_{f}.$ Therefore (cf.
\eqref{ProofSparseDomination1})
\begin{equation}
\frac{\big(1-(\log|Q|)_{-}\big)^{\alpha}}{|Q|^{\frac{1}{p^{\prime}}+\frac
{1}{q}}}\,\int_{Q}|f(y)|\,dy<2\,\max\bigg\{1,e^{\lambda-\alpha}\bigg(\frac
{\alpha}{\lambda}\bigg)^{\alpha}\bigg\}\,\frac{\big(1-(\log|Q_{0}%
|)_{-}\big)^{\alpha}}{|Q_{0}|^{\frac{1}{p^{\prime}}+\frac{1}{q}}}\,\int%
_{Q_{0}}|f(y)|\,dy. \label{210}%
\end{equation}
On the other hand, since $Q_{i}\in\mathcal{Q}_{f}$, it follows that%
\begin{align}
2\,\max\bigg\{1,e^{\lambda-\alpha}\bigg(\frac{\alpha}{\lambda}\bigg)^{\alpha
}\bigg\}\,\frac{\big(1-(\log|Q_{0}|)_{-}\big)^{\alpha}}{|Q_{0}|^{\frac
{1}{p^{\prime}}+\frac{1}{q}}}\int_{Q_{0}}|f(y)|\,dy  &  \leq\frac
{\big(1-(\log|Q_{i}|)_{-}\big)^{\alpha}}{|Q_{i}|^{\frac{1}{p^{\prime}}%
+\frac{1}{q}}}\int_{Q_{i}}|f(y)|\,dy\nonumber\\
&  \leq M_{n(\frac{1}{p}-\frac{1}{q}),\alpha,Q_{i}}f(x). \label{Estimnew23}%
\end{align}
Putting together \eqref{210} and \eqref{Estimnew23},
\[
\frac{\big(1-(\log|Q|)_{-}\big)^{\alpha}}{|Q|^{\frac{1}{p^{\prime}}+\frac
{1}{q}}}\,\int_{Q}|f(y)|\,dy<M_{n(\frac{1}{p}-\frac{1}{q}),\alpha,Q_{i}}f(x),
\]
and taking now the supremum over all possible dyadic cubes $Q\subset Q_{0}$
with $x\in Q$, we arrive at the desired upper estimate $\leq$ in
(\ref{igualdad}).

By \eqref{igualdad}, we can write (B) as follows%
\begin{equation}
(B) = \sum_{i=1}^{\infty}M_{n(\frac{1}{p}-\frac{1}{q}),\alpha,Q_{i}%
}f(x)\,\mathbf{1}_{Q_{i}}(x). \label{noq}%
\end{equation}
The proof can be now completed applying the procedure used to estimate $(A)$
to each of the terms that appear on the right-hand side of (\ref{noq}).
\end{proof}

We are now ready for the

\begin{proof}
[Proof of Theorem \ref{ThmSparseRieszMaximal}]Let $f\in SR_{p,q}\log^{\alpha
}(Q_{0})$. In light of Proposition \ref{SparseDomination2}, there exists
$(Q_{i})_{i\in I}\in S(Q_{0})$ (depending, in particular, on $f$) such that
the estimate \eqref{SparseDominationEq2} holds. Then, taking $L^{q}$-norms on
both sides of this estimate, we find%
\begin{equation*}
\Vert M_{n(\frac{1}{p}-\frac{1}{q}),\alpha,Q_{0}}f\Vert_{L^{q}(Q_{0})}%
\lesssim\,\bigg\|\sum_{i\in I}\bigg(\frac{\big(1-(\log|Q_{i}|)_{-}%
\big)^{\alpha}}{|Q_{i}|^{\frac{1}{p^{\prime}}+\frac{1}{q}}}\int_{Q_{i}%
}|f(y)|\,dy\bigg)\mathbf{1}_{Q_{i}}\bigg\|_{L^{q}(Q_{0})}.
\end{equation*}
To estimate the right-hand side, we shall use duality, the properties of
sparseness and the Hardy--Littlewood maximal theorem (recall that $q<\infty$).
This requires a number of elementary manipulations, but to facilitate the
reading we present all the steps,
\begin{align*}
\Vert M_{n(\frac{1}{p}-\frac{1}{q}),\alpha,Q_{0}}f\Vert_{L^{q}(Q_{0})}  &
\lesssim\sup_{\left\Vert g\right\Vert _{L^{q^{\prime}}(Q_{0})}\leq1}\int%
_{Q_{0}}\left(  \sum_{i\in I}\frac{\big(1-(\log|Q_{i}|)_{-}\big)^{\alpha}%
}{|Q_{i}|^{\frac{1}{p^{\prime}}+\frac{1}{q}}}\int_{Q_{i}}%
|f(y)|\,dy\,\mathbf{1}_{Q_{i}}(x)\right)  \left\vert g(x)\right\vert dx\\
&  \hspace{-2cm}=\sup_{\left\Vert g\right\Vert _{L^{q^{\prime}}(Q_{0})}\leq
1}\sum_{i\in I}\frac{\big(1-(\log|Q_{i}|)_{-}\big)^{\alpha}}{|Q_{i}|^{\frac
{1}{p^{\prime}}+\frac{1}{q}}}\int_{Q_{i}}|f(y)|\,dy\int_{Q_{i}}\left\vert
g(x)\right\vert dx\\
&  \hspace{-2cm}\lesssim\,\sup_{\left\Vert g\right\Vert _{L^{q^{\prime}}%
(Q_{0})}\leq1}\sum_{i\in I}\frac{\big(1-(\log|Q_{i}|)_{-}\big)^{\alpha}%
}{|Q_{i}|^{\frac{1}{p^{\prime}}+\frac{1}{q}}}\,\frac{|E_{Q_{i}}|}{|Q_{i}%
|}\,\int_{Q_{i}}|f(y)|\,dy\int_{Q_{i}}\left\vert g(x)\right\vert dx\\
&  \hspace{-2cm}=\sup_{\left\Vert g\right\Vert _{L^{q^{\prime}}(Q_{0})}\leq
1}\sum_{i\in I}\frac{\big(1-(\log|Q_{i}|)_{-}\big)^{\alpha}}{|Q_{i}|^{\frac
{1}{p^{\prime}}+\frac{1}{q}}}\int_{Q_{i}}|f(y)|\,dy\int_{E_{Q_{i}}}\left(
\frac{1}{|Q_{i}|}\int_{Q_{i}}\left\vert g(u)\right\vert du\right)  dx\\
&  \hspace{-2cm}\leq\sup_{\left\Vert g\right\Vert _{L^{q^{\prime}}(Q_{0})}%
\leq1}\sum_{i\in I}\frac{\big(1-(\log|Q_{i}|)_{-}\big)^{\alpha}}%
{|Q_{i}|^{\frac{1}{p^{\prime}}+\frac{1}{q}}}\int_{Q_{i}}|f(y)|\,dy\int%
_{E_{Q_{i}}}M_{Q_{0}}g(x)\,dx\\
&  \hspace{-2cm}=\sup_{\left\Vert g\right\Vert _{L^{q^{\prime}}(Q_{0})}\leq
1}\int_{Q_{0}}\left(  \sum_{i\in I}\frac{\big(1-(\log|Q_{i}|)_{-}%
\big)^{\alpha}}{|Q_{i}|^{\frac{1}{p^{\prime}}+\frac{1}{q}}}\int_{Q_{i}%
}|f(y)|\,dy\,\mathbf{1}_{E_{Q_{i}}}(x)\right)  M_{Q_{0}}g(x)\,dx\\
&  \hspace{-2cm}\leq\sup_{\left\Vert g\right\Vert _{L^{q^{\prime}}(Q_{0})}%
\leq1}\left\Vert M_{Q_{0}}g\right\Vert _{L^{q^{\prime}}(Q_{0})}\left\Vert
\sum_{i\in I}\frac{\big(1-(\log|Q_{i}|)_{-}\big)^{\alpha}}{|Q_{i}|^{\frac
{1}{p^{\prime}}+\frac{1}{q}}}\int_{Q_{i}}|f(y)|\,dy\,\mathbf{1}_{E_{Q_{i}}%
}\right\Vert _{L^{q}(Q_{0})}\\
&  \hspace{-2cm}\lesssim q\,\left\Vert \sum_{i\in I}\frac{\big(1-(\log
|Q_{i}|)_{-}\big)^{\alpha}}{|Q_{i}|^{\frac{1}{p^{\prime}}+\frac{1}{q}}}%
\int_{Q_{i}}|f(y)|\,dy\,\mathbf{1}_{E_{Q_{i}}}\right\Vert _{L^{q}(Q_{0})}\\
&  \hspace{-2cm}=q\,\left(  \sum_{i\in I}\left(  \frac{\big(1-(\log
|Q_{i}|)_{-}\big)^{\alpha}}{|Q_{i}|^{\frac{1}{p^{\prime}}+\frac{1}{q}}}%
\int_{Q_{i}}|f(y)|\,dy\right)  ^{q}\left\vert E_{Q_{i}}\right\vert \right)
^{1/q}\\
&  \hspace{-2cm}\leq q\,\left(  \sum_{i\in I}\left(  \frac{\big(1-(\log
|Q_{i}|)_{-}\big)^{\alpha}}{|Q_{i}|^{\frac{1}{p^{\prime}}}}\int_{Q_{i}%
}|f(y)|\,dy\right)  ^{q}\right)  ^{1/q}\\
&  \hspace{-2cm}\leq q\,\left\Vert f\right\Vert _{SR_{p,q}\log^{\alpha}%
(Q_{0})}.
\end{align*}

Conversely, for any $(Q_{i})_{i\in I}\in S(Q_{0})$, we have (recalling the
sparseness condition given in Definition \ref{DefnSparse}(ii))
\begin{align*}
\sum_{i\in I}\bigg(\frac{\big( 1 - (\log|Q_{i}|)_{-} \big)^{\alpha}}%
{|Q_{i}|^{\frac{1}{p^{\prime}}}}\int_{Q_{i}}|f(y)|\,dy\bigg)^{q}  &  \leq2 \,
\sum_{i\in I}\bigg(\frac{\big( 1 - (\log|Q_{i}|)_{-} \big)^{\alpha}}%
{|Q_{i}|^{\frac{1}{p^{\prime}}+\frac{1}{q}}}\int_{Q_{i}}|f(y)|\,dy\bigg)^{q}
|E_{Q_{i}}|\\
&  \hspace{-4cm} =2\int_{Q_{0}} \sum_{i\in I} \bigg(\frac{\big( 1 -
(\log|Q_{i}|)_{-} \big)^{\alpha}}{|Q_{i}|^{\frac{1}{p^{\prime}}+\frac{1}{q}}%
}\int_{Q_{i}}|f(y)|\,dy\bigg)^{q} \mathbf{1}_{E_{Q_{i}}}(x)\,dx\\
&  \hspace{-4cm}\leq2 \, \Vert M_{n(\frac{1}{p}-\frac{1}{q}),\alpha,Q_{0}%
}f\Vert_{L^{q}(Q_{0})}^{q}.
\end{align*}
Taking the supremum over all $(Q_{i})_{i\in I}\in S(Q_{0})$, we arrive at
\[
\Vert f\Vert_{SR_{p,q}\log^{\alpha}(Q_{0})}^{q} \leq2 \, \Vert M_{n(\frac
{1}{p}-\frac{1}{q}),\alpha,Q_{0}}f\Vert^{q}_{L^{q}(Q_{0})},
\]
concluding the proof.
\end{proof}

\begin{remark}
Note that, since the Hardy--Littlewood maximal function is not bounded on
$L^{1},$ the above proof does not work if $q=\infty$. However, Theorem
\ref{ThmSparseRieszMaximal} holds trivially if $q=\infty$ (for any value of
$\alpha\in\mathbb{R}$) since
\[
\Vert f\Vert_{SR_{p,\infty}\log^{\alpha}(Q_{0})}=\Vert f\Vert_{M^{p,\alpha
}(Q_{0})}=\Vert M_{\frac{n}{p},\alpha,Q_{0}}f\Vert_{L^{\infty}(Q_{0})};
\]
cf. (\ref{aparecio}) and (\ref{maxlam}).
\end{remark}

\subsection{Spaces defined on the whole space\label{sec:whole}}
%


The analogue of Theorem \ref{ThmSparseRieszMaximal} for $SR_{p,q}\log^{\alpha
}(\mathbb{R}^{n})$ can be now formulated in terms of the (dyadic) maximal
function
\begin{equation}
M_{\lambda,\alpha}f(x)=\sup_{\substack{Q\in\mathcal{D}(\mathbb{R}^{n})\\x\in
Q}}|Q|^{\frac{\lambda}{n}-1}\big(1-(\log|Q|)_{-}\big)^{\alpha}\int%
_{Q}|f(y)|\,dy,\qquad x\in\mathbb{R}^{n}, \label{maxlamRn}%
\end{equation}
where $\lambda\in\lbrack0,n)$ and\footnote{In the
absence of the $\log$-parameter (i.e., $\alpha=0$), we simply write
$M_{\lambda}$ instead of $M_{\lambda,0}$. If, in addition, $\lambda=0$ then we
get back the classical Hardy--Littlewood maximal function $M$.} $\alpha\in\mathbb{R}$. For
$q\in\lbrack1,\infty)$, we define
\[
M_{\lambda,\alpha}L^{q}(\mathbb{R}^{n})=\{f\in L_{\text{loc}}^{1}%
(\mathbb{R}^{n}):\Vert M_{\lambda,\alpha}f\Vert_{L^{q}(\mathbb{R}^{n})}%
<\infty\}.
\]

\begin{theorem}
\label{ThmSparseRieszMaximalRn} Suppose that $p,q, \alpha$ satisfy \eqref{AssumptionMon}.  Then
\[
SR_{p,q}\log^{\alpha}(\mathbb{R}^{n})=M_{n(\frac{1}{p}-\frac{1}{q}),\alpha
}L^{q}(\mathbb{R}^{n}).
\]

\end{theorem}

\begin{proof}
Consider the sequence of (not dyadic) cubes
\[
Q_{k} := [-2^{k}, 2^{k}]^{n}, \qquad k \in\mathbb{N}.
\]
According to Theorem \ref{ThmSparseRieszMaximal}, with equivalence constants
independent of $f$ and $k$,
\begin{equation}
\label{FatouMax0}\|f \mathbf{1}_{Q_{k}}\|_{SR_{p, q} \log^{\alpha}(Q_{k})}
\approx\|M_{n (\frac{1}{p} -\frac{1}{q}), \alpha, Q_{k}} (f \mathbf{1}_{Q_{k}%
})\|_{L^{q}(Q_{k})}.
\end{equation}

We claim that, for every $k\in\mathbb{N}$ and $x\in Q_{k}$,
\begin{equation}
M_{n(\frac{1}{p}-\frac{1}{q}),\alpha,Q_{k}}(f\mathbf{1}_{Q_{k}})(x)\approx
M_{n(\frac{1}{p}-\frac{1}{q}),\alpha}(f\mathbf{1}_{Q_{k}})(x),
\label{CorSparSobProof1}%
\end{equation}
(cf. \eqref{maxlam} and \eqref{maxlamRn}). Indeed, the estimate $\lesssim$
follows from the simple observation that $\mathcal{D}(Q_{k})\backslash
\{Q_{k}\}\subset\mathcal{D}(\mathbb{R}^{n})$, and the fact that the first
(dyadic) generation of $Q_{k}$, say $\{Q_{k,l}^{1}:l=1,\ldots,2^{n}\}$, gives
a pairwise disjoint decomposition of $Q_{k}$. In particular,
\[
Q_{k}=\bigcup_{l=1}^{2^{n}}Q_{k,l}^{1}%
\]
and $|Q_{k,l}^{1}|=2^{-n}|Q_{k}|$. Hence, given any $x\in Q_{k}$,
\begin{align*}
|Q_{k}|^{\frac{1}{p}-\frac{1}{q}-1}\,\int_{Q_{k}}|f(y)|\,dy  &  =2^{n(\frac
{1}{p}-\frac{1}{q}-1)}\,\sum_{l=1}^{2^{n}}|Q_{k,l}^{1}|^{\frac{1}{p}-\frac
{1}{q}-1}\int_{Q_{k,l}^{1}}|f(y)|\,dy\\
&  \leq2^{n(\frac{1}{p}-\frac{1}{q})}\,M_{n(\frac{1}{p}-\frac{1}{q}),\alpha
}(f\mathbf{1}_{Q_{k}})(x).
\end{align*}
Accordingly
\begin{align*}
M_{n(\frac{1}{p}-\frac{1}{q}),\alpha,Q_{k}}(f\mathbf{1}_{Q_{k}})(x)  &
\leq\sup_{\substack{Q\in\mathcal{D}(Q_{k})\backslash\{Q_{k}\}\\x\in
Q}}\,|Q|^{\frac{1}{p}-\frac{1}{q}-1}(1-(\log|Q|)_{-})^{\alpha}\,\int_{Q\cap
Q_{k}}|f(y)|\,dy\\
&  \hspace{1cm}+|Q_{k}|^{\frac{1}{p}-\frac{1}{q}-1}\,\int_{Q_{k}}|f(y)|\,dy\\
&  \lesssim\sup_{\substack{Q\in\mathcal{D}(\mathbb{R}^{n})\\x\in Q}%
}\,|Q|^{\frac{1}{p}-\frac{1}{q}-1}(1-(\log|Q|)_{-})^{\alpha}\,\int%
_{Q}|f(y)|\,\mathbf{1}_{Q_{k}}(y)\,dy\\
&  \hspace{1cm}+M_{n(\frac{1}{p}-\frac{1}{q}),\alpha}(f\mathbf{1}_{Q_{k}%
})(x)\\
&  \approx M_{n(\frac{1}{p}-\frac{1}{q}),\alpha}(f\mathbf{1}_{Q_{k}})(x).
\end{align*}

Next we focus on the estimate $\gtrsim$ in \eqref{CorSparSobProof1}. Consider
$x\in Q_{k},$ and $Q\in\mathcal{D}(\mathbb{R}^{n})$ with $Q\ni x$ and moreover
$Q\not \subset Q_{k}$ (indeed, if $Q\subset Q_{k}$ then $Q\in\mathcal{D}%
(Q_{k})$). We cannot assert that $Q_{k}\subset Q$ (recall that $Q_{k}$ is not
dyadic), but what is certainly true is that there exists $Q_{k}^{1}%
\in\mathcal{D}(Q_{k}),$ in the first dyadic generation (i.e., $2\,\ell
(Q_{k}^{1})=\ell(Q_{k})$), such that $Q_{k}^{1}\subset Q$ (because
$\mathcal{D}(Q_{k})\backslash\{Q_{k}\}\subset\mathcal{D}(\mathbb{R}^{n})$ and
$x\in Q_{k}$). Note that, in particular, $|Q|\geq|Q_{k}^{1}|=\ell(Q_{k}%
^{1})^{n}=2^{-n}\ell(Q_{k})^{n}=2^{-n}|Q_{k}|=2^{kn}>1$. Since $\frac{1}%
{p}-\frac{1}{q}-1<0,\ $we have
\begin{align*}
|Q|^{\frac{1}{p}-\frac{1}{q}-1}(1-(\log|Q|)_{-})^{\alpha}  &  =|Q|^{\frac
{1}{p}-\frac{1}{q}-1}\leq|Q_{k}^{1}|^{\frac{1}{p}-\frac{1}{q}-1}%
=2^{n(1+\frac{1}{q}-\frac{1}{p})}\,|Q_{k}|^{\frac{1}{p}-\frac{1}{q}-1}\\
&  =2^{n(1+\frac{1}{q}-\frac{1}{p})}\,|Q_{k}|^{\frac{1}{p}-\frac{1}{q}%
-1}(1-(\log|Q_{k}|)_{-})^{\alpha},
\end{align*}
which yields
\begin{align*}
|Q|^{\frac{1}{p}-\frac{1}{q}-1}(1-(\log|Q|)_{-})^{\alpha}\int_{Q\cap Q_{k}%
}|f(y)|\,dy  &  \leq 2^{n(1+\frac{1}{q}-\frac{1}{p})}\,|Q_{k}|^{\frac{1}{p}%
-\frac{1}{q}-1}(1-(\log|Q_{k}|)_{-})^{\alpha}\int_{Q_{k}}|f(y)|\,dy\\
&  \leq2^{n(1+\frac{1}{q}-\frac{1}{p})}\,M_{n(\frac{1}{p}%
-\frac{1}{q}),\alpha,Q_{k}}(f\mathbf{1}_{Q_{k}})(x).
\end{align*}
Consequently,
\[
M_{n(\frac{1}{p}-\frac{1}{q}),\alpha}(f\mathbf{1}_{Q_{k}})(x)\leq
2^{n(1+\frac{1}{q}-\frac{1}{p})}\,M_{n(\frac{1}{p}-\frac{1}{q}),\alpha,Q_{k}%
}(f\mathbf{1}_{Q_{k}})(x).
\]
This completes the proof of \eqref{CorSparSobProof1}.

On the other hand, since
\begin{equation}
\label{CubesStruct}Q_{k} \subset Q_{k+1} \qquad\text{and} \qquad\bigcup_{k
\in\mathbb{N}} Q_{k} = \mathbb{R}^{n},
\end{equation}
we have
\begin{equation}
\label{FatouMax1}\lim_{k \to\infty} M_{n (\frac{1}{p} - \frac{1}{q}), \alpha}
(f \mathbf{1}_{Q_{k}}) (x) \mathbf{1}_{Q_{k}} (x) = M_{n (\frac{1}{p} -
\frac{1}{q}), \alpha} f (x), \qquad x \in\mathbb{R}^{n}.
\end{equation}
Indeed, given any fixed $x \in\mathbb{R}^{n}$, we have
\[
M_{n (\frac{1}{p} - \frac{1}{q}), \alpha} (f \mathbf{1}_{Q_{k}}) (x)
\mathbf{1}_{Q_{k}} (x) \leq M_{n (\frac{1}{p} - \frac{1}{q}), \alpha} f (x),
\qquad\text{for all} \qquad k \in\mathbb{N},
\]
and (cf. \eqref{CubesStruct})
\begin{equation}
\label{FatouMax2}M_{n (\frac{1}{p} - \frac{1}{q}), \alpha} (f \mathbf{1}%
_{Q_{k}}) (x) \mathbf{1}_{Q_{k}} (x) \leq M_{n (\frac{1}{p} - \frac{1}{q}),
\alpha} (f \mathbf{1}_{Q_{k+1}}) (x) \mathbf{1}_{Q_{k+1}} (x).
\end{equation}
By the monotone convergence theorem for sequences of real numbers, we derive
\begin{align*}
\lim_{k \to\infty} M_{n (\frac{1}{p} - \frac{1}{q}), \alpha} (f \mathbf{1}%
_{Q_{k}}) (x) \mathbf{1}_{Q_{k}} (x)  &  = \sup_{k \in\mathbb{N}} M_{n
(\frac{1}{p} - \frac{1}{q}), \alpha} (f \mathbf{1}_{Q_{k}}) (x)\\
&  = \sup_{\substack{Q\in\mathcal{D}(\mathbb{R}^{n})\\x\in Q}}
|Q|^{\frac{1}{p}-\frac{1}{q}-1} \big(1- (\log|Q|)_{-} \big)^{\alpha} \sup_{k
\in\mathbb{N}} \int_{Q \cap Q_{k}}|f(y)| \,dy\\
&  = M_{n (\frac{1}{p} - \frac{1}{q}), \alpha} f (x),
\end{align*}
where we have used \eqref{CubesStruct} in the last step.

It follows from \eqref{CorSparSobProof1} that
\[
\Vert M_{n(\frac{1}{p}-\frac{1}{q}),\alpha,Q_{k}}(f\mathbf{1}_{Q_{k}}%
)\Vert_{L^{q}(Q_{k})}\approx\Vert M_{n(\frac{1}{p}-\frac{1}{q}),\alpha
}(f\mathbf{1}_{Q_{k}})\mathbf{1}_{Q_{k}}\Vert_{L^{q}(\mathbb{R}^{n})},
\]
uniformly with respect to $k$, consequently applying the monotone convergence
theorem (cf. \eqref{FatouMax1} and \eqref{FatouMax2}):
\begin{equation}
\lim_{k\rightarrow\infty}\Vert M_{n(\frac{1}{p}-\frac{1}{q}),\alpha,Q_{k}%
}(f\mathbf{1}_{Q_{k}})\Vert_{L^{q}(Q_{k})}\approx\Vert M_{n(\frac{1}{p}%
-\frac{1}{q}),\alpha}f\Vert_{L^{q}(\mathbb{R}^{n})}. \label{FatouMax3}%
\end{equation}

Next we deal with the left-hand side of \eqref{FatouMax0}. We claim that
\begin{equation}
\label{FatouMax4}\|f \mathbf{1}_{Q_{k}}\|_{SR_{p, q} \log^{\alpha}(Q_{k})}
\approx\|f \mathbf{1}_{Q_{k}}\|_{SR_{p, q} \log^{\alpha}(\mathbb{R}^{n})}%
\end{equation}
uniformly with respect to $k$ and $f$.

The estimate $\lesssim$ can be obtained as follows. Given any $(Q_{i})_{i\in
I}\in S(Q_{k})$, there are two possible scenarios. (I) $Q_{i}\neq Q_{k}$ for
every $i\in I$. Then $(Q_{i})_{i\in I}\in S(\mathbb{R}^{n}),$ since
$\mathcal{D}(Q_{k})\backslash\{Q_{k}\}\subset\mathcal{D}(\mathbb{R}^{n})$.
Clearly, this implies $\Vert f\mathbf{1}_{Q_{k}}\Vert_{SR_{p,q}\log^{\alpha
}(Q_{k})}\leq\Vert f\mathbf{1}_{Q_{k}}\Vert_{SR_{p,q}\log^{\alpha}%
(\mathbb{R}^{n})}$. (II) Suppose now that there is $i_{0}\in I$ such that
$Q_{i_{0}}=Q_{k}$. In particular, $(Q_{i})_{i\in I\backslash\{i_{0}\}}\in
S(\mathbb{R}^{n})$ and $Q_{i}\subset Q_{k}$ for $i\in I\backslash\{i_{0}\}$.
Now, the first dyadic decomposition of $Q_{k}$ (i.e., $\{Q_{k,l}%
^{1}:l=1,\ldots,2^{n}\}$) is formed by pairwise disjoint cubes in
$\mathcal{D}(\mathbb{R}^{n})$ (so, in particular, $\{Q_{k,l}^{1}%
:l=1,\ldots,2^{n}\}\in S(\mathbb{R}^{n})$) with $|Q_{k,l}^{1}|=2^{-n}|Q_{k}|$.
Hence, in this case, we can split the sum related to the $SR_{p,q}\log
^{\alpha}(Q_{k})$-norm as
\begin{align*}
\sum_{i\in I}\bigg(\frac{(1-(\log|Q_{i}|)_{-})^{\alpha}}{|Q_{i}|^{\frac
{1}{p^{\prime}}}}\,\int_{Q_{i}}|f(y)|\,dy\bigg)^{q}  &  =\\
&  \hspace{-5cm}\sum_{\substack{i\in I\\i\neq i_{0}}}\bigg(\frac
{(1-(\log|Q_{i}|)_{-})^{\alpha}}{|Q_{i}|^{\frac{1}{p^{\prime}}}}\,\int_{Q_{i}%
}|f(y)|\,dy\bigg)^{q}+\bigg(\frac{1}{|Q_{k}|^{\frac{1}{p^{\prime}}}}%
\,\int_{Q_{k}}|f(y)|\,dy\bigg)^{q}\\
&  \hspace{-5cm}\leq\Vert f\mathbf{1}_{Q_{k}}\Vert_{SR_{p,q}\log^{\alpha
}(\mathbb{R}^{n})}^{q}+\frac{1}{2^{\frac{nq}{p^{\prime}}}}\,\bigg(\sum
_{l=1}^{2^{n}}\frac{1}{|Q_{k,l}^{1}|^{\frac{1}{p^{\prime}}}}\int_{Q_{k,l}^{1}%
}|f(y)|\,dy\bigg)^{q}\\
& \hspace{-5cm}\lesssim \Vert f\mathbf{1}_{Q_{k}}\Vert_{SR_{p,q}\log^{\alpha
}(\mathbb{R}^{n})}^{q}+ \sum_{l=1}^{2^{n}} \bigg(\frac{1}{|Q_{k,l}^{1}|^{\frac{1}{p^{\prime}}}}\int_{Q_{k,l}^{1}%
}|f(y)|\,dy\bigg)^q\\
&  \hspace{-5cm}\lesssim\Vert f\mathbf{1}_{Q_{k}}\Vert_{SR_{p,q}\log^{\alpha
}(\mathbb{R}^{n})}^{q}.
\end{align*}
Therefore, taking the supremum over all possible $(Q_{i})_{i\in I}\in
S(Q_{k})$, we achieve
\[
\Vert f\mathbf{1}_{Q_{k}}\Vert_{SR_{p,q}\log^{\alpha}(Q_{k})}\lesssim\Vert
f\mathbf{1}_{Q_{k}}\Vert_{SR_{p,q}\log^{\alpha}(\mathbb{R}^{n})},
\]
i.e., the estimate $\lesssim$ in \eqref{FatouMax4} is shown.

To deal with the converse estimate, for any $\mathcal{Q} = (Q_{i})_{i \in I}
\in S(\mathbb{R}^{n})$, we consider the index set
\[
I_{k} := \{i \in I : Q_{i} \subset Q_{k}\}.
\]
Therefore we can split
\begin{align}
\sum_{i \in I} \bigg( \frac{(1-(\log|Q_{i}|)_{-})^{\alpha}}{|Q_{i}|^{\frac
{1}{p^{\prime}}}} \, \int_{Q_{i}} |f(y)| \, \mathbf{1}_{Q_{k}}(y) \, dy
\bigg)^{q}  &  =\nonumber\\
&  \hspace{-6cm} \sum_{i \in I_{k}} \bigg( \frac{(1-(\log|Q_{i}|)_{-}%
)^{\alpha}}{|Q_{i}|^{\frac{1}{p^{\prime}}}} \, \int_{Q_{i}} |f(y)| \, dy
\bigg)^{q}\nonumber\\
&  \hspace{-5cm}+ \sum_{i \in I \backslash I_{k}} \bigg( \frac{(1-(\log
|Q_{i}|)_{-})^{\alpha}}{|Q_{i}|^{\frac{1}{p^{\prime}}}} \, \int_{Q_{i} \cap
Q_{k}} |f(y)| \, dy \bigg)^{q}\nonumber\\
&  \hspace{-6cm} =: R_{1} + R_{2}. \label{SparseSumDecomp1}%
\end{align}

Note that $(Q_{i})_{i \in I_{k}} \in S(Q_{k})$ (since $(Q_{i})_{i \in I_{k}}
\in S(\mathbb{R}^{n}) \cap\mathcal{D}(Q_{k})$). Accordingly
\begin{equation}
\label{SparseSumDecomp2}R_{1} \leq\|f \mathbf{1}_{Q_{k}}\|_{SR_{p, q}
\log^{\alpha}(Q_{k})}^{q}.
\end{equation}

Concerning $R_{2}$, we argue as follows. Let $i\in I\backslash I_{k}$, i.e.,
$Q_{i}\not \subset Q_{k}$. Assume further that $Q_{i}\cap Q_{k}\neq\emptyset$.
Note that $Q_{k}$ is not a dyadic cube in $\mathbb{R}^{n}$, but its first
dyadic generation $\{Q_{k,l}^{1}:l=1,\ldots,2^{n}\}$ is formed by dyadic cubes
in $\mathbb{R}^{n}$. Since $Q_{k}$ can be expressed as the disjoint union of
the cubes $Q_{k,l}^{1}$, we can assert that there exists a unique $l(i)$ such
that $Q_{i}\cap Q_{k,l(i)}^{1}\neq\emptyset$. By the structure of dyadic cubes
in $\mathbb{R}^{n}$, we have either $Q_{i}\subset Q_{k,l(i)}^{1}$ or
$Q_{k,l(i)}^{1}\subset Q_{i}$. The former is not possible; otherwise,
$Q_{i}\subset Q_{k}$ but $i\not \in I_{k}$. Hence $Q_{k,l(i)}^{1}\subset
Q_{i}$, therefore
\begin{equation}
\int_{Q_{i}\cap Q_{k}}|f(y)|\,dy=\int_{Q_{k,l(i)}^{1}}|f(y)|\,dy.
\label{SparseSumDecomp3}%
\end{equation}

For $l \in\{1, \ldots, 2^{n}\}$, we define
\[
\mathcal{Q}_{k, l} := \{Q_{i} : i \in I \backslash I_{k} \quad\text{and} \quad
Q^{1}_{k, l} \subset Q_{i}\}.
\]
The above argument leads to
\begin{equation}
\label{SparseSumDecomp4}(Q_{i})_{i \in I \backslash I_{k}} = \bigcup
_{l=1}^{2^{n}} \mathcal{Q}_{k, l}.
\end{equation}
Moreover, since the $Q_{i}$'s are dyadic cubes in $\mathbb{R}^{n}$, the
elements of $\mathcal{Q}_{k, l} = \{Q_{1}, Q_{2}, \ldots\}$ can be ordered in
such a way that
$
Q^{1}_{k, l} \subset Q_{1} \subset Q_{2} \subset\ldots
$
We cannot exclude the possibility that some of the cubes in $\mathcal{Q}_{k,
l}$ coincide, but the number of these cubes is uniformly bounded by the sparse
constant $\eta$ in Definition \ref{DefnSparse}. Therefore, by
\eqref{SparseSumDecomp3} and \eqref{SparseSumDecomp4},
\begin{align}
R_{2}  &  \leq\sum_{i \in I \backslash I_{k}} \bigg( \frac{1}{|Q_{i}%
|^{\frac{1}{p^{\prime}}}} \, \int_{Q^{1}_{k, l(i)}} |f(y)| \, dy
\bigg)^{q}\nonumber\\
&  = \sum_{l=1}^{2^{n}} \bigg(\int_{Q^{1}_{k, l}} |f(y)| \, dy \bigg)^{q}
\sum_{Q_{i} \in\mathcal{Q}_{k, l}} \frac{1}{|Q_{i}|^{\frac{q}{p^{\prime}}}%
}\nonumber\\
&  \lesssim\sum_{l=1}^{2^{n}} \bigg(\int_{Q^{1}_{k, l}} |f(y)| \, dy \bigg)^{q}
\sum_{j=k}^{\infty}2^{-j \frac{n q}{p^{\prime}}}\nonumber\\
&  \approx\sum_{l=1}^{2^{n}} \bigg(\int_{Q^{1}_{k, l}} |f(y)| \, dy \bigg)^{q}
2^{-k \frac{n q}{p^{\prime}}}\nonumber\\
&  = \sum_{l=1}^{2^{n}} \bigg(\frac{1}{|Q^{1}_{k, l}|^{\frac{1}{p^{\prime}}}}
\int_{Q^{1}_{k, l}} |f(y)| \, dy \bigg)^{q}\nonumber\\
&  \leq\|f \mathbf{1}_{Q_{k}}\|^{q}_{SR_{p, q} \log^{\alpha}(Q_{k})}.
\label{SparseSumDecomp5}%
\end{align}

Combining \eqref{SparseSumDecomp1}, \eqref{SparseSumDecomp2} and
\eqref{SparseSumDecomp5},
\[
\sum_{i\in I}\bigg(\frac{(1-(\log|Q_{i}|)_{-})^{\alpha}}{|Q_{i}|^{\frac
{1}{p^{\prime}}}}\,\int_{Q_{i}}|f(y)|\,\mathbf{1}_{Q_{k}}(y)\,dy\bigg)^{q}%
\lesssim\Vert f\mathbf{1}_{Q_{k}}\Vert_{SR_{p,q}\log^{\alpha}(Q_{k})}^{q},
\]
for all $(Q_{i})_{i\in I}\in S(\mathbb{R}^{n})$. In particular,
\[
\Vert f\mathbf{1}_{Q_{k}}\Vert_{SR_{p,q}\log^{\alpha}(\mathbb{R}^{n})}%
\lesssim\Vert f\mathbf{1}_{Q_{k}}\Vert_{SR_{p,q}\log^{\alpha}(Q_{k})},
\]
completing the proof of \eqref{FatouMax4}.

By the lattice property of sparse RMT spaces, we have (recall
\eqref{CubesStruct})
\[
\Vert f\mathbf{1}_{Q_{k}}\Vert_{SR_{p,q}\log^{\alpha}(\mathbb{R}^{n})}%
\leq\Vert f\mathbf{1}_{Q_{k+1}}\Vert_{SR_{p,q}\log^{\alpha}(\mathbb{R}^{n})}%
\]
and
\[
\Vert f\mathbf{1}_{Q_{k}}\Vert_{SR_{p,q}\log^{\alpha}(\mathbb{R}^{n})}%
\leq\Vert f\Vert_{SR_{p,q}\log^{\alpha}(\mathbb{R}^{n})},\qquad k\in
\mathbb{N}.
\]
Hence, applying the monotone convergence theorem and the dominated convergence
theorem, we obtain
\begin{align}
\lim_{k\rightarrow\infty}\Vert f\mathbf{1}_{Q_{k}}\Vert_{SR_{p,q}\log^{\alpha
}(\mathbb{R}^{n})}  &  =\sup_{k\in\mathbb{N}}\Vert f\mathbf{1}_{Q_{k}}%
\Vert_{SR_{p,q}\log^{\alpha}(\mathbb{R}^{n})}\nonumber\\
&  =\sup_{(Q_{i})_{i\in I}\in S(\mathbb{R}^{n})}\sup
_{k\in\mathbb{N}}\bigg\{\sum_{i\in I}\bigg[\frac{\big(1-(\log|Q_{i}%
|)_{-}\big)^{\alpha}}{|Q_{i}|^{\frac{1}{p^{\prime}}}}\,\int_{Q_{i}\cap Q_{k}%
}|f|\bigg]^{q}\bigg\}^{\frac{1}{q}}\nonumber\\
&  =\sup_{(Q_{i})_{i\in I}\in S(\mathbb{R}^{n})}\lim
_{k\rightarrow\infty}\bigg\{\sum_{i\in I}\bigg[\frac{\big(1-(\log|Q_{i}%
|)_{-}\big)^{\alpha}}{|Q_{i}|^{\frac{1}{p^{\prime}}}}\,\int_{Q_{i}\cap Q_{k}%
}|f|\bigg]^{q}\bigg\}^{\frac{1}{q}}\label{SparseSumDecomp6}\\
&  =\sup_{(Q_{i})_{i\in I}\in S(\mathbb{R}^{n})}\bigg\{\sum
_{i\in I}\bigg[\frac{\big(1-(\log|Q_{i}|)_{-}\big)^{\alpha}}{|Q_{i}|^{\frac
{1}{p^{\prime}}}}\,\int_{Q_{i}}|f|\bigg]^{q}\bigg\}^{\frac{1}{q}}\nonumber\\
&  =\Vert f\Vert_{SR_{p,q}\log^{\alpha}(\mathbb{R}^{n}%
)}.\nonumber
\end{align}
Finally, taking limits on both sides of \eqref{FatouMax0} as $k\rightarrow
\infty,$ and invoking \eqref{FatouMax3} and \eqref{SparseSumDecomp6}, we
achieve the desired estimate
\[
\Vert f\Vert_{SR_{p,q}\log^{\alpha}(\mathbb{R}^{n})}\approx\Vert M_{n(\frac
{1}{p}-\frac{1}{q}),\alpha}f\Vert_{L^{q}(\mathbb{R}^{n})}.
\]

\end{proof}

\section{A sparse approach to $H^{-1}$-stability\label{Section1.3}}

As already mentioned in Section \ref{sec:functionsbackground}, one of the main features of the theory of sparse function spaces lies in the
fact that, unlike their classical parent spaces, they often admit complete
explicit characterizations. Indeed, Theorem \ref{CorSparSob}  provides
us with the following surprising (informal) characterization: \emph{$SR_{p,q}$
spaces can be identified with negative Sobolev spaces}. Before we give the
proof of this result, we introduce some basic notation.

Consider the \emph{Riesz potential operators}
$I_{\lambda}, \, \lambda \in (0, n),$ formally defined, for $f\in L_{\text{loc}}^{1}(\mathbb{R}%
^{n}),$ by
\begin{equation*}
I_{\lambda}f(x):=\int_{\mathbb{R}^{n}}\frac{f(y)}{|x-y|^{n-\lambda}%
}\,dy,\qquad x\in\mathbb{R}^{n}.
\end{equation*}
For $1<q<\infty,$ we let
\begin{equation}
H_{q}^{-\lambda}(\mathbb{R}^{n}):=\{f\in L_{\text{loc}}^{1}(\mathbb{R}%
^{n}):\left\Vert f\right\Vert _{H_{q}^{-\lambda}}=\left\Vert I_{\lambda
}f\right\Vert _{L^{q}(\mathbb{R}^{n})}<\infty\},\label{rpot0}%
\end{equation}
the \emph{Riesz potential space}, and its associated lattice%
\begin{equation}
\mathcal{H}_{q}^{-\lambda}(\mathbb{R}^{n}):=\{f\in L_{\text{loc}}%
^{1}(\mathbb{R}^{n}):\left\Vert f\right\Vert _{\mathcal{H}_{q}^{-\lambda}%
}=\left\Vert I_{\lambda} (\left\vert f\right\vert) \right\Vert _{L^{q}%
(\mathbb{R}^{n})}<\infty\}.\label{rpot1}%
\end{equation}
It is plain that%
\begin{equation*}
\mathcal{H}_{q}^{-\lambda}(\mathbb{R}^{n})\subset H_{q}^{-\lambda}%
(\mathbb{R}^{n}).
\end{equation*}
Furthermore, as it is customary, we shall suppress the subindex $q=2$ and
simply write%
\begin{equation}\label{SLat}
\mathcal{H}^{-\lambda}(\mathbb{R}^{n}):=\mathcal{H}_{2}^{-\lambda}%
(\mathbb{R}^{n})\text{ \ \ (resp. }H^{-\lambda}(\mathbb{R}^{n}):=H_{2}%
^{-\lambda}(\mathbb{R}^{n})).
\end{equation}


%
%

\subsection{Proof of Theorem \ref{CorSparSob}}

In order to be able to use a result of Muckenhoupt-Wheeden \cite{MW74} we
introduce the fractional maximal\footnote{Compare with the dyadic local version $M_{\lambda,Q_0}$ defined in (\ref{maxlam}).} operator, defined for $f\in L_{\text{loc}}%
^{1}(\mathbb{R}^{n})$,%
\begin{equation*}
\mathcal{M}_{\lambda}f(x):=\sup_{x\in Q}|Q|^{\frac{\lambda}{n}-1}\int%
_{Q}|f(y)|\,dy,\qquad x\in\mathbb{R}^{n},
\end{equation*}
where the supremum runs over all (not necessarily dyadic) cubes $Q$ in
$\mathbb{R}^{n},$ with $x\in Q$. It is plain (cf. \eqref{maxlamRn}) that
$M_{\lambda,0}f(x)\leq\mathcal{M}_{\lambda}f(x),$ and, although this pointwise
inequality cannot be reversed, it is well-known that by the $1/3$-translation
trick (cf. \cite{Ch88}) we have the equivalence%
\begin{equation}
\Vert M_{\lambda,0}f\Vert_{L^{q}(\mathbb{R}^{n})}\approx\Vert\mathcal{M}%
_{\lambda}f\Vert_{L^{q}(\mathbb{R}^{n})}.\label{vale}%
\end{equation}

Putting together Theorem \ref{ThmSparseRieszMaximalRn} (with $\alpha=0$ and
$\lambda= n (\frac{1}{p} -\frac{1}{q})$) and \eqref{vale}, we get%
\begin{equation}
\Vert f\Vert_{SR_{p,q}(\mathbb{R}^{n})}\approx\Vert\mathcal{M}_{n(\frac{1}%
{p}-\frac{1}{q})}f\Vert_{L^{q}(\mathbb{R}^{n})}. \label{CorSparSobProof6}%
\end{equation}
For the maximal operator $\mathcal{M}_{\lambda}$ we have the trivial estimate
\[
\mathcal{M}_{\lambda}f(x)\leq c_{n} I_{\lambda}(|f|)(x),\qquad x\in
\mathbb{R}^{n},
\]
where $c_{n}$ depends only on $n$. In fact, via the Muckenhoupt--Wheeden
theorem \cite[Theorem 1]{MW74}, we achieve, for $0<q<\infty,$
\begin{equation}
\Vert\mathcal{M}_{\lambda}f\Vert_{L^{q}(\mathbb{R}^{n})}\approx\Vert
I_{\lambda}(|f|)\Vert_{L^{q}(\mathbb{R}^{n})}. \label{CorSparSobProof7}%
\end{equation}

Combining \eqref{rpot1}, \eqref{CorSparSobProof6} and \eqref{CorSparSobProof7}
we arrive at
\[
\Vert f\Vert_{SR_{p,q}(\mathbb{R}^{n})}\approx\Vert f\Vert_{\mathcal{H}%
_{q}^{-n(\frac{1}{p}-\frac{1}{q})}(\mathbb{R}^{n})},
\]
as we wished to show. \qed
%
%
%


Our next result refers to the limiting case $p=q$ in Theorem \ref{CorSparSob}
and it can be viewed as the sparse counterpart of the Riesz's theorem (cf.
\cite[p. 1062]{DMComptes})%
\begin{equation}
R_{p,p}(\mathbb{R}^{n})=L^{p}(\mathbb{R}^{n}).\label{rieszclasico}%
\end{equation}

\begin{theorem}
\label{CorSparLeb} Let $1<p<\infty$. Then
\[
SR_{p,p}(\mathbb{R}^{n})=L^{p}(\mathbb{R}^{n}).
\]

\end{theorem}

\begin{proof}
This is an immediate consequence of Theorem \ref{ThmSparseRieszMaximalRn}
(with $p=q$ and $\alpha=0$) and the classical Hardy--Littlewood maximal
theorem:
\[
\Vert Mf\Vert_{L^{p}(\mathbb{R}^{n})}\lesssim\Vert f\Vert_{L^{p}(\mathbb{R}%
^{n})},\qquad p>1.
\]

\end{proof}

\subsection{On the difference between $R_{1,2}(\mathbb{R}^{2})$ and
$SR_{1,2}(\mathbb{R}^{2})$\label{sec:lions}}

In view of Theorem \ref{CorSparLeb} and \eqref{rieszclasico},
\[
SR_{p,p}(\mathbb{R}^{n})=L^{p}(\mathbb{R}^{n})=R_{p,p}(\mathbb{R}^{n}%
),\qquad1<p<\infty.
\]
One may be tempted to think that $SR_{p,q}(\mathbb{R}^{n})=R_{p,q}%
(\mathbb{R}^{n})$ for general values of $p$ and $q$. However, this is far from
being true. Next we concentrate on the most relevant case for the purposes of
this paper, i.e., we will show that the embedding
\[
SR_{1,2}(\mathbb{R}^{2})\subset R_{1,2}(\mathbb{R}^{2})
\]
is strict, in the sense that,
\begin{equation}
SR_{1,2}(\mathbb{R}^{2})\neq R_{1,2}(\mathbb{R}^{2}).\label{GoalSR12}%
\end{equation}
As a by-product (cf. \eqref{SobIde})
\begin{equation}
R_{1,2}(\mathbb{R}^{2})\nsubseteq H^{-1}(\mathbb{R}^{2}).\label{agregada}%
\end{equation}

We shall use an elementary but indirect method: It is well known (see, e.g.,
\cite[p. 141]{Lio96}) that the largest rearrangement invariant space embedded
in $H_{\text{loc}}^{-1}(\mathbb{R}^{2})$ is the \emph{Lorentz space}%
\footnote{Without loss of generality, we may assume that $|Q_{0}|=1$.}
\[
L^{(1,2)}(Q_{0}):=\bigg\{f:\left\Vert f\right\Vert _{L^{(1,2)}(Q_{0})}=\left[
\int_{0}^{1}(tf^{\ast\ast}(t))^{2} \, \frac{dt}{t}\right]  ^{\frac{1}{2}}%
<\infty\bigg\}.
\]
As usual, here $f^{\ast}$ is the non-increasing rearrangement of a measurable
function $f$ and $f^{\ast\ast}$ its maximal function, $f^{\ast\ast}%
(t)=t^{-1}\int_{0}^{t}f^{\ast}(s)\,ds$. It is easy to see that\footnote{Using
that $t\mapsto tf^{**}(t)$ is an increasing function, we have, for every
$u\in(0,1)$,
\[
\Vert f\Vert_{L^{(1,2)}(Q_{0})}\geq\left[  \int_{u}^{1}(tf^{**}(t))^{2}%
\frac{dt}{t}\right]  ^{\frac{1}{2}}\geq(-\log u)^{\frac{1}{2}}uf^{**}(u).
\]
}
\[
L^{(1,2)}(Q_{0})\subset L^{(1,\infty)}(\log L)^{\frac{1}{2}}(Q_{0}),
\]
where $L^{(1,\infty)}(\log L)^{\frac{1}{2}}(Q_{0})$ is the
\emph{Lorentz--Zygmund space} defined by\footnote{A basic reference to
Lorentz--Zygmund spaces is \cite{BR80}.}
\[
\left\Vert f\right\Vert _{L^{(1,\infty)}(\log L)^{\frac{1}{2}}(Q_{0})}%
:=\sup_{0<t<1}t(1-\log t)^{\frac{1}{2}}f^{\ast\ast}(t);
\]
Then 
\[
L^{(1,\infty)}(\log L)^{\frac{1}{2}}(Q_{0})\nsubseteq H_{\text{loc}}%
^{-1}(\mathbb{R}^{2}),
\]
which in turn yields (cf. Theorem \ref{CorSparSob})
\begin{equation}
L^{(1,\infty)}(\log L)^{\frac{1}{2}}(Q_{0})\nsubseteq SR_{1,2,\text{loc}%
}(\mathbb{R}^{2}).\label{Coin1}%
\end{equation}
On the other hand, by the Hardy--Littlewood inequality for rearrangements (see
e.g. \cite[Lemma 2.1, p. 44]{BS88}), and (\ref{Coin1}),
\begin{align*}
\left\Vert f\right\Vert _{R_{1,2}(Q_{0})}^{2} &  =\sup_{(Q_{i})_{i\in I}\in
\Pi(Q_{0})}\sum_{i\in I}\left(  \int_{Q_{i}}\left\vert f\right\vert \right)
^{2}\\
&  \leq\sup_{(Q_{i})_{i\in I}\in\Pi(Q_{0})}\sum_{i\in I} \bigg(\int_{0}^{|Q_{i}%
|}f^{\ast} \bigg)\bigg(  \int_{Q_{i}}\left\vert f\right\vert \bigg)  \\
&  =\sup_{(Q_{i})_{i\in I}\in\Pi(Q_{0})}\sum_{i\in I}|Q_{i}|\,f^{\ast\ast
}(|Q_{i}|)\int_{Q_{i}}\left\vert f\right\vert \\
&  \leq\left\Vert f\right\Vert _{L^{(1,\infty)}(\log L)^{\frac{1}{2}}(Q_{0}%
)}\sup_{(Q_{i})\in\Pi(Q_{0})}\sum_{i\in I}(1-\log|Q_{i}|)^{-\frac{1}{2}}%
\int_{Q_{i}}\left\vert f\right\vert \\
&  \leq\left\Vert f\right\Vert _{L^{(1,\infty)}(\log L)^{\frac{1}{2}}(Q_{0}%
)}\sup_{(Q_{i})_{i\in I}\in\Pi(Q_{0})}\sum_{i\in I}\int_{Q_{i}}\left\vert
f\right\vert \\
&  \leq\left\Vert f\right\Vert _{L^{(1,\infty)}(\log L)^{\frac{1}{2}}(Q_{0}%
)}\left\Vert f\right\Vert _{L^{1}(Q_{0})}\\
&  \leq\left\Vert f\right\Vert _{L^{(1,\infty)}(\log L)^{\frac{1}{2}}(Q_{0}%
)}^{2}.
\end{align*}
It follows that
\begin{equation}
L^{(1,\infty)}(\log L)^{\frac{1}{2}}(Q_{0})\subset R_{1,2,\text{loc}%
}(\mathbb{R}^{2}).\label{Coin2}%
\end{equation}
Consequently, \eqref{GoalSR12} now follows from \eqref{Coin1} and \eqref{Coin2}.

\subsection{Proof of Theorem \ref{teo:compacto}\label{Section4.1}}

(i):  Let $Q_0$ be a cube. Recalling that
\begin{equation}
\Vert f\Vert_{SR_{\frac{2n}{n+2},2}(Q_{0})}=\sup_{(Q_{i})_{i\in I}\in
S(Q_{0})}\left\{  \sum_{i\in I}\bigg(|Q_{i}|^{\frac{1}{n}-\frac{1}{2}}%
\int_{Q_{i}}|f|\bigg)^{2}\right\}  ^{\frac{1}{2}},\label{Compactness3+}%
\end{equation}
we observe $s_1 (f) = \|f\|_{SR_{\frac{2 n}{n+2}, 2} (Q_0)}$ (cf. \eqref{spinf}). Then
$$
	s_1(X) = \sup_{\|f\|_{X(Q_0)} \leq 1}  \|f\|_{SR_{\frac{2 n}{n+2}, 2} (Q_0)} < \infty \iff X(Q_0) \hookrightarrow SR_{\frac{2 n}{n+2}, 2} (Q_0). 
$$
The desired assertions follow immediately from \eqref{SobIde} and \eqref{SobIde2}.


(ii): We need to introduce some notation: Given $\mathcal{Q}=(Q_{i})_{i\in
I}\in S(Q_{0})$, note that 
$\mathcal{Q}$ can be split as $\mathcal{Q}=\cup_{k=0}^{\infty}\mathbb{D}%
_{k;Q_{0}}(\mathcal{Q})$, where $\mathbb{D}_{k;Q_{0}}:=\{Q \in \mathcal{D}(Q_{0}): \ell(Q) = 2^{-k}\ell(Q_{0})\}$,
$\mathbb{D}_{k; Q_0}(\mathcal{Q)}:=\mathbb{D}_{k;Q_{0}}\cap\mathcal{Q}$.  When there is no danger of confusion,  we use the simplified notation $\mathbb{D}_{k}$  and $\mathbb{D}_{k}(\mathcal{Q)}$. By construction, $\mathbb{D}_{k}(\mathcal{Q})$
is formed by pairwise disjoint cubes (i.e., $\mathbb{D}_{k}%
(\mathcal{Q})\subset\Pi(Q_{0})$), and for $Q_{i}\in\mathbb{D}_{k
}(\mathcal{Q}),$ we have $|Q_{i}|=2^{-kn}|Q_{0}|$.

Assume that $\lim_{N \to \infty} s_N(X) =0$, i.e., given any $\varepsilon>0$ there exists $N_{0}\in\mathbb{N}$,
such that for all $N>N_{0}$
\begin{equation}
\sup_{\left\Vert f\right\Vert _{X(Q_{0})}} s_{N}(f)\leq\varepsilon.\label{Compactness31+}%
\end{equation}
Let $\mathcal{Q}=(Q_{i})_{i\in I}\in S(Q_{0})$ and let $f\in X(Q_{0})$ be such
that $\left\Vert f\right\Vert _{X(Q_{0})}\leq1.$ Then
\begin{align}
\left\{  \sum_{i\in I}\bigg(|Q_{i}|^{\frac{1}{n}-\frac{1}{2}}\int_{Q_{i}%
}|f|\bigg)^{2}\right\}  ^{\frac{1}{2}} &  =\left\{  \sum_{k=0}^{\infty}%
\sum_{i\in I:Q_{i}\in\mathbb{D}_{k}(\mathcal{Q})}\bigg(|Q_{i}|^{\frac{1}%
{n}-\frac{1}{2}}\int_{Q_{i}}|f|\bigg)^{2}\right\}  ^{\frac{1}{2}}\nonumber\\
&  \leq I+II,\label{Compactness4+}%
\end{align}
where
\[
I:=\left\{  \sum_{k=0}^{N_{0}}\sum_{i\in I:Q_{i}\in\mathbb{D}_{k}%
(\mathcal{Q})}\bigg(|Q_{i}|^{\frac{1}{n}-\frac{1}{2}}\int_{Q_{i}}%
|f|\bigg)^{2}\right\}  ^{\frac{1}{2}}%
\]
and
\[
II:=\left\{  \sum_{k=N_{0}+1}^{\infty}\sum_{i\in I:Q_{i}\in\mathbb{D}%
_{k}(\mathcal{Q})}\bigg(|Q_{i}|^{\frac{1}{n}-\frac{1}{2}}\int_{Q_{i}%
}|f|\bigg)^{2}\right\}  ^{\frac{1}{2}}.
\]

It follows from (\ref{spinf}) and \eqref{Compactness31+} that
\begin{equation}
II = \left\{ \sum_{i\in I:Q_{i}\in\mathbb{D}%
_{\leq N_0 +1}(\mathcal{Q})}\bigg(|Q_{i}|^{\frac{1}{n}-\frac{1}{2}}\int_{Q_{i}%
}|f|\bigg)^{2}\right\}  ^{\frac{1}{2}}  \leq s_{N_0 + 2}(f) \leq\varepsilon.\label{Compactness41+}%
\end{equation}
On the other hand, we obviously have
\begin{equation}
I\leq\left\{  \sum_{k=0}^{N_{0}}\sum_{Q\in\mathbb{D}_{k}}%
\bigg(|Q|^{\frac{1}{n}-\frac{1}{2}}\int_{Q}|f|\bigg)^{2}\right\}  ^{\frac
{1}{2}}.\label{Compactness6+}%
\end{equation}

Combining \eqref{Compactness4+}, \eqref{Compactness41+} and
\eqref{Compactness6+}, we find that, for all $(Q_{i})_{i\in I}\in S(Q_{0})$
and for all $f$ in the unit ball of $X(Q_{0}),$%
\[
\left\{  \sum_{i\in I}\bigg(|Q_{i}|^{\frac{1}{n}-\frac{1}{2}}\int_{Q_{i}%
}|f|\bigg)^{2}\right\}  ^{\frac{1}{2}}\leq\left\{  \sum_{k=0}^{N_{0}}%
\sum_{Q\in\mathbb{D}_{k}}\bigg(|Q|^{\frac{1}{n}-\frac{1}{2}}\int%
_{Q}|f|\bigg)^{2}\right\}  ^{\frac{1}{2}}+\varepsilon.
\]
Therefore, by \eqref{Compactness3+},
\begin{equation}
\Vert f\Vert_{SR_{\frac{2n}{n+2},2}(Q_{0})}\leq\left\{  \sum_{k=0}^{N_{0}}%
\sum_{Q\in\mathbb{D}_{k}}\bigg(|Q|^{\frac{1}{n}-\frac{1}{2}}\int%
_{Q}|f|\bigg)^{2}\right\}  ^{\frac{1}{2}}+\varepsilon.\label{Compactness7+}%
\end{equation}

Let $\mathbb{D}_{\geq N_{0};Q_{0}} = \mathbb{D}_{\geq N_0 }=\cup_{k=0}^{N_{0}}\mathbb{D}_{k}$, then
the cardinality of $\mathbb{D}_{\geq N_{0}}$ is $L:=\frac{2^{n(N_{0}+1)}%
-1}{2^{n}-1}$. Consider the linear operator
\[
T:f\in X(Q_{0})\mapsto\bigg(|Q|^{\frac{1}{n}-\frac{1}{2}}\int_{Q}%
f\bigg)_{Q\in\mathbb{D}_{\geq N_{0}}}\in\ell_{2}^{L}.
\]
It is easy to see that $T$ is well-defined: If $f\in X(Q_0)$ (and hence $f \geq 0$) then
\begin{align*}
\Vert Tf\Vert_{\ell_{2}^{L}} &  = \left\{  \sum_{k=0}^{N_{0}}\sum
_{Q\in\mathbb{D}_{k}}\bigg(|Q|^{\frac{1}{n}-\frac{1}{2}}\int%
_{Q}|f|\bigg)^{2}\right\}  ^{\frac{1}{2}}\\
&  =|Q_{0}|^{\frac{1}{n}-\frac{1}{2}}\,\left\{  \sum_{k=0}^{N_{0}}%
2^{-kn(\frac{1}{n}-\frac{1}{2})2}\sum_{Q\in\mathbb{D}_{k}}\bigg(\int%
_{Q}|f|\bigg)^{2}\right\}  ^{\frac{1}{2}}\\
&  \leq|Q_{0}|^{\frac{1}{n}-\frac{1}{2}}\,\left\{  \sum_{k=0}^{N_{0}%
}2^{-kn(\frac{1}{n}-\frac{1}{2})2}\bigg(\sum_{Q\in\mathbb{D}_{k}}%
\int_{Q}|f|\bigg)^{2}\right\}  ^{\frac{1}{2}}\\
&  =|Q_{0}|^{\frac{1}{n}-\frac{1}{2}}\,\left\{  \sum_{k=0}^{N_{0}}%
2^{-kn(\frac{1}{n}-\frac{1}{2})2}\right\}  ^{\frac{1}{2}}\,\Vert f\Vert_{L^{1}(Q_{0})}\\
&  \lesssim2^{N_{0}(\frac{n}{2}-1)}\,\Vert f\Vert_{X(Q_{0})}.
\end{align*}
Furthermore, $T$ is compact, since it is a finite rank operator. We can
equivalently rewrite \eqref{Compactness7+} in terms of $T$ as follows, for
every $f\in X(Q_{0}),\,\Vert f\Vert_{X(Q_{0})}\leq1$,
\begin{equation}
\Vert f\Vert_{SR_{\frac{2n}{n+2},2}(Q_{0})}\leq\Vert Tf\Vert_{\ell_{2}^{L}%
}+\varepsilon.\label{Compactness8+}%
\end{equation}

Let $\{f_{l}\}_{l\in \mathbb{N}}$ be a bounded sequence in $X(Q_{0})$ (without loss we
may assume that for all $l$, $\Vert f_{l}\Vert_{X}\leq\frac{1}{2}$). The
compactness of $T:X(Q_{0})\rightarrow\ell_{2}^{L}$ guarantees (modulo passing
to a subsequence) that $\{Tf_{l}\}_{l \in \mathbb{N}}$ is convergent in $\ell_{2}^{L}$.
Accordingly, there exists $l_{0}$ such that
\[
\Vert Tf_{l}-Tf_{l^{\prime}}\Vert_{\ell_{2}^{L}}\leq\varepsilon,\qquad
\text{if}\qquad l,l^{\prime}\geq l_{0}.
\]
Therefore, by \eqref{Compactness8+},
\[
\Vert f_{l}-f_{l^{\prime}}\Vert_{SR_{\frac{2n}{n+2},2}(Q_{0})}\leq
2\varepsilon.
\]
Consequently, from \eqref{SobIde2} we
see that $\{f_{l}\}_{l \in \mathbb{N}}$ is a Cauchy sequence in $H^{-1}$.

Next we show the converse statement, i.e., if\footnote{Recall that $X \subset L^1_{\text{loc}, +}(\R^n)$, cf. \eqref{SLat}.}  $X_{c}\overset{compactly}{\hookrightarrow}\mathcal{H}_{\text{loc}}^{-1}(\R^n)$  then
\begin{equation}\label{Target}
	\lim_{N \to \infty} s_N(X) =0. 
\end{equation}
By assumption $U_{X(Q_0)}$, the closure of the unit ball of $X(Q_0)$, is a  compact set in $\mathcal{H}^{-1}(\R^n)$. In particular, for any $\delta > 0$ there exist $f_1, \ldots, f_L \in U_{X(Q_0)}$ such that
$$
	U_{X(Q_0)} \subset\bigcup_{l=1}^{L}B\Big(f_{l},\frac{\delta}{2}\Big),
$$
where $B(f_{l},\delta/2)$ denotes the ball in $\mathcal{H}^{-1}$ centered at $f_l$ and radius $\delta/2$. Hence, for any $f \in X(Q_0), \|f\|_{X(Q_0)} \leq 1$, there exists $l \in \{1, \ldots, L\}$ such that
$$
	\|f-f_l\|_{\mathcal{H}^{-1}(\R^n)} < \frac{\delta}{2}. 
$$
As a consequence (cf. \eqref{SobIde})
\begin{align}
	s_N(f) &\leq s_N(f-f_l) + s_N(f_l)  \leq s_1 (f-f_l) + s_N(f_l) \nonumber \\
	& \lesssim \|f-f_l\|_{\mathcal{H}^{-1}(\R^n)} + s_N(f_l)  < \frac{\delta}{2} + \sup_{l \in \{1, \ldots, L\}}s_N(f_l). \label{422} 
\end{align}

Assume momentarily that
\begin{equation}\label{ClaimSob}
	\lim_{N \to \infty} s_N(\omega) = 0, \qquad \text{for every} \qquad \omega \in \mathcal{H}^{-1}(\R^n). 
\end{equation}
In particular,  we have, for $N$ sufficiently large depending only on $\delta$, 
$$
	 \sup_{l \in \{1, \ldots, L\}} s_N(f_l) \leq \frac{\delta}{2}.
$$
Inserting this estimate into \eqref{422}, we conclude that \eqref{Target} holds.

To complete the proof, it remains to show \eqref{ClaimSob}: 
Fix $\chi\in C^{\infty}(\mathbb{R}^{n})$ with\footnote{One may think that
$\chi(x) = e^{-\frac{1}{1-|x|^{2}}} \mathbf{1}_{B(0, 1)}(x)$.} $\text{supp }
\chi\subset B(0, 1)$ and $\chi\geq0$. For every dyadic cube $Q_{j m}
\in\mathbb{D}_{j}$, we let\footnote{$\chi_{j m}(f)$ should be adequately
interpreted in the distributional sense.}
\begin{equation}
\label{Chijm}\chi_{j m}(f) := (\chi_{j m}, f) = \int_{\mathbb{R}^{n}} \chi_{j
m}(x) f(x) \, dx, \qquad m \in\mathbb{Z}^{n},
\end{equation}
where $\chi_{j m}(x) := 2^{j n} \chi(2^{j} x -m)$. Without loss of generality,
we may assume that $\text{supp } \chi_{j m} \subset d Q_{j m}$ for a fixed
constant $d > 1$ and
\begin{equation}
\label{11e}\inf_{x \in Q_{j m}} \chi_{j m}(x) \gtrsim2^{j n}.
\end{equation}
 Then,
we have
\begin{align}
s_{N}(\omega) &  =\sup_{\mathcal{Q}\in
S(\R^n)}\left[\sum_{k=N-1}^{\infty}\sum_{i\in I:Q_{i}\in\mathbb{D}_{k}(\mathcal{Q)}%
}\Big(|Q_{i}|^{\frac{1}{n}-\frac{1}{2}}\int_{Q_{i}}d\omega\Big)^{2}%
\right]^{\frac{1}{2}}\nonumber\\
&  \leq\left[\sum_{k=N-1}^{\infty}\sum_{Q\in\mathbb{D}_{k}}\Big(|Q|^{\frac
{1}{n}-\frac{1}{2}}\int_{Q}d\omega\Big)^{2}\right]^{\frac{1}{2}}\nonumber\\
&  =\left[\sum_{k=N-1}^{\infty}\sum_{Q\in\mathbb{D}_{k}}\Big(2^{-kn(\frac
{1}{n}+\frac{1}{2})}\int_{Q}2^{kn}\,d\omega\Big)^{2}\right]^{\frac{1}{2}%
}\nonumber\\
&  \leq\left[\sum_{k=N-1}^{\infty}2^{k(-1-\frac{n}{2})2}\sum_{Q\in
\mathbb{D}_{k}}\Big(\int_{Q}\chi_{Q}(x)\,d\omega\Big)^{2}\right]^{\frac{1}{2}%
}.\label{0.21}%
\end{align}
Note that the last step is true because both $\chi_{Q}\geq0$ and $\omega\geq
0$. Furthermore, using well-known estimates of function spaces in terms of local means  (see e.g. \cite[Theorem 1.15]{Triebel}), we get
\[
\left[\sum_{k=0}^{\infty}2^{k(-1-\frac{n}{2})2}\sum_{Q\in\mathbb{D}_{k}%
}\Big(\int_{Q}\chi_{Q}(x)\,d\omega\Big)^{2}\right]^{\frac{1}{2}}\lesssim
\Vert\omega\Vert_{H^{-1}(\R^n)}.
\]
In particular, this implies
\[
\lim_{N\rightarrow\infty}\left[\sum_{k=N-1}^{\infty}2^{k(-1-\frac{n}{2})2}%
\sum_{Q\in\mathbb{D}_{k}}\Big(\int_{Q}\chi_{Q}(x)\,d\omega\Big)^{2}%
\right]^{\frac{1}{2}}=0
\]
provided that $\omega\in H^{-1}(\R^n)\cap BM^{+}_c$ and (cf. \eqref{0.21})
\[
\lim_{N\rightarrow\infty} s_{N}(\omega)=0.
\]
This shows the desired result \eqref{ClaimSob}. \qed

\section{Computability of sparse indices\label{Section4.2}}

In this section we compute the sparse indices for familar scales of spaces.


\begin{proposition}
[Sparse indices for $L^{p}$]\label{ThmSparseLebesgue} Let $n\geq2$ and 
$p>\frac{2n}{n+2}$.  Then, for every $N\in\mathbb{N}$,
\begin{equation}
s_{N}(L^p)\lesssim 2^{-Nn(\frac{2+n}{2n}-\frac{1}{\min\{2,p\}})}. \label{sNLp}%
\end{equation}
\end{proposition}


\begin{proof}
 We can estimate $s_{N}(f)$ (cf. (\ref{spinf})) as
follows: Let $\mathcal{Q} = (Q_i)_{i \in I} \in S(Q_0)$, by H\"{o}lder's inequality, we have
\begin{align}
\sum_{Q \in \mathbb{D}_{\leq N-1} (\mathcal{Q})}\bigg(|Q|^{\frac{1}%
{n}-\frac{1}{2}}\int_{Q}|f|\bigg)^{2}  
& =  \sum_{k=N-1}^{\infty}\sum_{i\in I:Q_{i}\in\mathbb{D}_{k}(\mathcal{Q}%
)}  \bigg(|Q_i|^{\frac{1}%
{n}-\frac{1}{2}}\int_{Q_i}|f|\bigg)^{2} \nonumber  \\
& \leq \sum_{k=N-1}^{\infty}\sum_{i\in I:Q_{i}\in\mathbb{D}_{k}(\mathcal{Q}%
)}   |Q_{i}|^{\frac{2}{n}-\frac{2}{p}+1}\bigg(\int_{Q_{i}}|f|^{p}\bigg)^{\frac
{2}{p}}\nonumber\\
&  =|Q_{0}|^{(\frac{2+n}{2n}-\frac{1}{p})2}\,\sum_{k=N-1}^{\infty}%
2^{-kn(\frac{2+n}{2n}-\frac{1}{p})2}\sum_{i\in I:Q_{i}\in\mathbb{D}_{k}%
}\bigg(\int_{Q_{i}}|f|^{p}\bigg)^{\frac{2}{p}}.\label{4.10}%
\end{align}

We distinguish two possible cases. First, assume that $p \leq2$. Then
\begin{align}
\sum_{k=N-1}^{\infty}2^{-kn(\frac{2+n}{2n}-\frac{1}{p})2}\sum_{i\in I:Q_{i}%
\in\mathbb{D}_{k}}\bigg(\int_{Q_{i}}|f|^{p}\bigg)^{\frac{2}{p}}  &
\leq\sum_{k=N-1}^{\infty}2^{-kn(\frac{2+n}{2n}-\frac{1}{p})2} \bigg(\sum_{i\in
I:Q_{i}\in\mathbb{D}_{k}} \int_{Q_{i}}|f|^{p}\bigg)^{\frac{2}{p}%
}\nonumber\\
&  \leq\|f\|_{L^{p}(Q_{0})}^{2} \, \sum_{k=N-1}^{\infty
}2^{-kn(\frac{2+n}{2n}-\frac{1}{p})2}\nonumber\\
&   \leq c_{n} \, \bigg(\frac{2+n}{2n}-\frac{1}{p} \bigg)^{-1} \,
2^{-N n (\frac{2 + n}{2 n} -\frac{1}{p}) 2} \, \|f\|_{L^{p}(Q_{0})}^{2}.
\label{4.11}%
\end{align}
On the other hand, if $p > 2$ then, by H\"older's inequality,
\begin{align}
\sum_{k=N-1}^{\infty}2^{-kn(\frac{2+n}{2n}-\frac{1}{p})2}\sum_{i\in I:Q_{i}%
\in\mathbb{D}_{k}}\bigg(\int_{Q_{i}}|f|^{p}\bigg)^{\frac{2}{p}}  &
\leq\sum_{k=N-1}^{\infty}2^{-kn(\frac{2+n}{2n}-\frac{1}{2})2} \,
\bigg( \sum_{i\in I:Q_{i}\in\mathbb{D}_{k}} \int_{Q_{i}}|f|^{p}
\bigg)^{\frac{2}{p}}\nonumber\\
& \hspace{-5cm}  \leq\|f\|^{2}_{L^{p}(Q_{0})} \, \sum_{k=N-1}^{\infty
}2^{-kn(\frac{2+n}{2n}-\frac{1}{2})2}  \leq c_{n} \, 2^{-N n (\frac{2+n}{2 n} -\frac{1}{2}) 2} \,
\|f\|_{L^{p}(Q_{0})}^{2}. \label{4.12}%
\end{align}

Combining \eqref{4.10}--\eqref{4.12} (and noting that all estimates are uniform with respect to $\mathcal{Q}$) we obtain  
$$
	s_N(f) \lesssim  2^{-Nn(\frac{2+n}{2n}-\frac{1}{\min\{2,p\}})} \|f\|_{L^p(Q_0)}.
$$
Taking now the supremum over all $f\in L^p(Q_0), \|f\|_{L^p(Q_0)} \leq 1$, we achieve the desired estimate \eqref{sNLp}. 
\end{proof}

%
%
%
%


\begin{proposition}[Sparse indices for $M^{p, \alpha}$] \label{ThmSparseMorrey} Let $n\geq2$. If $N$ is sufficiently large\footnote{To be more precise, $2^{(N-1)n}>|Q_{0}|$. This assumption is not restrictive since we are only interested in the asymptotic behaviour of indices. For the sake of completeness, we mention that $s_N(M^{p, \alpha})$ with $2^{(N-1)n}\leq |Q_{0}|$ can be also computed using the same ideas, but now the log-parameter $\alpha$ does not play any role.} then
\[
s_{N}(M^{p, \alpha})\lesssim \left\{
\begin{array}
[c]{ll}%
 2^{-N(\frac{2}{n}%
-\frac{1}{p})\frac{n}{2}} N^{-\frac{\alpha}{2}}, & \text{if}\qquad p>\frac{n}{2}, \quad \alpha \in \R,\\
& \\
N^{\frac{-\alpha + 1}{2}}, &
\text{if}\qquad p = \frac{n}{2}, \quad \alpha > 1.
\end{array}
\right.
\]
\end{proposition}
\begin{proof}
 By definition (cf. (\ref{aparecio}))
\[
\int_{Q}|f|\leq|Q|^{\frac{1}{p^{\prime}}}   (1-(\log |Q|)_-)^{-\alpha}\,\Vert f\Vert_{M^{p, \alpha}(Q_{0})}, 
\qquad \forall Q\in\mathcal{D}(Q_{0}).
\]
Therefore, for any $\mathcal{Q} = (Q_i)_{i \in I} \in S(Q_0)$, 
\begin{align}
\sum_{Q \in \mathbb{D}_{\leq N-1} (\mathcal{Q})}\bigg(|Q|^{\frac{1}%
{n}-\frac{1}{2}}\int_{Q}|f|\bigg)^{2}   &  =\sum_{k=N-1}^{\infty}\sum_{i\in I:Q_{i}%
\in\mathbb{D}_{k}(\mathcal{Q})}\bigg(|Q_{i}|^{\frac{1}{n}-\frac{1}{2}}%
\int_{Q_{i}}|f|\bigg)^{2}  \nonumber \\
&\hspace{-2.5cm}  \leq\Vert f\Vert_{M^{p, \alpha}(Q_{0})}\,\sum_{k=N-1}^{\infty}\sum_{i\in I:Q_{i}%
\in\mathbb{D}_{k}(\mathcal{Q})}|Q_{i}|^{\frac{2}{n}-\frac{1}{p}} (1-(\log |Q_i|)_-)^{-\alpha}   \,\int_{Q_{i}%
}|f| \nonumber \\
& \hspace{-2.5cm}  \approx \Vert f\Vert_{M^{p, \alpha}(Q_{0})} 
\,\sum_{k=N-1}^{\infty}2^{-kn(\frac{2}{n}-\frac{1}{p})} k^{-\alpha}   \sum_{i\in I:Q_{i}%
\in\mathbb{D}_{k}(\mathcal{Q})}\int_{Q_{i}}|f| \nonumber\\
&\hspace{-2.5cm}  \leq \Vert f\Vert_{M^{p, \alpha}(Q_{0})}\,\Vert f\Vert_{L^{1}(Q_{0})}\, \sum_{k=N-1}^{\infty}2^{-kn(\frac{2}{n}-\frac
{1}{p})} k^{-\alpha} \nonumber\\
& \hspace{-2.5cm}  \leq \Vert f\Vert_{M^{p, \alpha}(Q_{0})}^2  \sum_{k=N-1}^{\infty}2^{-kn(\frac{2}{n}-\frac
{1}{p})} k^{-\alpha}. \label{55}
\end{align}
Furthermore
\begin{equation}\label{56}
	 \sum_{k=N-1}^{\infty}2^{-kn(\frac{2}{n}-\frac
{1}{p})} k^{-\alpha} \approx \left\{
\begin{array}
[c]{ll}%
 2^{-N n (\frac{2}{n}%
-\frac{1}{p})} N^{-\alpha}, & \text{if}\qquad p>\frac{n}{2}, \quad \alpha \in \R,\\
& \\
N^{-\alpha + 1}, &
\text{if}\qquad p = \frac{n}{2}, \quad \alpha > 1.
\end{array}
\right.
\end{equation}
The desired result follows then from \eqref{55} and \eqref{56}.
\end{proof}

%
%
%


\begin{proposition}[Sparse indices for RMT spaces] \label{ThmSparseCongruentTadmor}
Let $n\geq2.$  If $N$ is sufficiently large then
\[
s_{N}(R_{p, 2} \log^\alpha) \lesssim \left\{
\begin{array}
[c]{ll}%
 2^{-Nn(\frac{n+2}{2n}-\frac{1}{p})} N^{-\alpha}, & \text{if}\qquad p>\frac{2n}{n+2}, \quad \alpha \in \R,\\
& \\
N^{-\alpha + \frac{1}{2}}, &
\text{if}\qquad p = \frac{2n}{n+2}, \quad \alpha > \frac{1}{2}.
\end{array}
\right.
\]

%
\end{proposition}

\begin{proof}
Let $f\in R_{p,2}\log^{\alpha}(Q_{0})$ and $\mathcal{Q} = (Q_i)_{i \in I} \in S(Q_0)$. Since $\mathbb{D}_{k; Q_0}(\mathcal{Q})\subset\mathbb{D}%
_{k;Q_0} \subset \Pi(Q_0)$, we get
\begin{align*}
\sum_{Q \in \mathbb{D}_{\leq N-1} (\mathcal{Q})}\bigg(|Q|^{\frac{1}%
{n}-\frac{1}{2}}\int_{Q}|f|\bigg)^{2}  &  =\sum_{k=N-1}^{\infty}\sum_{i\in I:Q_{i}%
\in\mathbb{D}_{k}(\mathcal{Q})}\bigg(|Q_{i}|^{\frac{1}{n}-\frac{1}{2}}%
\int_{Q_{i}}|f|\bigg)^{2}\\
& \hspace{-3cm}  \lesssim\sum_{k=N-1}^{\infty}2^{-kn2(\frac{n+2}{2n}-\frac{1}{p}%
)}\,(1+k)^{-2\alpha}\,\sum_{Q\in\mathbb{D}_{k}}\bigg(\frac
{\big(1-(\log|Q|)_{-}\big)^{\alpha}}{|Q|^{\frac{1}{p^{\prime}}}}\,\int%
_{Q}|f|\bigg)^{2}\\
&\hspace{-3cm}   \leq\Vert f\Vert_{R_{p,2}\log^{\alpha}(Q_{0})}^{2}\,\sum_{k=N-1}^{\infty
}2^{-kn2(\frac{n+2}{2n}-\frac{1}{p})}\,(1+k)^{-2\alpha}.
\end{align*}
Since the previous estimates are uniform with respect to $\mathcal{Q}$, we derive
\begin{equation}
s_{N}(f)^{2}\lesssim\mathcal{I}_{N}\,\Vert f\Vert_{R_{p,2}%
\log^{\alpha}(Q_{0})}^{2}\label{4.15}%
\end{equation}
where
\[
\mathcal{I}_{N}:=\sum_{k=N-1}^{\infty}2^{-kn2(\frac{n+2}{2n}-\frac{1}{p}%
)}\,(1+k)^{-2\alpha}.
\]

Observe that
\begin{equation}\label{4.16}
	\mathcal{I}_N \approx \left\{
\begin{array}
[c]{ll}%
 2^{-Nn2(\frac{n+2}{2n}-\frac{1}{p})}\,N^{-2\alpha}, & \text{if}\qquad p>\frac{2n}{n+2}, \quad \alpha \in \R,\\
& \\
N^{-2 \alpha + 1}, &
\text{if}\qquad p = \frac{2n}{n+2}, \quad \alpha > \frac{1}{2}.
\end{array}
\right.
\end{equation}
Plugging \eqref{4.16} into \eqref{4.15} and taking the supremum over all $f \in R_{p, 2} \log^\alpha, \|f\|_{R_{p, 2} \log^\alpha} \leq 1$, we arrive at the desired estimate for $s_N(R_{p, 2} \log^\alpha)$. 
\end{proof}

\begin{remark}
	The proof of Proposition \ref{ThmSparseCongruentTadmor} gives a slightly stronger result using the refined class $CR_{p, q} \log^\alpha$. These spaces are defined using
congruent cubes by the condition\footnote{Note that $CR_{p,q}\log^{\alpha
}(Q_{0})$ can be equivalently introduced as the set of all $f\in L^{1}(Q_{0})$
such that, for every $k\geq0$,
\[
\bigg\{\sum_{Q_{i}\in\mathbb{D}_{k;Q_{0}}}\bigg(\int_{Q_{i}}|f|\bigg)^{q}%
\bigg\}^{\frac{1}{q}}\lesssim\left\{
\begin{array}
[c]{ll}%
2^{\frac{kn}{p^{\prime}}}\,(1+k)^{-\alpha}, & \text{if}\qquad2^{k}\geq
\ell(Q_{0}),\\
& \\
2^{\frac{kn}{p^{\prime}}}, & \text{if}\qquad2^{k}\leq\ell(Q_{0}).
\end{array}
\right.
\]
}
\[
\Vert f\Vert_{CR_{p,q}\log^{\alpha}(Q_{0})}:=\sup_{k\in\mathbb{N}_{0}}%
\sup_{(Q_{i})_{i\in I}\in\mathbb{D}_{k;Q_{0}}}\left\{\sum_{i\in I}%
\bigg[\frac{\big(1-(\log|Q_{i}|)_{-}\big)^{\alpha}}{|Q_{i}|^{\frac
{1}{p^{\prime}}}}\,\int_{Q_{i}}|f|\bigg]^{q}\right\}^{\frac{1}{q}}<\infty.
\]
Clearly $\|f\|_{CR_{p,q}\log^{\alpha}(Q_{0})} \leq \|f\|_{R_{p,q}\log^{\alpha}(Q_{0})}$ and hence $$s_N(R_{p,q}\log^{\alpha}(Q_{0})) \leq s_N(CR_{p,q}\log^{\alpha}(Q_{0})).$$
Then Proposition \ref{ThmSparseCongruentTadmor} with $CR_{p, 2} \log^\alpha$ also holds.  
\end{remark}

\section{Proof of Theorem \ref{teo:exhaust}\label{sec:construcciones}}

(ii) $\implies$ (i): Assume that $\{u^\varepsilon\}_{\varepsilon > 0}$  is sparse stable, i.e., $\{\omega^\varepsilon\}_{\varepsilon > 0}$ is uniformly bounded in $S_\Psi(\R^n)$ for some decay $\Psi$.  It follows from \eqref{116} that $\{\omega^\varepsilon\}_{\varepsilon > 0}$ is a precompact set in $H^{-1}_{\text{loc}}(\R^n)$, i.e., $\{u^\varepsilon\}_{\varepsilon > 0}$ is $H^{-1}$-stable. 

(i) $\implies$ (ii): Define
\begin{equation}
\Psi(N):=\sup_{\varepsilon>0}s_{N}(\omega^{\varepsilon
}).\label{PsiDef}%
\end{equation}
It is clear that $\Psi$ is decreasing. Furthermore $
s_{N}(\omega^{\varepsilon})\leq\Psi(N),$
which yields that for all $\varepsilon>0$,
\[
\Vert\omega^{\varepsilon}\Vert_{S_{\Psi}(\R^n)}=\sup_{N\in\mathbb{N}%
}\frac{s_{N}(\omega^{\varepsilon})}{\Psi(N)}\leq1
\]
i.e., $\{\omega^{\varepsilon}\}_{\varepsilon > 0}$ is bounded in $S_{\Psi}(\R^n)$.

It remains to show that $\Psi(N)\rightarrow0$ as $N\rightarrow\infty$. The proof follows the same line of ideas as that of  \eqref{Target}. On
account of (i), the set $W=\overline{\{\omega^{\varepsilon}\}}_{\varepsilon > 0} \subset BM^+_c$ is compact (in
$\mathcal{H}^{-1}(\R^n)$). In particular, for any $\delta>0$ there exists
$\omega_{1},\ldots,\omega_{L}\in W$ such that
\[
W\subset\bigcup_{l=1}^{L}B\Big(\omega_{l},\frac{\delta}{2}\Big).
\]
 As a
byproduct, for any $\varepsilon>0$ one can find $l\in\{1,\ldots,L\}$ such
that
\begin{equation}
\Vert \omega^{\varepsilon}-\omega_{l} \Vert_{\mathcal{H}^{-1}(\R^n)}<\frac{\delta}%
{2}.\label{0.20new}%
\end{equation}
Therefore
\begin{equation*}
	s_N(\omega^\varepsilon) \leq s_N(\omega^\varepsilon-\omega_l) + s_N(\omega_l)  \lesssim \|\omega^\varepsilon-\omega_l\|_{\mathcal{H}^{-1}(\R^n)} +  s_N(\omega_l),
\end{equation*}
where the hidden equivalence constant is independent of $\varepsilon$. As a consequence (cf. \eqref{PsiDef} and \eqref{0.20new})
$$
	\Psi(N) \lesssim \frac{\delta}{2} + \sup_{l \in \{1, \ldots, L\}} s_N(\omega_l).
$$
Then, by  \eqref{ClaimSob},  $\lim_{N \to \infty} \Psi(N) =0$. 	\qed

\section{Sharpening Morrey regularity of DiPerna--Majda via $V_\Psi$}\label{SectionDMaj}

As already mentioned in Section \ref{Section1Intro}, a famous $2$D result  due to DiPerna and Majda \cite{DiPernaMajda} asserts strong convergence of approximate solutions with initial vortex sheet satisfying Morrey regularity $M^{1, \alpha}(\R^2)$ with $\alpha > 1$. This can be obtained from the compactness assertion  (cf. \eqref{comp})
\begin{equation}\label{71}
M^{\frac{n}{2},\alpha}_c(\R^n)\overset{compactly}{\hookrightarrow}H^{-1}_{\text{loc}}(\R^n
), \qquad \alpha > 1, \quad n \geq 2.
\end{equation}
The goal of this section is to show that these results can be further improved using sparse techniques developed in previous sections,  together with extrapolation techniques.

\subsection{$V_\Psi$-spaces} Given a decay function $\Psi,$ in this section we construct a new Besov-type
space $V_{\Psi}$, whose sparce indices are controlled by $\Psi.$

\begin{definition}\label{DefV}
Let $V_{\Psi}(\mathbb{R}^{n})$ be the space
formed by all tempered distributions $f\in\mathcal{S}^{\prime}(\mathbb{R}^{n})$ such that\footnote{As usual, $\{\Delta_j\}_{j \in \mathbb{N}_0}$ refers to  standard (inhomogeneous) Littlewood--Paley operators on $\R^n$. }
\[
\Vert f\Vert_{V_{\Psi}(\mathbb{R}^{n})}:=\sup_{N\in\mathbb{N}_{0}}\frac
{1}{\Psi(N)^{2}}\sum_{j=N}^{\infty}2^{-2j}\Vert\Delta_{j}f\Vert_{L^{\infty
}(\mathbb{R}^{n})} < \infty.
\]
Let
$
	V^+_\Psi(\R^n) = V_\Psi(\R^n) \cap BM^+_c.
$
\end{definition}

\begin{remark}
\label{Remark1} The construction of the space $V_{\Psi}$ is in some sense
``dual" to the one used to define the classical Vishik space\footnote{cf.
\cite{DMEuler}.} $B_{\Gamma}$, where $\Gamma:[0,\infty)\rightarrow(0,\infty)$
is an increasing function with $\lim_{t \to \infty} \Gamma(t) = \infty$ (\textquotedblleft a growth
function\textquotedblright) and the norm of $B_{\Gamma}$ is given by
\begin{equation}
\Vert f\Vert_{B_{\Gamma}(\mathbb{R}^{n})}:=\sup_{N\in\mathbb{N}_{0}}\frac
{1}{\Gamma(N)}\sum_{j=0}^{N}\Vert\Delta_{j}f\Vert_{L^{\infty}(\mathbb{R}^{n}%
)}.\label{VishikSpace}%
\end{equation}

Note that we could have elements $f\in$ $B_{\Gamma}(\R^n),$ such that $\sum
_{j=0}^{\infty}\Vert\Delta_{j}f\Vert_{L^{\infty}(\mathbb{R}^{n})}=\infty$ (i.e., $f$ does not belong to the Besov space\footnote{Recall that the Besov spaces $B^s_{p, q}(\R^n), s \in \R, p, q \in [1, \infty]$, are endowed with the norm:
$$
	\|f\|_{B^s_{p, q}(\R^n)} = \bigg(\sum_{j=0}^\infty 2^{j s q} \|\Delta_j f\|^q_{L^p(\R^n)} \bigg)^{\frac{1}{q}}.
$$
See e.g. \cite{S70}, \cite{BS88}, \cite{Triebel}.
}  $B_{\infty,1}^{0}(\mathbb{R}^{n})$),  as long as the growth of
the corresponding partial sums is controlled by $\Gamma$. On the other hand,
$V_{\Psi}(\mathbb{R}^{n})$ is formed by elements $f\in B_{\infty,1}%
^{-2}(\mathbb{R}^{n})$, the classical Besov space of negative order, equipped
with
\begin{equation}
\Vert f\Vert_{B_{\infty,1}^{-2}(\mathbb{R}^{n})}=\sum_{j=0}^{\infty}%
2^{-2j}\Vert\Delta_{j}f\Vert_{L^{\infty}(\mathbb{R}^{n})},\label{1}%
\end{equation}
such that the remainder of the corresponding series in \eqref{1} has a
prescribed decay given by $\Psi(N)^{2}$. The connection of these spaces
becomes apparent through the use of stream functions. Let $\omega\in V_{\Psi
}(\R^n),$ and let $\psi$ be a stream function, i.e. $\Delta\psi=\omega.$ Using
Fourier multipliers one can show that
\[
\omega\in V_{\Psi}(\R^n) \iff\sup_{N\in\mathbb{N}_{0}}\frac{1}{\Psi(N)^{2}}\sum
_{j=N}^{\infty}\Vert\Delta_{j}\psi\Vert_{L^{\infty}(\mathbb{R}^{n})}<\infty.
\]
In other words, the space $V_{\Psi}(\R^n)$ is formed by vorticities $\omega$ with
corresponding stream functions $\psi$ satisfying the \textquotedblleft dual" of the 
Vishik condition  \eqref{VishikSpace}.
\end{remark}

\subsection{$V_\Psi$-regularity of Euler flows}  In this section, we restrict ourselves to the following sufficiently rich class of decays.

\begin{definition}[Admissible/doubling decays]\label{DefAdm}
Let $\Psi$ be a decay. We say that $\Psi$ is:
	\begin{enumerate}
	\item[(i)] \emph{admissible}\footnote{This is a very weak assumption on the monotonicity properties of $\Psi$.  Basic examples of admissible decays are $\Psi(t) = t^{-\lambda}, \Psi(t) = (\log t)^{-\lambda}$, where $\lambda > 0$, (or more generally, concatenations of logarithms) and their products. Exponential decays $\Psi(t) = 2^{-C t}, \, C \geq 1$, are excluded. However, this is not restrictive since $V_{2^{-C t}}(\R^n) \hookrightarrow V_{t^{-\lambda}}(\R^n)$.}  provided that
	$$
\sum_{r=0}^{N}(2^{r}\Psi(r))^{2}\lesssim(2^{N}\Psi(N))^{2}, \qquad N \in \mathbb{N}_0.
$$
\item[(ii)] \emph{doubling} provided that $\Psi(ct)\gtrsim\Psi(t),$ for some $c>1$.
\end{enumerate}
\end{definition}

We are now ready to state  the main result of this section. 

\begin{theorem}\label{ThmVPsi} 
Let $\Psi$ be an admissible doubling decay.  
Then:
	\begin{enumerate}
		\item[(i)]  $V^+_\Psi(\R^n)_c \hookrightarrow S_\Psi(\R^n)_{c}$. As a consequence (cf. \eqref{116}), $V^+_\Psi(\R^n)_c\overset{compactly}{\hookrightarrow}H^{-1}_{\emph{loc}}(\R^n
)$. 
		\item[(ii)] Let $\{u^\varepsilon\}_{\varepsilon > 0}$ be a family of approximate solutions to Euler equations \eqref{Euler}, such that the related family of vorticities $\{\omega^\varepsilon\}_{\varepsilon > 0}$ is uniformly bounded in $L^\infty([0, T]; V^+_\Psi(\R^n)_c)$. Then $\{u^\varepsilon\}_{\varepsilon > 0}$ has a strong limit $u$ in $L^\infty([0, T]; L^2_{\emph{loc}}(\R^n))$, where $u$ is a solution with no concentrations. 
	\end{enumerate}
\end{theorem}

Specialising the previous result with $\Psi(t) = t^{\frac{1-\alpha}{2}}$, we are able to improve\footnote{We only focus on the most interesting case $p = n/2$, but similar improvements can also be obtained in the non-critical regime $p > n/2$.} \eqref{71} in the following sense. 

\begin{theorem}\label{Thm11}
Assume that $\alpha > 1$. Then $$M^{\frac{n}{2}, \alpha}(\R^n) \hookrightarrow V_\Psi(\R^n).$$ Furthermore, this embedding is strict in the sense that $M^{\frac{n}{2}, \alpha}(\R^n) \neq V_\Psi(\R^n)$.
\end{theorem}

The proofs of these results are based on extrapolation methods developed in the following section. 

\subsection{Extrapolation characterization of $V_\Psi$}

Let $(A_0, A_1)$ be an interpolation pair\footnote{Loosely speaking, $A_0+A_1$ makes sense.} of Banach spaces. Recall that the \emph{$K$-functional} relative to $(A_0, A_1)$ is defined by
$$
	K(t, f; A_0, A_1) = \|f\|_{A_0 + t A_1} = \inf_{f = f_0 + f_1} (\|f_0\|_{A_0} + t \|f_1\|_{A_1})
$$
for $t > 0$ and $f \in A_0 + A_1$.

\begin{theorem}
\label{PropInterpol} Suppose that  $\Psi$ is an admissible doubling decay. Then
\begin{equation}\label{LI}
\Vert f\Vert_{V_{\Psi}(\mathbb{R}^{n})}\approx\sup_{t\in(0,1)}\frac
{K(t,f;B_{\infty,1}^{-2}(\mathbb{R}^{n}),B_{\infty,1}^{0}(\mathbb{R}^{n}%
))}{\Psi(-\log t)^{2}}.
\end{equation}

\end{theorem}

\begin{remark}
This result shows that the $V_{\Psi}$ spaces can be described as extrapolation
spaces\footnote{with respect to the so-called $\Sigma$-method of extrapolation (cf.
\cite{JM91}).} relative to the
classical Besov pair $(B_{\infty,1}^{-2},B_{\infty,1}^{0})$, a fact that will
be very useful later, since it enables the transfer of fundamental properties
of the classical Besov spaces to $V_{\Psi}$.
\end{remark}

\begin{remark}
	The assumption $\Psi$ is doubling is necessary in order to ensure that the right-hand side of \eqref{LI} is non trivial. For instance, for the admissible decay $\Psi(t) = 2^{-C t}, \, C \in (\frac{1}{2},1)$,  $$\sup_{t\in(0,1)}\frac
{K(t,f;B_{\infty,1}^{-2}(\mathbb{R}^{n}),B_{\infty,1}^{0}(\mathbb{R}^{n}%
))}{t^{2C}} < \infty \iff f = 0.$$
\end{remark}

\begin{proof}[Proof of Theorem \ref{PropInterpol}]
We use the retraction method of interpolation theory. Recall that $\ell
_{1}^{s}(L^{\infty}(\mathbb{R}^{n}))$, $\,s\in\mathbb{R}$, is the
vector-valued sequence space equipped with the norm
\begin{equation*}
\Vert\{f_{j}\}_{j\in\mathbb{N}_{0}}\Vert_{\ell_{1}^{s}(L^{\infty}%
(\mathbb{R}^{n}))}=\sum_{j=0}^{\infty}2^{js}\Vert f_{j}\Vert_{L^{\infty
}(\mathbb{R}^{n})}. 
\end{equation*}
It is well known (cf. \cite[Appendix A1]{DMEuler}) that $B_{\infty,1}%
^{s}(\mathbb{R}^{n})$ is a retract of $\ell_{1}%
^{s}(L^{\infty}(\mathbb{R}^{n}))$ via
\[
f\mapsto\{\Delta_{j}f\}_{j\in\mathbb{N}_{0}}.
\]
As a consequence,
\[
K(t,f;B_{\infty,1}^{-2}(\mathbb{R}^{n}),B_{\infty,1}^{s}(\mathbb{R}%
^{n}))\approx K(t,\{\Delta_{j}f\}_{j\in\mathbb{N}_{0}};\ell_{1}^{-2}%
(L^{\infty}(\mathbb{R}^{n})),\ell_{1}^{s}(L^{\infty}(\mathbb{R}^{n}))).
\]
Furthermore, a well-known estimate for $K$-functionals asserts
\[
K(t,\{f_{j}\}_{j\in\mathbb{N}_{0}};\ell_{1}^{-2}(L^{\infty}(\mathbb{R}%
^{n})),\ell_{1}^{s}(L^{\infty}(\mathbb{R}^{n})))\approx\sum_{j=0}^{\infty
}2^{js}\min\{2^{(-2-s)j},t\}\Vert f_{j}\Vert_{L^{\infty}(\mathbb{R}^{n})}.
\]
Hence (letting $s=0$)
\begin{align*}
K(2^{-2N},f;B_{\infty,1}^{-2}(\mathbb{R}^{n}),B_{\infty,1}^{0}(\mathbb{R}%
^{n}))  &  \approx\sum_{j=0}^{\infty}\min\{2^{-2j},2^{-2N}\}\Vert\Delta
_{j}f\Vert_{L^{\infty}(\mathbb{R}^{d})}\\
&  \hspace{-4cm}=2^{-2N}\sum_{j=0}^{N}\Vert\Delta_{j}f\Vert_{L^{\infty
}(\mathbb{R}^{n})}+\sum_{j=N+1}^{\infty}2^{-2j}\Vert\Delta_{j}f\Vert
_{L^{\infty}(\mathbb{R}^{n})}.
\end{align*}
This implies that
\begin{equation}
\sup_{N\in\mathbb{N}_{0}}\frac{K(2^{-2N},f;B_{\infty,1}^{-2}(\mathbb{R}%
^{n}),B_{\infty,1}^{s}(\mathbb{R}^{n}))}{\Psi(N)^{2}}\approx\mathcal{A}+\Vert
f\Vert_{V_{\Psi}(\mathbb{R}^{n})}, \label{1.2new}%
\end{equation}
where
\[
\mathcal{A}:=\sup_{N\in\mathbb{N}_{0}}\frac{2^{-2N}\sum_{j=0}^{N}\Vert
\Delta_{j}f\Vert_{L^{\infty}(\mathbb{R}^{n})}}{\Psi(N)^{2}}.
\]
Furthermore, we claim that
\begin{equation}
\mathcal{A}\lesssim\Vert f\Vert_{V_{\Psi}(\mathbb{R}^{n})}. \label{1.1}%
\end{equation}
Indeed, by admissibility of $\Psi$,
\begin{align*}
\sum_{j=0}^{N}\Vert\Delta_{j}f\Vert_{L^{\infty}(\mathbb{R}^{n})}  &
=\sum_{j=0}^{N}2^{2j}\Psi(j)^{2}\,\frac{2^{-2j}}{\Psi(j)^{2}}\,\Vert\Delta
_{j}f\Vert_{L^{\infty}(\mathbb{R}^{n})}\\
&  \leq\sup_{j\in\mathbb{N}_{0}}\frac{2^{-2j}}{\Psi(j)^{2}}\,\Vert\Delta
_{j}f\Vert_{L^{\infty}(\mathbb{R}^{n})}\sum_{j=0}^{N}2^{2j}\Psi(j)^{2}\\
&  \lesssim2^{2N}\Psi(N)^{2}\sup_{j\in\mathbb{N}_{0}}\frac{2^{-2j}}%
{\Psi(j)^{2}}\,\Vert\Delta_{j}f\Vert_{L^{\infty}(\mathbb{R}^{n})}\\
&  \leq2^{2N}\Psi(N)^{2}\Vert f\Vert_{V_{\Psi}(\mathbb{R}^{n})},
\end{align*}
which yields the desired estimate \eqref{1.1}.

Combining \eqref{1.2new} and \eqref{1.1}, we see that
\[
\sup_{N\in\mathbb{N}_{0}}\frac{K(2^{-2N},f;B_{\infty,1}^{-2}(\mathbb{R}%
^{n}),B_{\infty,1}^{s}(\mathbb{R}^{n}))}{\Psi(N)^{2}}\approx\Vert
f\Vert_{V_{\Psi}(\mathbb{R}^{n})}.
\]
Note that, by basic monotonicity properties of the expressions involved (recall that $\Psi$ is doubling),
\begin{align*}
\sup_{t\in(0,1)}\frac{K(t,f;B_{\infty,1}^{-2}(\mathbb{R}^{n}),B_{\infty,1}%
^{0}(\mathbb{R}^{n}))}{\Psi(-\log t)^{2}}  & \\
&  \hspace{-4cm}  \approx \sup_{N\in\mathbb{N}_{0}}K(2^{-2N},f;B_{\infty,1}%
^{-2}(\mathbb{R}^{n}),B_{\infty,1}^{0}(\mathbb{R}^{n}))\sup_{t\in
(2^{-2(N+1)},2^{-2N})}\frac{1}{\Psi(-\log t)^{2}}\\
&  \hspace{-4cm}\approx\sup_{N\in\mathbb{N}_{0}}\frac{K(2^{-2N},f;B_{\infty
,1}^{-2}(\mathbb{R}^{n}),B_{\infty,1}^{0}(\mathbb{R}^{n}))}{\Psi(N)^{2}}.
\end{align*}
\end{proof}

\begin{remark}
\label{Remark16} The proof above shows that $B_{\infty,1}^{0}(\mathbb{R}^{n})$
plays an auxiliary role in Theorem \ref{PropInterpol}, in the sense that the
same result obtains if we replace it by $B_{\infty,1}^{s}(\mathbb{R}^{n}),$
for any $s>-2$. More precisely, suppose that the admissibility condition given in Definition \ref{DefAdm}(i) is replaced by
\[
\sum_{r=0}^{N}2^{(2+s)r}\Psi(r)^{2}\lesssim2^{(2+s)N}\Psi(N)^{2},\qquad
N\in\mathbb{N}_{0}.
\]
Then, using the same methodology we readily see that
\[
\Vert f\Vert_{V_{\Psi}(\mathbb{R}^{n})}\approx\sup_{t\in(0,1)}\frac
{K(t,f;B_{\infty,1}^{-2}(\mathbb{R}^{n}),B_{\infty,1}^{s}(\mathbb{R}^{n}%
))}{\Psi(-\log t)^{2}}.
\]

\end{remark}

\subsection{Proof of Theorem \ref{ThmVPsi}\label{SecThm1}}

The proof relies strongly on the extrapolation
description of $V_{\Psi}$. In particular, Theorem \ref{PropInterpol} will be
applied to decompose functions in $V_{\Psi}$ in terms of wavelets (cf.
Proposition \ref{PropositionWavelet} below).
%
%

 Let $\{\Upsilon^G_{N l}: N \in \mathbb{N}_0, G \in G^N, l \in \mathbb{Z}^n\}$ be an orthonormal wavelet\footnote{We briefly recall that such bases may be constructed in a standard way from two compactly supported (Daubechies) wavelets $\psi_F \in C^1(\R)$ (\emph{father wavelet}) and $\psi_M \in C^1(\R)$ (\emph{mother wavelet}) satisfying certain moment conditions. More precisely
 $$
 	\Upsilon^G_{N l} (x) = 2^{N n/2} \prod_{r=1}^n \psi_{G_r} (2^N x_r-l_r).
 $$
 Here $G^0 = \{F, M\}^n$ and $G^N = \{F, M\}^{n*}, N \in \mathbb{N}$, where $*$ indicates that at least one of the components of $G \in G^N$ must be an $M$. The role played by
the tensor index $G \in G^{N}$ is auxiliary (note that $\text{card } G^{N}
\approx1$). To simplify the exposition, the index $G$ may be
safely removed from our computations.}  basis in $L^2(\R^n)$.

\begin{proposition}
\label{PropositionWavelet} Suppose that $\Psi$ is an admissible doubling decay. Then, $f\in V_{\Psi}(\mathbb{R}^{n})$ if and only
if
\begin{equation}
f=\sum_{N\in\mathbb{N}_{0},G\in G^{N},l\in\mathbb{Z}^{n}}\lambda_{N l}%
^{G}2^{-Nn/2}\Upsilon_{N l}^{G},\qquad\{\lambda_{N l}^{G}\}\in v_{\Psi}
\label{PropositionWavelet1}%
\end{equation}
(unconditional convergence in the sense of $\mathcal{S}^{\prime}%
(\mathbb{R}^{n})$), where
	$$
		\|\{\lambda^{G}_{N l}\}\|_{v_\Psi} := \sup_{N \in \mathbb{N}_0} \frac{1}{\Psi(N)^2} \, \sum_{k=N}^\infty 2^{- 2 k} \sup_{G \in G^k, l \in \mathbb{Z}^n} |\lambda^{G}_{k l}| < \infty. 
	$$
 The representation of $f$ is unique, and the coefficients
$\lambda_{N l}^{G}$ are determined by
\begin{equation}
\lambda_{N l}^{G}=2^{Nn/2}(f,\Upsilon_{N l}^{G}), \label{PropositionWavelet2}%
\end{equation}
and the operator
\begin{equation}
I:f\mapsto\{\lambda_{N l}^{G}\} \label{OperI}%
\end{equation}
defines an isomorphism of $V_{\Psi}(\mathbb{R}^{n})$ onto $v_{\Psi}$. Furthermore
\begin{equation}
\Vert f\Vert_{V_{\Psi}(\mathbb{R}^{n})}\approx\Vert\{\lambda_{N l}
\}\Vert_{v_{\Psi}}.\label{vanilla}%
\end{equation}
\end{proposition}

\begin{proof}
By the classical wavelet theory for Besov spaces \cite[Theorem 1.20]{Triebel}, the
operator $I$ given by \eqref{OperI} acts as an isomorphism
\begin{equation}
I:B_{\infty,1}^{-2}(\mathbb{R}^{n})\longrightarrow \ell^{-2}_1(\ell_\infty)%
\qquad\text{and}\qquad I:B_{\infty,1}^{0}(\mathbb{R}^{n})\longrightarrow \ell_1(\ell_\infty).\label{proofWav1}%
\end{equation}
As usual,  $\ell^s_1 (\ell_\infty), \, s \in \R,$ is the mixed sequence space formed by all those $\{\lambda_{N l}\}$ such that
$$
	\|\{\lambda_{N l}\}\|_{\ell^s_1 (\ell_\infty)} = \sum_{N=0}^\infty 2^{N s} \sup_{l \in \mathbb{Z}^n} |\lambda_{N l}| < \infty.
	$$
We let $\ell_1(\ell_\infty) = \ell_1^0 (\ell_\infty)$. 

It follows from \eqref{proofWav1} that
\[
K(t,f;B_{\infty,1}^{-2}(\mathbb{R}^{n}),B_{\infty,1}^{0}(\mathbb{R}%
^{n}))\approx K(t,\{\lambda_{N l}\};\ell^{-2}_1(\ell_\infty), \ell_1(\ell_\infty)).
\]
Consequently, combining with Theorem \ref{PropInterpol} yields,
\begin{equation}
\Vert f\Vert_{V_{\Psi}(\mathbb{R}^{n})}\approx\sup_{t\in(0,1)}\frac
{K(t,\{\lambda_{N l}\};\ell^{-2}_1(\ell_\infty), \ell_1(\ell_\infty))}{\Psi(-\log
t)^{2}}.\label{proofWav3}%
\end{equation}
At
this point the method of proof developed in Theorem \ref{PropInterpol} can be
applied line by line (with $\ell_{\infty}$ now playing the role played by
$L^{\infty}(\mathbb{R}^{n})$) to show that
\begin{equation}
\sup_{t\in(0,1)}\frac
{K(t,\{\lambda_{N l}\};\ell^{-2}_1(\ell_\infty), \ell_1(\ell_\infty))}{\Psi(-\log
t)^{2}}  \approx\Vert\{\lambda_{N l}\}\Vert_{v_{\Psi}%
}.\label{proofWav4}%
\end{equation}
Combining \eqref{proofWav3} and \eqref{proofWav4}, we obtain \eqref{vanilla}.

For $f\in V_{\Psi}(\mathbb{R}^{n}),$ the convergence of the wavelet expansion
\eqref{PropositionWavelet1}, as well as the uniqueness of wavelet coefficients
given by \eqref{PropositionWavelet2}, is guaranteed by the fact that $V_{\Psi
}(\mathbb{R}^{n})\hookrightarrow B_{\infty,1}^{-2}(\mathbb{R}^{n})$ (cf.
Remark \ref{Remark1}) since the corresponding assertions are valid for
classical Besov spaces.
\end{proof}

\begin{proof}
[Proof of Theorem \ref{ThmVPsi}]

(i): Without loss of generality, we may assume that
$Q_{0}=(0,1)^{n}$. Let $L\in\mathbb{N}_{0}$ and $\mathcal{Q} = (Q_i)_{i \in I}\in S(Q_{0})$. If $f \geq 0$ is compactly supported in $Q_0$, then
\begin{align*}
 \left[\sum_{Q \in \mathbb{D}_{\leq L-1} (\mathcal{Q})}\bigg(|Q|^{\frac{1}%
{n}-\frac{1}{2}}\int_{Q}|f|\bigg)^{2}\right]^{\frac{1}{2}} &  =\left[\sum_{k=L-1}^{\infty}\sum_{i\in I:Q_{i}%
\in\mathbb{D}_{k}(\mathcal{Q)}}\bigg(|Q_{i}|^{\frac{1}{n}-\frac{1}{2}}%
\int_{Q_{i}}f\bigg)^{2}\right]^{\frac{1}{2}}\\
&  \leq\Vert f\Vert_{L^{1}(Q_{0})}^{\frac{1}{2}}\,\left[\sum_{k=L-1}^{\infty
}2^{-kn(\frac{1}{n}-\frac{1}{2})2}\sup_{Q\in\mathbb{D}_{k}}\int_{Q}%
f\right]^{\frac{1}{2}}\\
&  \leq\Vert f\Vert_{L^{1}(Q_{0})}^{\frac{1}{2}}\Psi(L)\sup_{N\in
\mathbb{N}_{0}}\left[\frac{1}{\Psi(N)^{2}}\sum_{k=N-1}^{\infty}2^{k(-2+n)}%
\sup_{Q\in\mathbb{D}_{k}}\int_{Q}f\right]^{\frac{1}{2}}\\
&  \leq\Psi(L)\sup_{N\in\mathbb{N}_{0}}\frac{1}{\Psi(N)^{2}}\sum
_{k=N-1}^{\infty}2^{k(-2+n)}\sup_{Q\in\mathbb{D}_{k}}\int_{Q}f.
\end{align*}
As a by-product (cf. \eqref{psi2})
\begin{equation*}
\Vert f\Vert_{S_{\Psi}(Q_{0})}=\sup_{L\in\mathbb{N}}\frac{s_{L}(f)}{\Psi(L)}\leq\sup_{N\in\mathbb{N}_{0}%
}\frac{1}{\Psi(N)^{2}}\sum_{k=N-1}^{\infty}2^{k(-2+n)}\sup_{Q\in\mathbb{D}%
_{k}}\int_{Q}f.
\end{equation*}
Hence the desired embedding $(V^+_\Psi(\R^n))_c \hookrightarrow (S_\Psi(\R^n))_c$ follows if we show that
\begin{equation}
\sup_{N\in\mathbb{N}_{0}}\frac{1}{\Psi(N)^{2}}\sum_{k=N-1}^{\infty}%
2^{k(-2+n)}\sup_{Q\in\mathbb{D}_{k}}\int_{Q}f\lesssim\Vert f\Vert_{V_{\Psi
}(\mathbb{R}^{n})}.\label{5}%
\end{equation}

Let $f \in V_{\Psi}(\mathbb{R}^{n})$. According to Proposition
\ref{PropositionWavelet}, $f$ can be expressed as
\begin{equation}
\label{11c}f = \sum_{r=0}^{\infty}\sum_{l \in\mathbb{Z}^{n}} \lambda_{r l}
2^{-r n/2} \Upsilon_{r l},
\end{equation}
where $\lambda_{r, l}$ is given by \eqref{PropositionWavelet2} The wavelets
$\Upsilon_{r l}$ can be chosen such that
\begin{equation}
\label{11b}\text{supp } \Upsilon_{r l} \subset c Q_{r l} = c (2^{-r} l + 2^{-r} Q_0),
\end{equation}
\begin{equation}
\label{11a}|\Upsilon_{r l}(x)| \lesssim2^{r n/2}, \qquad r \in\mathbb{N}_{0},
\quad l \in\mathbb{Z}^{n},
\end{equation}
and there exists $A \in\mathbb{N}, \, A > 2$, satisfying
\begin{equation}
\label{11}\int_{\mathbb{R}^{n}} x^{\beta}\Upsilon_{r l}(x) \, dx = 0,
\qquad|\beta| < A, \qquad r \in\mathbb{N}, \qquad l \in\mathbb{Z}^{n}.
\end{equation}
%
%

We are going to compute $\chi_{j m}(f)$ given by \eqref{Chijm}. For every $j\in\mathbb{N}_{0}$, we can split $f$ as follows (cf. \eqref{11c})
\begin{equation}
f=f_{j}+f^{j}:=\sum_{r=0}^{j}\sum_{l\in\mathbb{Z}^{n}}\lambda_{rl}%
2^{-rn/2}\Upsilon_{rl}+\sum_{r=j+1}^{\infty}\sum_{l\in\mathbb{Z}^{n}}%
\lambda_{rl}2^{-rn/2}\Upsilon_{rl}. \label{11d}%
\end{equation}
Then, for $m\in\mathbb{Z}^{n}$,
\begin{equation}
\chi_{jm}(f)=\chi_{jm}(f_{j})+\chi_{jm}(f^{j}). \label{6}%
\end{equation}
Next, we estimate each of these terms separately.

We first estimate $\chi_{jm}(f_{j})$. For $r\leq j$, we let (recall that $\text{supp } \chi_{j m} \subset d Q_{j m}$)
\begin{equation}
\ell_{r}^{j}(m):=\{l\in\mathbb{Z}^{n}:dQ_{jm}\cap cQ_{rl}\neq\emptyset\}.
\label{7}%
\end{equation}
Obviously (cf. \eqref{11b})
\begin{equation}
\chi_{jm}(\Upsilon_{rl})=0\qquad\text{for}\qquad l\not \in \ell_{r}^{j}(m).
\label{8}%
\end{equation}
On the other hand, if $l\in\ell_{r}^{j}(m)$ then (noting that $|\chi_{j m}(x)| \lesssim 2^{j n}$)
\begin{equation}
|\chi_{jm}(\Upsilon_{rl})|\leq\int_{dQ_{jm}}|\chi_{jm}(x)||\Upsilon
_{rl}(x)|\,dx\lesssim2^{jn}\int_{dQ_{jm}}|\Upsilon_{rl}(x)|\,dx\lesssim
2^{rn/2}, \label{9}%
\end{equation}
where we have used the property \eqref{11a} in the last step. In light of
\eqref{8} and \eqref{9}, we derive
\begin{align}
|\chi_{jm}(f_{j})|  &  \leq\sum_{r=0}^{j}\sum_{l\in\mathbb{Z}^{n}}%
|\lambda_{rl}|2^{-rn/2}|\chi_{jm}(\Upsilon_{rl})|\nonumber\\
&  =\sum_{r=0}^{j}\sum_{l\in\ell_{r}^{j}(m)}|\lambda_{rl}|2^{-rn/2}|\chi
_{jm}(\Upsilon_{rl})|\nonumber\\
&  \lesssim\sum_{r=0}^{j}\sum_{l\in\ell_{r}^{j}(m)}|\lambda_{rl}|. \label{221}%
\end{align}
Note that since $\text{card }\ell_{r}^{j}(m)\approx1,$ if $0\leq r\leq j,$
from (\ref{221}) it follows that
\begin{equation}
|\chi_{jm}(f_{j})|\lesssim\sum_{r=0}^{j}\sup_{l\in\mathbb{Z}^{n}}|\lambda
_{rl}|. \label{10}%
\end{equation}

Next we deal with $\chi_{jm}(f^{j})$. Let $r>j$ and $l\in\ell_{r}^{j}(m)$ (cf.
\eqref{7}). Using the Taylor expansion of $\chi_{jm}$ around $2^{-r}l$ and
using the cancellation conditions \eqref{11}, one can show that
\begin{align*}
|\chi_{jm}(\Upsilon_{rl})|  &  =\bigg|\int_{\mathbb{R}^{n}}\chi_{jm}%
(x)\Upsilon_{rl}(x)\,dx\bigg|\\
&  \leq\sum_{|\gamma|=A}\sup_{x\in\mathbb{R}^{n}}|D^{\gamma}\chi_{jm}%
(x)|\int_{\mathbb{R}^{n}}|\Upsilon_{rl}(x)||x-2^{-r}l|^{A}\,dx\\
&  \lesssim2^{j(n+A)}\int_{\mathbb{R}^{n}}2^{rn/2}|\psi(2^{r}x-l)|2^{-rA}%
|2^{r}x-l|^{A}\,dx\\
&  =2^{(j-r)(n+A)}2^{rn/2}\int_{\mathbb{R}^{n}}|\psi(x)||x|^{A}\,dx\\
&  \approx2^{(j-r)(n+A)}2^{rn/2}.
\end{align*}
Using this estimate we obtain
\begin{equation}
|\chi_{jm}(f^{j})|    \leq\sum_{r=j}^{\infty}\sum_{l\in\ell_{r}^{j}%
(m)}|\lambda_{rl}|2^{-rn/2}|\chi_{jm}(\Upsilon_{rl})|  \lesssim\sum_{r=j}^{\infty}2^{(j-r)(n+A)}\sum_{l\in\ell_{r}^{j}(m)}%
|\lambda_{rl}|. \label{222}%
\end{equation}
Note that $\text{card }\ell_{r}^{j}(m)\approx2^{n(r-j)}$ if $r\geq j$. Hence,
by \eqref{222},
\begin{equation}
|\chi_{jm}(f^{j})|    \lesssim\sum_{r=j}^{\infty}2^{(j-r)(n+A)}\sup
_{l\in\mathbb{Z}^{n}}|\lambda_{rl}|\,2^{n(r-j)}  =2^{jA}\sum_{r=j}^{\infty}2^{-rA}\sup_{l\in\mathbb{Z}^{n}}|\lambda_{rl}|.
\label{16}%
\end{equation}

Putting together \eqref{6}, \eqref{10} and \eqref{16},
\begin{equation*}
|\chi_{jm}(f)|   \leq|\chi_{jm}(f_{j})|+|\chi_{jm}(f^{j})|  \lesssim\sum_{r=0}^{j}\sup_{l\in\mathbb{Z}^{n}}|\lambda_{rl}|+2^{jA}%
\sum_{r=j}^{\infty}2^{-rA}\sup_{l\in\mathbb{Z}^{n}}|\lambda_{rl}|
\end{equation*}
uniformly with respect to $m\in\mathbb{Z}^{n}$. Consequently,
\begin{equation}
\sup_{m\in\mathbb{Z}^{n}}|\chi_{jm}(f)|\lesssim\sum_{r=0}^{j}\sup
_{l\in\mathbb{Z}^{n}}|\lambda_{rl}|+2^{jA}\sum_{r=j}^{\infty}2^{-rA}\sup
_{l\in\mathbb{Z}^{n}}|\lambda_{rl}|. \label{17}%
\end{equation}

Let $N\in\mathbb{N}_{0}$. From \eqref{17}, changing the order of summation and using $A > 2$,  we get
\begin{align*}
\sum_{j=N}^{\infty}2^{-2j}\sup_{m\in\mathbb{Z}^{n}}|\chi_{jm}(f)|  &
\lesssim\sum_{j=N}^{\infty}2^{-2j}\sum_{r=0}^{j}\sup_{l\in\mathbb{Z}^{n}%
}|\lambda_{rl}| +\sum_{j=N}^{\infty}2^{(A-2)j}\sum_{r=j}^{\infty}2^{-rA}%
\sup_{l\in\mathbb{Z}^{n}}|\lambda_{rl}|\\
&  \lesssim2^{-2N}\sum_{r=0}^{N}\sup_{l\in\mathbb{Z}^{n}}|\lambda_{rl}%
|+\sum_{r=N}^{\infty}2^{-2r}\sup_{l\in\mathbb{Z}^{n}}|\lambda_{rl}|.
\end{align*}
It follows that
\begin{align}
\sup_{N\in\mathbb{N}_{0}}\frac{1}{\Psi(N)^{2}}\sum_{j=N}^{\infty}2^{-2j}%
\sup_{m\in\mathbb{Z}^{n}}|\chi_{jm}(f)|  &  \lesssim\sup_{N\in\mathbb{N}_{0}%
}\frac{2^{-2N}}{\Psi(N)^{2}}\,\sum_{r=0}^{N}\sup_{l\in\mathbb{Z}^{n}}%
|\lambda_{rl}|\nonumber\\
&  \hspace{0.5cm}+\sup_{N\in\mathbb{N}_{0}}\frac{1}{\Psi(N)^{2}}\sum
_{r=N}^{\infty}2^{-2r}\sup_{l\in\mathbb{Z}^{n}}|\lambda_{rl}|\nonumber\\
&  =:\mathcal{I}+\mathcal{II}. \label{18}%
\end{align}

Next we show $\mathcal{I}\lesssim\mathcal{II}$. Indeed, by condition (i) from Definition \ref{DefAdm},
\begin{align*}
\sum_{r=0}^{N}\sup_{l\in\mathbb{Z}^{n}}|\lambda_{rl}|  &  =\sum_{r=0}%
^{N}2^{2r}\Psi(r)^{2}\frac{2^{-2r}}{\Psi(r)^{2}}\sup_{l\in\mathbb{Z}^{n}%
}|\lambda_{rl}|\\
&  \leq\left(  \sup_{M\in\mathbb{N}_{0}}\,\left(  \frac{2^{-2M}}{\Psi(M)^{2}%
}\sup_{l\in\mathbb{Z}^{n}}|\lambda_{Ml}|\right)  \right)  \sum_{r=0}^{N}%
2^{2r}\Psi(r)^{2}\\
&  \lesssim\left(  \sup_{M\in\mathbb{N}_{0}}\left(  \frac{2^{-2M}}{\Psi
(M)^{2}}\sup_{l\in\mathbb{Z}^{n}}|\lambda_{Ml}|\right)  \right)  2^{2N}%
\Psi(N)^{2}.
\end{align*}
Consequently,%
\begin{equation*}
\frac{2^{-2N}}{\Psi(N)^{2}}\sum_{r=0}^{N}\sup_{l\in\mathbb{Z}^{n}}%
|\lambda_{rl}|    \lesssim\sup_{M\in\mathbb{N}_{0}}\left(  \,\frac{1}%
{\Psi(M)^{2}}\sum_{r=M}^{\infty}2^{-2r}\sup_{l\in\mathbb{Z}^{n}}|\lambda
_{rl}|\right)   =\mathcal{II}.
\end{equation*}
Now, taking supremum over all $N\in\mathbb{N}_{0}$, we arrive at
\[
\mathcal{I}\lesssim\mathcal{II}.
\]
Consequently \eqref{18} reads as
\begin{equation}
\sup_{N\in\mathbb{N}_{0}}\frac{1}{\Psi(N)^{2}}\sum_{j=N}^{\infty}2^{-2j}%
\sup_{m\in\mathbb{Z}^{n}}|\chi_{jm}(f)|\lesssim\mathcal{II}. \label{19}%
\end{equation}

Since $f\geq0$ and \eqref{11e} holds, we have (using the notation $\chi
_{Q}=\chi_{jm}$ if $Q=Q_{jm} \in \mathbb{D}_j$)
\[
\sup_{N\in\mathbb{N}_{0}}\frac{1}{\Psi(N)^{2}}\sum_{k=N-1}^{\infty}%
2^{k(-2+n)}\sup_{Q\in\mathbb{D}_{k}}\int_{Q}f\lesssim\sup_{N\in\mathbb{N}_{0}%
}\frac{1}{\Psi(N)^{2}}\sum_{k=N-1}^{\infty}2^{-2k}\sup_{Q\in\mathbb{D}_{k}%
}\chi_{Q}(f).
\]
Consequently by \eqref{18}-\eqref{19},
\[
\sup_{N\in\mathbb{N}_{0}}\frac{1}{\Psi(N)^{2}}\sum_{k=N-1}^{\infty}%
2^{k(-2+n)}\sup_{Q\in\mathbb{D}_{k}}\int_{Q}f\lesssim\sup_{N\in\mathbb{N}_{0}%
}\frac{1}{\Psi(N)^{2}}\sum_{r=N}^{\infty}2^{-2r}\sup_{l\in\mathbb{Z}^{n}%
}|\lambda_{rl}|,
\]
which combined with Proposition \ref{PropositionWavelet} yields
\[
\sup_{N\in\mathbb{N}_{0}}\frac{1}{\Psi(N)^{2}}\sum_{k=N-1}^{\infty}%
2^{k(-2+n)}\sup_{Q\in\mathbb{D}_{k}}\int_{Q}f\lesssim\Vert f\Vert_{V_{\Psi
}(\mathbb{R}^{n})},
\]
concluding the proof of \eqref{5}.

(ii): Combine (i) with Theorem \ref{CorFinal}. 
\end{proof}

\subsection{Proof of Theorem \ref{Thm11}\label{Section3}}
%
%
%
%
%
%
%
 Let $\Psi(t) = t^{\frac{1-\alpha}{2}}, \, \alpha > 1$. It follows from \eqref{PropositionWavelet2}, \eqref{11b} and \eqref{11a}
that
\begin{align*}
\sum_{k=N}^{\infty}2^{-2k}\sup_{l\in\mathbb{Z}^{n}}|\lambda
_{k l}|  &  \lesssim\sum_{k=N}^{\infty}2^{-2k}\sup_{l\in\mathbb{Z}^{n}}2^{kn/2}\int_{cQ_{kl}}|f||\Upsilon_{k l}|\\
&  \lesssim\sum_{k=N}^{\infty}2^{-2k}\sup_{l\in\mathbb{Z}^{n}}2^{kn}%
\int_{cQ_{kl}}|f|\\
&  \lesssim\Vert f\Vert_{M^{\frac{n}{2},\alpha}(\mathbb{R}^{n})}\sum
_{k=N}^{\infty}2^{-2k}\sup_{l\in\mathbb{Z}^{n}}2^{kn}|Q_{kl}|^{1-\frac{2}{n}%
}(1-(\log|Q_{kl}|)_{-})^{-\alpha}\\
&  \approx\Vert f\Vert_{M^{\frac{n}{2},\alpha}(\mathbb{R}^{n})}\sum
_{k=N}^{\infty}(1+k)^{-\alpha}\\
&  \approx N^{-\alpha+1}\,\Vert f\Vert_{M^{\frac{n}{2},\alpha}(\mathbb{R}%
^{n})}=\Psi(N)^{2}\,\Vert f\Vert_{M^{\frac{n}{2},\alpha}(\mathbb{R}^{n})}.
\end{align*}
Taking the supremum over all $N\in\mathbb{N}_{0}$ and invoking Proposition
\ref{PropositionWavelet}, we get
\begin{equation}
M^{\frac{n}{2},\alpha}(\mathbb{R}^{n})\hookrightarrow V_{\Psi}(\mathbb{R}^{n})
\label{fromMtoV}%
\end{equation}
as desired.

 To show that \eqref{fromMtoV} is strict, we apply again Proposition \ref{PropositionWavelet},
and the fact that $\alpha>1,$ to derive 
\begin{equation}
\Vert2^{(2-\frac{n}{2})K}\Upsilon_{K (0,\ldots,0)}\Vert_{V_{\Psi
}(\mathbb{R}^{n})}\approx\sup_{0\leq N\leq K}\frac{1}{N^{-\alpha+1}}=\frac
{1}{K^{-\alpha+1}}. \label{1.2}%
\end{equation}
On the other hand, we have
\begin{align*}
\Vert \Upsilon_{K (0,\ldots,0)}\Vert_{M^{\frac{n}{2},\alpha}(\mathbb{R}%
^{n})}  &  =\sup_{Q }\frac{(1-(\log|Q|)_-)^{\alpha}}{|Q|^{\frac{n-2}{n}%
}}\,\int_{Q}|\Upsilon_{K,(0,\ldots,0)}|\\
&  \geq \frac{(1-(\log|Q_{K(0,\ldots,0)} |)_-)^{\alpha}}{|Q_{K(0,\ldots,0)}%
|^{\frac{n-2}{n}}}\,\int_{Q_{K(0,\ldots,0)}}|\Upsilon_{K,(0,\ldots,0)}|\\
&  \approx\frac{K^{\alpha}}{2^{-K(n-2)}}\,\int_{Q_{K(0,\ldots,0)}}%
|\Upsilon_{K,(0,\ldots,0)}|\\
&  \gtrsim\frac{K^{\alpha}}{2^{-K(n-2)}}\,2^{-Kn}2^{Kn/2}=2^{K(\frac{n}{2}%
-2)}K^{\alpha},
\end{align*}
where we have used the wavelet properties in the penultimate step. 
Consequently,
\begin{equation}
\Vert2^{(2-\frac{n}{2})K}\Upsilon_{K (0,\ldots,0)}\Vert_{M^{\frac{n}%
{2},\alpha}(\mathbb{R}^{n})}\gtrsim K^{\alpha}. \label{1.3}%
\end{equation}

We now argue by contradiction. Suppose that, to the contrary,
\begin{equation*}
V_{\Psi}(\mathbb{R}^{n})\hookrightarrow M^{\frac{n}{2},\alpha}(\mathbb{R}%
^{n}).
\end{equation*}
In particular, we have, uniformly with respect to $K,$
\[
\Vert2^{(2-\frac{n}{2})K}\Upsilon_{K (0,\ldots,0)}\Vert_{M^{\frac{n}%
{2},\alpha}(\mathbb{R}^{n})}\lesssim\Vert2^{(2-\frac{n}{2})K}\Upsilon
_{K (0,\ldots,0)}\Vert_{V_{\Psi}(\mathbb{R}^{n})}.
\]
Combining the last inequality with \eqref{1.2} and \eqref{1.3}, yields
\[
K^{\alpha}\lesssim K^{\alpha-1},
\]
and letting $K\rightarrow\infty$ we arrive at a contradiction. \qed

\section{Sharpening Tadmor's regularity via $T_\Psi$\label{Section4}}

As already mentioned in Section \ref{sec:tad}, Tadmor \cite{Tadmor} proposed an approach, based on $R_{p, 2} \log^\alpha$-spaces, guaranteeing existence of Euler solutions with no concentration. In particular, in the distinguished  $2$D case, the author was able to improve the Morrey regularity of vortex sheets obtained in \cite{DiPernaMajda} from $\alpha = 1$ to the borderline regularity $\alpha = 1/2$. This is an application of the $H^{-1}$-stability method since (cf. \eqref{comp1})
\begin{equation}
R_{\frac{2 n}{n+2},2}\log^{\alpha} (\R^n)_c\overset{compactly}{\hookrightarrow}H_{\text{loc}}%
^{-1}(\R^n), \qquad \alpha > \frac{1}{2}.\label{comp1.8}%
\end{equation}

The goal of this section is to show that the results from \cite{Tadmor} admit improvements in terms of new scales of extrapolation spaces ($T_\Psi$-spaces, cf. Definition \ref{DefTPsi} below) and the method of sparse stability developed in previous sections. 

\subsection{$T_\Psi$-spaces}\label{Sub81}
To motivate the constructions that follow it is instructive to compare the
scalings of the spaces involved in the critical embeddings \eqref{comp}. For a function
space $X(\mathbb{R}^{n}),$ let the scaling parameter of $X$ be the number
$i_{X}$ such for all $\lambda>0,$ and all $\Vert f\Vert_{X}=1,$ we have
$\lambda^{-i_{X}}\Vert f(\lambda\cdot)\Vert_{X}=1.$ For the Morrey spaces
$M^{p}(\mathbb{R}^{n}),$ for $\lambda>0,$ and $\Vert f\Vert_{M^{p}%
(\mathbb{R}^{n})}=1,$ we have $\Vert f(\lambda\cdot)\Vert_{M^{p}%
(\mathbb{R}^{n})}=\lambda^{-n/p}$ so that $i_{M^{p}(\mathbb{R}^{n})}=-n/p.$
Likewise, for $H^{-1}(\mathbb{R}^{n})$, $i_{H^{-1}(\mathbb{R}^{n})}=-1-n/2.$
Comparing scaling parameters in the critical case $p=n/2,$ we see that
$\ i_{M^{n/2}(\mathbb{R}^{n})}=i_{H^{-1}(\mathbb{R}^{n})},$ only when $n=2,$
in which case the common value of the parameter is $-2.$ This suggests to seek
for an alternative to (\ref{comp}) where the involved
spaces have the same scaling parameter as $H^{-1}$ (i.e., $-1-n/2$). With this
in mind, we propose a new space $T_{\Psi}$, a variant of $V_{\Psi}$ introduced in Definition \ref{DefV}, which is
obtained by measuring the dyadic frequencies $\Delta_{j}f$ in the $L^{2}$-norm
rather than the $L^{\infty}$-norm.

\begin{definition}\label{DefTPsi}
Let $T_{\Psi}(\mathbb{R}^{n})$ be the space formed by all $f\in\mathcal{S}%
^{\prime}(\mathbb{R}^{n})$ such that
\[
\Vert f\Vert_{T_{\Psi}(\mathbb{R}^{n})}^{2}:=\sup_{N\in\mathbb{N}_{0}}\frac
{1}{\Psi(N)^{2}}\sum_{j=N}^{\infty}2^{-2j}\Vert\Delta_{j}f\Vert_{L^{2}%
(\mathbb{R}^{n})}^{2} < \infty. 
\]
Let $T^+_\Psi(\R^n) = T_\Psi(\R^n) \cap BM^+_c$. 
\end{definition}

\begin{remark}
	Similarly as in Remark  \ref{Remark1}, the space $T_\Psi$ can be seen as a dual counterpart of classical Vishik spaces involving the Besov space $B^{-1}_{2, 2}$. Note that $B^{-1}_{2, 2}$ can be identified with the inhomogeneous version of $H^{-1}$. 
\end{remark}

\begin{remark}
The space $T_{\Psi}$ admits a somewhat simpler characterization in terms of
Fourier transforms. Namely
\[
\Vert f\Vert_{T_{\Psi}(\mathbb{R}^{n})}^{2}\approx\sup_{N\in\mathbb{N}_{0}%
}\frac{1}{\Psi(N)^{2}}\int_{|\xi|>2^{N}-1}(1+|\xi|^{2})^{-1}\,|\widehat{f}%
(\xi)|^{2}\,d\xi.
\]
Indeed, this is a consequence of Plancherel's and Fubini's theorem, together
with basic properties of\footnote{Here $\varphi_j$ denotes the Fourier multiplier associated with $\Delta_j$, i.e., $\widehat{\Delta_j f}(\xi) = \varphi_j(\xi) \widehat{f}(\xi)$.} $\{\varphi_{j}\}_{j\in\mathbb{N}_{0}}$,
\begin{align*}
\Vert f\Vert_{T_{\Psi}(\mathbb{R}^{n})}^{2}  &  \approx\sup_{N\in
\mathbb{N}_{0}}\frac{1}{\Psi(N)^{2}}\int_{\mathbb{R}^{n}}\sum_{j=N}^{\infty
}[2^{-j}\varphi_{j}(\xi)]^{2}\,|\widehat{f}(\xi)|^{2}\,d\xi\\
&  \approx\sup_{N\in\mathbb{N}_{0}}\frac{1}{\Psi(N)^{2}}\int_{\mathbb{R}^{n}%
}(1+|\xi|^{2})^{-1}\mathbf{1}_{(2^{N}-1,\infty)}(|\xi|)\,|\widehat{f}%
(\xi)|^{2}\,d\xi.
\end{align*}
\end{remark}

\subsection{$T_\Psi$-regularity of Euler flows} We state now the main results of this section

\begin{theorem}\label{ThmTPsi} 
Let $\Psi$ be an admissible doubling decay.  
Then:
	\begin{enumerate}
		\item[(i)]  $T^+_\Psi(\R^n)_c \hookrightarrow S_\Psi(\R^n)_{c}$. As a consequence (cf. \eqref{116}), $T^+_\Psi(\R^n)_c\overset{compactly}{\hookrightarrow}H^{-1}_{\emph{loc}}(\R^n
)$. 
		\item[(ii)] Let $\{u^\varepsilon\}_{\varepsilon > 0}$ be a family of approximate solutions to Euler equations \eqref{Euler}, such that the related family of vorticities $\{\omega^\varepsilon\}_{\varepsilon > 0}$ is uniformly bounded in $L^\infty([0, T]; T^+_\Psi(\R^n)_c)$. Then $\{u^\varepsilon\}_{\varepsilon > 0}$ has a strong limit $u$ in $L^\infty([0, T]; L^2_{\emph{loc}}(\R^n))$, where $u$ is a solution with no concentrations. 
	\end{enumerate}
\end{theorem}

Specialising the previous result with $\Psi(t) = t^{-\alpha + \frac{1}{2}}$, we are able to improve 	\eqref{comp1.8} in the following sense.

\begin{theorem}
\label{Theorem4.8} Assume that $\alpha>\frac{1}{2}$. Then
$$
R_{\frac{2n}{n+2},2}\log^{\alpha}(\mathbb{R}^{n})\hookrightarrow T_{\Psi
}(\mathbb{R}^{n}). \label{414}%
$$
Furthermore, this embedding is strict in the sense that $R_{\frac{2n}{n+2},2}\log^{\alpha}(\mathbb{R}^{n})\neq T_{\Psi}(\mathbb{R}%
^{n}).$

\end{theorem}

\subsection{Extrapolation characterization of $T_\Psi$}
The next result represents the $T_{\Psi}$ spaces as extrapolation spaces for
the pair $(H^{-1},L^{2}).$ Since the proof follows closely the one for Theorem
\ref{PropInterpol}, we shall leave the details to the interested reader.

\begin{theorem}
\label{PropInterpolDoubling} Suppose that $\Psi$ is an admissible doubling decay. Then\footnote{Since $T_\Psi$ is not   homogeneous, the space $H^{-1}(\R^n)$  in \eqref{valecuatro} should be adequately interpreted from the context as the inhomogeneous counterpart of \eqref{SLat}, equipped with the  norm $\|f\|_{H^{-1}(\R^n)} = \|(I-\Delta)^{-\frac{1}{2}} f\|_{L^2(\R^n)}$. Recall the well-known fact that $\|(-\Delta)^{-\frac{1}{2}} f\|_{L^2(\R^n)} \approx \|I_1 f\|_{L^2(\R^n)}$.}
\begin{equation}
\Vert f\Vert_{T_{\Psi}(\mathbb{R}^{n})}\approx\sup_{t\in(0,1)}\frac
{K(t,f;H^{-1}(\mathbb{R}^{n}),L^{2}(\mathbb{R}^{n}))}{\Psi(-\log t)}.
\label{valecuatro}%
\end{equation}

\end{theorem}

\begin{remark}
A variant of Remark \ref{Remark16} also holds for the $T_{\Psi}(\mathbb{R}%
^{n})$ spaces. Indeed, suppose that the admissibility condition stated Definition \ref{DefAdm}(i) is replaced by
\[
\sum_{r=0}^{N}(2^{(1+s)r}\Psi(r))^{2}\lesssim(2^{(1+s)N}\Psi(N))^{2}.
\]
Let $H^{s}(\mathbb{R}^{n})$ be the standard (fractional) Sobolev space endowed
with the norm $\Vert f\Vert_{H^{s}(\mathbb{R}^{n})}=\Vert(I-\Delta
)^{s/2}f\Vert_{L^{2}(\mathbb{R}^{n})},$ then formula (\ref{valecuatro}) holds
if we replace the pair $(H^{-1}(\mathbb{R}^{n}),L^{2}(\mathbb{R}^{n}))$ by
$(H^{-1}(\mathbb{R}^{n}),H^{s}(\mathbb{R}^{n})),$ where $s>-1.$
\end{remark}

%
%
%
%

\subsection{Proof of Theorem \ref{ThmTPsi}}
For the proof we will use the following analogue of Proposition
\ref{PropositionWavelet}, that can be obtained mutatis mutandi and we
therefore leave its proof to the interested reader.

\begin{proposition}
\label{PropositionWaveletT} Let $\{\Upsilon^G_{N l}:N\in\mathbb{N}%
_{0},\,G\in G^{N},\,l\in\mathbb{Z}^{n}\}$ be a wavelet system satisfying
\eqref{11b}--\eqref{11} with\footnote{The explanation behind $A> 1$ comes from Theorem \ref{PropInterpolDoubling} and  well-known wavelet assumptions on $H^{-1}$.} $A>1$. Assume that $\Psi$ is an admissible doubling decay. Then, $f\in T_{\Psi}%
(\mathbb{R}^{n})$ if and only if
\[
f=\sum_{N\in\mathbb{N}_{0},G\in G^{N},l\in\mathbb{Z}^{n}}\lambda_{N l}%
^{G}2^{-Nn/2}\Upsilon_{N l}^{G},\qquad\{\lambda_{N l}^{G}\}\in t_{\Psi}%
\]
(unconditional convergence in the sense of $\mathcal{S}^{\prime}%
(\mathbb{R}^{n})$), where
\begin{equation}
\Vert\{\lambda^{G}_{N l}\}\Vert_{t_{\Psi}}^{2} :=\sup_{N\in\mathbb{N}_{0}}%
\frac{1}{\Psi(N)^{2}}\,\sum_{k=N}^{\infty}2^{k(-2-n)}\sum_{G\in G^{k}%
,l\in\mathbb{Z}^{n}}|\lambda_{k l}^{G}|^{2}<\infty. \label{Deft}%
\end{equation}
This representation is unique in the sense that the coefficients $\lambda
_{N l}^{G}$ are determined by \eqref{PropositionWavelet2} and the operator $I$
given by \eqref{OperI} defines an isomorphism from $T_{\Psi}(\mathbb{R}^{n})$
onto $t_{\Psi}$. Furthermore
\begin{equation}
\Vert f\Vert_{T_{\Psi}(\mathbb{R}^{n})}\approx\Vert\{\lambda_{l}^{N,G}%
\}\Vert_{t_{\Psi}}. \label{vanilla2}%
\end{equation}

\end{proposition}

\begin{proof}[Proof of Theorem \ref{ThmTPsi}]
(i): Let  $L\in\mathbb{N}_{0}$.  Given $f \geq 0$ compactly supported on $Q_0$ (without loss of generality, we may assume $Q_0 = (0, 1)^n$), we observe
that
\begin{equation*}
s_{L}(f)    = \sup_{\mathcal{Q} \in S(Q_0)}  \bigg[\sum_{k=L}^{\infty}2^{k(-2+n)}\sum_{i\in I:Q_{i}\in\mathbb{D}%
_{k}(\mathcal{Q)}}\bigg(\int_{Q_{i}}f\bigg)^{2}\bigg]^{\frac{1}{2}}.
\end{equation*}
Using Proposition \ref{PropositionWaveletT} we will show that%
\begin{equation}
\sup_{L\in\mathbb{N}_{0}}\frac{1}{\Psi(L)}\,\bigg[\sum_{k=L}^{\infty
}2^{k(-2+n)}\sum_{i\in I:Q_{i}\in\mathbb{D}_{k}(\mathcal{Q)}}\bigg(\int%
_{Q_{i}}f\bigg)^{2}\bigg]^{\frac{1}{2}}\lesssim\,\Vert f\Vert_{T_{\Psi
}(\mathbb{R}^{n})}, \label{ProofThm1T1}%
\end{equation}
with a constant independent of the sparse family $\mathcal{Q}$. Then the desired
(local) embedding
\[
\Vert f\Vert_{S_{\Psi}(Q_{0})}\lesssim\Vert f\Vert_{T_{\Psi
}(\mathbb{R}^{n})}
\]
follows readily. 

Let $\chi$ be a smooth cut-off function introduced in the proof of Theorem  \ref{teo:compacto}, and define the corresponding coefficients $\chi_{jm}(f)$ via
\eqref{Chijm}. According to \eqref{11d}, \eqref{6}, \eqref{221} and \eqref{222}, these coefficients can be estimated as follows
%
\begin{equation}
|\chi_{jm}(f)|  \lesssim\sum_{r=0}^{j}\sum_{l\in\ell_{r}^{j}(m)}|\lambda_{rl}|+\sum
_{r=j}^{\infty}2^{(j-r)(n+A)}\sum_{l\in\ell_{r}^{j}(m)}|\lambda_{rl}|,
\label{ProofThm1T2}%
\end{equation}
where $\ell^j_r(m)$, which was introduced in \eqref{7}, satisfies
\begin{equation}
\text{card }\ell_{r}^{j}(m)\approx\left\{
\begin{array}
[c]{cl}%
2^{n(r-j)} & \text{if}\quad r\geq j,\\
& \\
1 & \text{if}\quad r\leq j.
\end{array}
\right.  \label{ProofThm1T3}%
\end{equation}
Using H\"{o}lder's inequality and \eqref{ProofThm1T3},
\[
\sum_{l\in\ell_{r}^{j}(m)}|\lambda_{rl}|\lesssim\bigg(\sum_{l\in\ell_{r}%
^{j}(m)}|\lambda_{rl}|^{2}\bigg)^{1/2}\times\left\{
\begin{array}
[c]{cl}%
2^{n\frac{r-j}{2}} & \text{if}\quad r\geq j,\\
& \\
1 & \text{if}\quad r\leq j,
\end{array}
\right.
\]
and inserting this into \eqref{ProofThm1T2}, we achieve
\begin{equation}
|\chi_{jm}(f)|\lesssim\sum_{r=0}^{j}\bigg(\sum_{l\in\ell_{r}^{j}(m)}%
|\lambda_{rl}|^{2}\bigg)^{1/2}+2^{j(A+\frac{n}{2})}\sum_{r=j}^{\infty
}2^{-r(A+\frac{n}{2})}\,\bigg(\sum_{l\in\ell_{r}^{j}(m)}|\lambda_{rl}%
|^{2}\bigg)^{1/2}. \label{ProofThm1T4}%
\end{equation}

Let $\varepsilon\in(0,\min\{1,A-1\})$ (recall that $A>1$ is fixed). By
H\"{o}lder's inequality, the two terms given in the right-hand side of
\eqref{ProofThm1T4} can be bounded by
\[
\sum_{r=0}^{j}\bigg(\sum_{l\in\ell_{r}^{j}(m)}|\lambda_{rl}|^{2}%
\bigg)^{1/2}\lesssim2^{j\varepsilon}\left(  \sum_{r=0}^{j}2^{-r\varepsilon
2}\sum_{l\in\ell_{r}^{j}(m)}|\lambda_{rl}|^{2}\right)  ^{1/2}%
\]
and
\[
\sum_{r=j}^{\infty}2^{-r(A+\frac{n}{2})}\,\bigg(\sum_{l\in\ell_{r}^{j}%
(m)}|\lambda_{rl}|^{2}\bigg)^{1/2}\lesssim2^{-j\varepsilon}\left(  \sum
_{r=j}^{\infty}2^{-r(A+\frac{n}{2}-\varepsilon)2}\sum_{l\in\ell_{r}^{j}%
(m)}|\lambda_{rl}|^{2}\right)  ^{1/2}.
\]
Hence
\[
|\chi_{jm}(f)|^{2}\lesssim2^{j\varepsilon2}\sum_{r=0}^{j}2^{-r\varepsilon
2}\sum_{l\in\ell_{r}^{j}(m)}|\lambda_{rl}|^{2}+2^{j(A+\frac{n}{2}%
-\varepsilon)2}\sum_{r=j}^{\infty}2^{-r(A+\frac{n}{2}-\varepsilon)2}\sum
_{l\in\ell_{r}^{j}(m)}|\lambda_{rl}|^{2}%
\]
and summing up over all $m\in\mathbb{Z}^{n}$, we have
\begin{equation}
\sum_{m\in\mathbb{Z}^{n}}|\chi_{jm}(f)|^{2}    \lesssim2^{j\varepsilon2}%
\sum_{r=0}^{j}2^{-r\varepsilon2}\sum_{m\in\mathbb{Z}^{n}}\sum_{l\in\ell
_{r}^{j}(m)}|\lambda_{rl}|^{2} +2^{j(A+\frac{n}{2}-\varepsilon)2}\sum_{r=j}^{\infty
}2^{-r(A+\frac{n}{2}-\varepsilon)2}\sum_{m\in\mathbb{Z}^{n}}\sum_{l\in\ell
_{r}^{j}(m)}|\lambda_{rl}|^{2}. \label{ProofThm1T5}%
\end{equation}
Furthermore, we remark that
\[
\text{card }\{m\in\mathbb{Z}^{n}:l\in\ell_{r}^{j}(m)\}\approx\left\{
\begin{array}
[c]{cl}%
2^{n(j-r)} & \text{if}\quad r\leq j,\\
& \\
1 & \text{if}\quad r\geq j.
\end{array}
\right.
\]
Using this information and changing the order of summation in
\eqref{ProofThm1T5}, we arrive at
\begin{align*}
\sum_{m\in\mathbb{Z}^{n}}|\chi_{jm}(f)|^{2}  &  \lesssim2^{j\varepsilon2}%
\sum_{r=0}^{j}2^{-r\varepsilon2}2^{n(j-r)}\sum_{l\in\mathbb{Z}^{n}}%
|\lambda_{rl}|^{2}+2^{j(A+\frac{n}{2}-\varepsilon)2}\sum_{r=j}^{\infty
}2^{-r(A+\frac{n}{2}-\varepsilon)2}\sum_{l\in\mathbb{Z}^{n}}|\lambda_{rl}%
|^{2},
\end{align*}
where the equivalence constant is independent of $j$. Summing up the last
estimate over all $j\geq L$ and using Fubini's theorem (recall that
$\varepsilon<\min\{1,A-1\}$), we have
\begin{align*}
\sum_{j=L}^{\infty}2^{j(-2-n)}\sum_{m\in\mathbb{Z}^{n}}|\chi_{jm}(f)|^{2}  & \\
&\hspace{-3.75cm}
\lesssim\sum_{j=L}^{\infty}2^{j(-1+\varepsilon)2}\sum_{r=0}^{j}%
2^{-r(\varepsilon2+n)}\sum_{l\in\mathbb{Z}^{n}}|\lambda_{rl}|^{2}+\sum_{j=L}^{\infty}2^{j2(-1+A-\varepsilon)}\sum_{r=j}^{\infty
}2^{-r(A+\frac{n}{2}-\varepsilon)2}\sum_{l\in\mathbb{Z}^{n}}|\lambda_{rl}%
|^{2}\\
& \hspace{-3.75cm} \lesssim2^{L(-1+\varepsilon)2}\sum_{r=0}^{L}2^{-r(\varepsilon2+n)}%
\sum_{l\in\mathbb{Z}^{n}}|\lambda_{rl}|^{2} +\sum_{r=L}^{\infty}2^{-r(\varepsilon2+n)}\sum_{l\in
\mathbb{Z}^{n}}|\lambda_{rl}|^{2}\sum_{j=r}^{\infty}2^{j(-1+\varepsilon)2}\\
&  \hspace{-3cm}+\sum_{r=L}^{\infty}2^{-r(A+\frac{n}{2}-\varepsilon)2}%
\sum_{l\in\mathbb{Z}^{n}}|\lambda_{rl}|^{2}\sum_{j=L}^{r}%
2^{j2(-1+A-\varepsilon)}\\
&  \hspace{-3.75cm}  \lesssim2^{L(-1+\varepsilon)2}\sum_{r=0}^{L}2^{-r(\varepsilon2+n)}%
\sum_{l\in\mathbb{Z}^{n}}|\lambda_{rl}|^{2}+\sum_{r=L}^{\infty}2^{-r(2+n)}\sum_{l\in\mathbb{Z}^{n}%
}|\lambda_{rl}|^{2}.
\end{align*}
Hence
\begin{equation}
\sup_{L\in\mathbb{N}_{0}}\frac{1}{\Psi(L)}\left[  \sum_{j=L}^{\infty
}2^{j(-2-n)}\sum_{m\in\mathbb{Z}^{n}}|\chi_{jm}(f)|^{2}\right]  ^{1/2}%
\lesssim\mathcal{I}+\mathcal{II}, \label{ProofThm1T6}%
\end{equation}
where
\[
\mathcal{I}:=\sup_{L\in\mathbb{N}_{0}}\frac{2^{L(-1+\varepsilon)}}{\Psi
(L)}\left[  \sum_{r=0}^{L}2^{-r(\varepsilon2+n)}\sum_{l\in\mathbb{Z}^{n}%
}|\lambda_{rl}|^{2}\right]  ^{1/2}%
\]
and
\[
\mathcal{II}:=\sup_{L\in\mathbb{N}_{0}}\frac{1}{\Psi(L)}\left[  \sum
_{r=L}^{\infty}2^{-r(2+n)}\sum_{l\in\mathbb{Z}^{n}}|\lambda_{rl}|^{2}\right]
^{1/2}.
\]

We have
\begin{align}
\left[  \sum_{r=0}^{L}2^{-r(\varepsilon2+n)}\sum_{l\in\mathbb{Z}^{n}}%
|\lambda_{rl}|^{2}\right]  ^{1/2}  &  \leq\left[  \sum_{r=0}^{L}%
2^{r(1-\varepsilon)2}\Psi(r)^{2}\right]  ^{1/2}\,\sup_{M\in\mathbb{N}_{0}%
}\frac{2^{-M(1+\frac{n}{2})}}{\Psi(M)}\,\left[\sum_{l\in\mathbb{Z}^{n}%
}|\lambda_{Ml}|^{2}\right]^{1/2}\nonumber\\
&  \leq\left[  \sum_{r=0}^{L}2^{r(1-\varepsilon)2}\Psi(r)^{2}\right]
^{1/2}\,\mathcal{II}. \label{ProofThm1T7}%
\end{align}
Furthermore, the following estimate holds
\begin{equation}
\left[  \sum_{r=0}^{L}2^{r(1-\varepsilon)2}\Psi(r)^{2}\right]  ^{1/2}%
\lesssim2^{L(1-\varepsilon)}\Psi(L). \label{ProofThm1T8}%
\end{equation}
Indeed, by monotonicity properties, a simple change of variables and the
doubling property of $\Psi$ (cf. (ii) in Definition \ref{DefAdm}),
\begin{align*}
\sum_{r=0}^{L}2^{r(1-\varepsilon)2}\Psi(r)^{2}  &  \approx\int_{0}%
^{L}2^{t(1-\varepsilon)2}\Psi(t)^{2}\,dt  \approx\int_{0}^{L(1-\varepsilon)}2^{t2}\Psi\bigg(\frac{t}{1-\varepsilon
}\bigg)^{2}\,dt\\
&  \approx\int_{0}^{L(1-\varepsilon)}2^{t2}\Psi(t)^{2}\,dt  \approx\sum_{r=0}^{\lfloor L(1-\varepsilon)\rfloor}2^{r2}\Psi(r)^{2}\\
&  \lesssim2^{L(1-\varepsilon)2}\Psi(L)^{2},
\end{align*}
where\footnote{As usual, $\lfloor x \rfloor$ denotes the integer part of $x \in \R$.} the last step follows from (i) in Definition \ref{DefAdm}. This proves \eqref{ProofThm1T8}.
Applying now \eqref{ProofThm1T8} in \eqref{ProofThm1T7}, we find
\[
\left[  \sum_{r=0}^{L}2^{-r(\varepsilon2+n)}\sum_{l\in\mathbb{Z}^{n}}%
|\lambda_{rl}|^{2}\right]  ^{1/2}\lesssim2^{L(1-\varepsilon)}\Psi
(L)\,\mathcal{II},
\]
i.e., we have shown that $\mathcal{I}\lesssim\mathcal{II}$. As a consequence
(cf. \eqref{ProofThm1T6})
\[
\sup_{L\in\mathbb{N}_{0}}\frac{1}{\Psi(L)}\left[  \sum_{j=L}^{\infty
}2^{j(-2-n)}\sum_{m\in\mathbb{Z}^{n}}|\chi_{jm}(f)|^{2}\right]  ^{1/2}%
\lesssim\mathcal{II},
\]
or equivalently (cf. \eqref{Deft})
\[
\sup_{L\in\mathbb{N}_{0}}\frac{1}{\Psi(L)}\left[  \sum_{j=L}^{\infty
}2^{j(-2-n)}\sum_{m\in\mathbb{Z}^{n}}|\chi_{jm}(f)|^{2}\right]  ^{1/2}%
\lesssim\Vert\{\lambda_{rl}\}\Vert_{t_{\Psi}}.
\]
Consequently, invoking Proposition \ref{PropositionWaveletT},%
\begin{equation}
\sup_{L\in\mathbb{N}_{0}}\frac{1}{\Psi(L)}\left[  \sum_{j=L}^{\infty
}2^{j(-2-n)}\sum_{m\in\mathbb{Z}^{n}}|\chi_{jm}(f)|^{2}\right]  ^{1/2}%
\lesssim\Vert f\Vert_{T_{\Psi}(\mathbb{R}^{n})}. \label{ProofThm1T9}%
\end{equation}

On the other hand, the assumption $f\geq0$ and \eqref{11e} enable us to
derive
\begin{align*}
\sup_{L\in\mathbb{N}_{0}}\frac{1}{\Psi(L)}\,\bigg[\sum_{k=L}^{\infty
}2^{k(-2+n)}\sum_{i\in I:Q_{i}\in\mathbb{D}_{k}(\mathcal{Q)}}\bigg(\int%
_{Q_{i}}f\bigg)^{2}\bigg]^{\frac{1}{2}}  &  =\\
&  \hspace{-7cm}\sup_{L\in\mathbb{N}_{0}}\frac{1}{\Psi(L)}\,\bigg[\sum
_{k=L}^{\infty}2^{k(-2-n)}\sum_{i\in I:Q_{i}\in\mathbb{D}_{k}(\mathcal{Q)}%
}\bigg(\int_{Q_{i}}2^{kn}f\bigg)^{2}\bigg]^{\frac{1}{2}}\\
&  \hspace{-7cm}\lesssim\sup_{L\in\mathbb{N}_{0}}\frac{1}{\Psi(L)}%
\,\bigg[\sum_{k=L}^{\infty}2^{k(-2-n)}\sum_{i\in I:Q_{i}\in\mathbb{D}%
_{k}(\mathcal{Q)}}\chi_{Q_{i}}(f)^{2}\bigg]^{\frac{1}{2}}\\
&  \hspace{-7cm}\lesssim\Vert f\Vert_{T_{\Psi}(\mathbb{R}^{n})},
\end{align*}
where in the last step we used \eqref{ProofThm1T9}. This concludes the proof
of \eqref{ProofThm1T1} and hence (i). 

Invoking Theorem \ref{CorFinal},  the statement (ii) is a consequence of (i).  
\end{proof}

\subsection{Proof of Theorem \ref{Theorem4.8}}
To avoid unnecessary technicalities, we assume, without loss of
generality, that the constant $c~\ $in \eqref{11b} is equal to $1.$ From \eqref{PropositionWavelet2} and 
\eqref{11a}, we find
\begin{align*}
\sum_{k=N}^{\infty}2^{k(-2-n)}\sum_{G\in G^{k},l\in\mathbb{Z}^{n}}|\lambda
_{k l}^{G}|^{2}  &  \lesssim\sum_{k=N}^{\infty}2^{k(-2+n)}\sum_{l\in
\mathbb{Z}^{n}}\bigg(\int_{Q_{kl}}|f|\bigg)^{2}\\
&  \approx \sum_{k=N}^{\infty}2^{k(-2+n)}2^{-k (-2+n)}k^{-2\alpha}%
\sum_{l\in\mathbb{Z}^{n}}\left(   \frac{|\log|Q_{kl}||^{\alpha}%
}{|Q_{kl}|^{\frac{1}{(\frac{2n}{n+2})^{\prime}}}}\int_{Q_{kl}}|f|\right)
^{2}\\
&  \leq\Vert f\Vert_{R_{\frac{2n}{n+2},2}\log^{\alpha}(\mathbb{R}^{n})}%
^{2}\,\sum_{k=N}^{\infty}k^{-2\alpha} \approx N^{-2\alpha+1}\,\Vert f\Vert_{R_{\frac{2n}{n+2},2}\log^{\alpha
}(\mathbb{R}^{n})}^{2}.
\end{align*}
Consequently,
\[
\sup_{N\in\mathbb{N}_{0}}\frac{1}{\Psi(N)^{2}}\,\sum_{k=N}^{\infty}%
2^{k(-2-n)}\sum_{G\in G^{k},l\in\mathbb{Z}^{n}}|\lambda_{k l}^{G}|^{2}%
\lesssim\Vert f\Vert_{R_{\frac{2n}{n+2},2}\log^{\alpha}(\mathbb{R}^{n})}^{2},
\]
where $\Psi(t)=t^{-\alpha+\frac{1}{2}}$. Then by (\ref{Deft}) and (\ref{vanilla2})
it follows that
\begin{equation}\label{strictR}
R_{\frac{2n}{n+2},2}\log^{\alpha}(\mathbb{R}^{n})\hookrightarrow T_{\Psi
}(\mathbb{R}^{n}),\qquad\Psi(t)=t^{-\alpha+\frac{1}{2}}.
\end{equation}

To show that the embedding \eqref{strictR} is strict we argue by contradiction. Suppose to the contrary that
\[
R_{\frac{2n}{n+2},2}\log^{\alpha}(\mathbb{R}^{n})=T_{\Psi}(\mathbb{R}^{n}).
\]
In particular (for a fixed $G$)
\begin{equation}
\Vert2^{K}\Upsilon_{K (0,\ldots,0)}^{G}\Vert_{T_{\Psi}(\mathbb{R}^{n})}%
\approx\Vert2^{K}\Upsilon_{K (0,\ldots,0)}^{G}\Vert_{R_{\frac{2n}{n+2},2}%
\log^{\alpha}(\mathbb{R}^{n})}\quad\text{for every}\quad K\in\mathbb{N}.
\label{417}%
\end{equation}
Using Proposition \ref{PropositionWaveletT}, we compute
\[
\Vert2^{K}\Upsilon_{K (0,\ldots,0)}^{G}\Vert_{T_{\Psi}(\mathbb{R}^{n})}%
\approx\sup_{N\leq K}\frac{1}{N^{-\alpha+\frac{1}{2}}}=\frac{1}{K^{-\alpha
+\frac{1}{2}}},
\]
which combined with \eqref{417} results in
\begin{align*}
\frac{1}{K^{-\alpha+\frac{1}{2}}}  &  \approx\Vert2^{K}\Upsilon_{K (0,\ldots
,0)}^{G}\Vert_{R_{\frac{2n}{n+2},2}\log^{\alpha}(\mathbb{R}^{n})}\\
&  \gtrsim \frac{|\log|Q_{K(0,\ldots,0)}||^{\alpha}}{|Q_{K(0,\ldots,0)}%
|^{\frac{n-2}{2n}}}\,\int_{Q_{K(0,\ldots,0)}}2^{K}|\Upsilon_{K (0,\ldots
,0)}^{G}|\\
&  \approx\frac{K^{\alpha}}{2^{-\frac{Kn}{2}}}\,\int_{Q_{K(0,\ldots,0)}%
}|\Upsilon_{K (0,\ldots,0)}^{G}|\\
&  \gtrsim\frac{K^{\alpha}}{2^{-\frac{Kn}{2}}}\,2^{\frac{Kn}{2}}\,|Q_{K(0,\ldots
,0)}|=K^{\alpha}.
\end{align*}
Taking limits as $K\rightarrow\infty$ we arrive to a contradiction. \qed

\section{Comparison between $V_{\Psi}$ and $T_{\Psi}$}

In view of the results obtained in Sections \ref{SectionDMaj} and \ref{Section4}, it is natural to compare\footnote{According to the discussion at the beginning of Section \ref{Sub81}, we may restrict our attention to the $2$D setting.} the
spaces $V_{\Psi}(\mathbb{R}^{2})$ and $T_{\Psi}(\mathbb{R}^{2})$ for a fixed decay $\Psi$. In this
section we show that neither space contains the other. More precisely, we
construct explicit examples of functions showing that $T_{\Psi}(\mathbb{R}%
^{2})\backslash V_{\Psi}(\mathbb{R}^{2})\neq\emptyset$ and $V_{\Psi
}(\mathbb{R}^{2})\backslash T_{\Psi}(\mathbb{R}^{2})\neq\emptyset$.

\begin{example}[$T_{\Psi}(\mathbb{R}^{2})\backslash V_{\Psi}(\mathbb{R}%
^{2})\neq\emptyset$]
\label{Section5.1} Given a scalar sequence $\{c_{N}\}_{N\in\mathbb{N}_{0}}$,
let $f$ be given by \eqref{PropositionWavelet1}, where
\[
\lambda_{N l}=\left\{
\begin{array}
[c]{cl}%
2^{2N}c_{N}, & N\in\mathbb{N}_{0},\quad l=(0,\ldots,0),\\
0, & \text{otherwise}.
\end{array}
\right.
\]
We compute the norms of $f$ in $V_{\Psi}$ and $T_{\Psi}$ using Propositions
\ref{PropositionWavelet} and \ref{PropositionWaveletT}, respectively,
\[
\Vert f\Vert_{V_{\Psi}(\mathbb{R}^{2})}\approx\sup_{N\in\mathbb{N}_{0}}%
\frac{1}{\Psi(N)^{2}}\,\sum_{k=N}^{\infty}|c_{k}|
\]
and
\[
\Vert f\Vert_{T_{\Psi}(\mathbb{R}^{2})}\approx\sup_{N\in\mathbb{N}_{0}}%
\frac{1}{\Psi(N)}\,\bigg(\sum_{k=N}^{\infty}|c_{k}|^{2}\bigg)^{\frac{1}{2}}.
\]
Therefore, if we select $\{c_{N}\}_{N\in\mathbb{N}_{0}}\in\ell_{2}$ such
that
\begin{equation}
\sum_{k=N}^{\infty}c_{k}^{2}\approx\Psi(N)^{2} \label{Approx}%
\end{equation}
then $\Vert f\Vert_{T_{\Psi}(\mathbb{R}^{2})}\approx1.$ The existence of such
sequences (even with $\approx$ replaced by $=$ in \eqref{Approx}) is
guaranteed by classical results in approximation theory (see e.g. \cite[Section 2.5]{Timan}). On the other hand, since
\[
\sum_{k=N}^{\infty}|c_{k}| \geq\left(  \inf_{l\geq N}\frac{1}{|c_{l}|}\right)  \,\sum_{k=N}^{\infty
}|c_{k}|^{2}\approx\frac{\Psi(N)^{2}}{\sup_{l\geq N}|c_{l}|},
\]
we have%
\begin{equation}
\frac{1}{\Psi(N)^{2}}\,\sum_{k=N}^{\infty}|c_{k}|\gtrsim\frac{1}{\sup_{l\geq
N}|c_{l}|} \label{davos}%
\end{equation}
but since $\{c_{N}\}_{N\in\mathbb{N}_{0}}\in\ell_{2},$ we have $\lim
_{N\rightarrow\infty}|c_{N}|=0$ and therefore the left-hand side of
(\ref{davos}) tends to $\infty$ as $N\rightarrow\infty,$ showing that
$f\not \in V_{\Psi}(\mathbb{R}^{2})$.
\end{example}

\begin{example}[$V_{\Psi}(\mathbb{R}^{2})\backslash T_{\Psi}(\mathbb{R}%
^{2}) \neq \emptyset$]
\label{Section5.2}  Given a scalar sequence $\{c_{l}\}_{l\in\mathbb{Z}^{n}}$, define $f$
using \eqref{PropositionWavelet1} with
\[
\lambda_{N l}=\left\{
\begin{array}
[c]{cl}%
c_{l}, & N=0,\quad l\in\mathbb{Z}^{n},\\
0, & \text{otherwise}.
\end{array}
\right.
\]
Computing norms using Propositions \ref{PropositionWavelet} and
\ref{PropositionWaveletT} \ we find
\[
\Vert f\Vert_{V_{\Psi}(\mathbb{R}^{2})}\approx\sup_{l\in\mathbb{Z}^{n}}%
|c_{l}| \qquad \text{and} \qquad \Vert f\Vert_{T_{\Psi}(\mathbb{R}^{2})}\approx\bigg(\sum_{l\in\mathbb{Z}^{n}%
}|c_{l}|^{2}\bigg)^{1/2}.
\]
Thus, if we select $\{c_{l}\}_{l\in\mathbb{Z}^{n}}\in\ell_{\infty}%
\backslash\ell_{2}$ we obtain an example of $f\in V_{\Psi}(\mathbb{R}^{2})$
but $f\not \in T_{\Psi}(\mathbb{R}^{2})$.
\end{example}

\begin{remark}
\label{Remark5.1} The above computations show a stronger assertion: Given any
decays $\Psi$ and $\Phi$, one can always construct $f\in V_{\Psi}%
(\mathbb{R}^{2})$ such that $f\not \in T_{\Phi}(\mathbb{R}^{2})$.
\end{remark}

The previous remark shows that for different decays $\Psi\neq\Phi,$ $V_{\Psi
}(\mathbb{R}^{2})$ cannot be contained in $T_{\Phi}(\mathbb{R}^{2}).$ The next
result shows that under some additional condition the reverse inclusion is
possible. 

\begin{proposition}
Suppose that
\begin{equation}
\sum_{j=N}^{\infty}\Phi(j)\lesssim\Psi(N)^{2},\qquad N\in\mathbb{N}_{0}.
\label{5.2}%
\end{equation}
Then
\[
T_{\Phi}(\mathbb{R}^{2})\hookrightarrow V_{\Psi}(\mathbb{R}^{2}).
\]

\end{proposition}

\begin{proof}
We use Nikolskii's inequality for entire functions of exponential type (see e.g. \cite[Section 4.9.53]{Timan}) to
estimate
\begin{align*}
\sum_{j=N}^{\infty}2^{-2j}\Vert\Delta_{j}f\Vert_{L^{\infty}(\mathbb{R}^{2})}
&  \lesssim\sum_{j=N}^{\infty}2^{-j}\Vert\Delta_{j}f\Vert_{L^{2}%
(\mathbb{R}^{2})}\\
&  \leq\left(  \sup_{M\in\mathbb{N}_{0}}\frac{2^{-M}}{\Phi(M)}\,\Vert
\Delta_{M}f\Vert_{L^{2}(\mathbb{R}^{2})}\right)  \,\sum_{j=N}^{\infty}%
\Phi(j)\\
&  \lesssim \Psi(N)^{2}  \sup_{M\in\mathbb{N}_{0}}\frac{2^{-M}}{\Phi(M)}%
\,\Vert\Delta_{M}f\Vert_{L^{2}(\mathbb{R}^{2})} \\
&  \leq \Psi(N)^{2}\,\sup_{M\in\mathbb{N}_{0}}\frac{1}{\Phi(M)}%
\bigg(\sum_{j=M}^{\infty}2^{-2j}\Vert\Delta_{j}f\Vert_{L^{2}(\mathbb{R}^{2}%
)}^{2}\bigg)^{1/2}\\
&  =\Psi(N)^{2}\,\Vert f\Vert_{T_{\Phi}(\mathbb{R}^{2})}.
\end{align*}
Therefore $T_{\Phi}(\mathbb{R}^{2})\hookrightarrow V_{\Psi}(\mathbb{R}^{2})$.
\end{proof}

\begin{remark}
The assumption \eqref{5.2} is quite restrictive, in particular, it forces the
series $\sum_{j=0}^{\infty}\Phi(j)$ to be convergent. This automatically
excludes many interesting examples of spaces $T_{\Phi}$ such as the
corresponding one to the decay $\Phi(j)=j^{-\alpha}$ with $\alpha\in(0,1]$,
which are connected with $R$-spaces (cf. Theorem \ref{Theorem4.8}). In
conclusion, the analysis of $T_{\Phi}$ can not be, in general, reduced to
study of the simpler spaces $V_{\Psi}$.
\end{remark}

\section{A sparse approach to energy conservation\label{sec:physically}}

Throughout this section, we work with the following special class of approximate solutions on $\mathbb{T}^2 \equiv [0, 2 \pi]^2$ introduced in \cite[Definition 3]{Ches}. 

\begin{definition}\label{DefVV}
	Let $u \in C(0, T; L^2(\T^2))$ with $u_0 \in L^2(\T^2)$. We say that a weak solution $u$ of Euler equations is  \emph{physically realizable} with initial velocity $u_0$ provided that there exists a family $\{u^\varepsilon\}_{\varepsilon > 0}$ of solutions of Navier-Stokes equations with viscosity $\varepsilon$, such that $u^\varepsilon \rightharpoonup u$ weakly* in $L^\infty(0, T; L^2(\T^2))$ and $u^\varepsilon_0 \to u_0$ strongly in $L^2(\T^2)$. In this case, $\{u^\varepsilon\}_{\varepsilon > 0}$ is called a \emph{physical realization} of $u$. 
\end{definition}

Next we provide the proof of Theorem \ref{ThmCon}. In this regard, the following interpolation inequality involving sparse function
spaces plays a crucial role. This result is of independent interest. 

\begin{lemma}
\label{LemmaInt} Let $Q_{0}$ be a cube in $\mathbb{R}^{2}$ or $Q_{0}%
=\mathbb{T}^{2}$, and let $\Psi$ be an admissible decay. Assume that $f\in S_{\Psi}(Q_{0})\cap\dot{H}^{1}%
(Q_{0}).$ Then, with absolute constants, we have
\begin{equation}
\Vert f-f_{Q_{0}}\Vert_{L^{2}(Q_{0})}\lesssim\frac{\Psi(-\log r)}{r}\Vert
f\Vert_{S_{\Psi}(Q_{0})}+r\Vert\nabla f\Vert_{L^{2}(Q_{0})}%
,\qquad\forall r\in(0,1).\label{IntIneqS}%
\end{equation}

\end{lemma}

\begin{remark}[cf. eq. (11) in \cite{Ches}] Let us show how (\ref{IntIneqS}) can be
applied to produce a classical Gagliardo-Nirenberg inequality. Let $f\in L^{p}%
,\,p\in(1,2),$ then by Proposition \ref{ThmSparseLebesgue}, with decay $\Psi(t)=2^{-2t(1-\frac
{1}{p})},$
$
L^{p}(Q_{0})\hookrightarrow S_{\Psi}(Q_{0}).$
Applying (\ref{IntIneqS}) for this special decay, we obtain 
\begin{equation*}
\Vert f-f_{Q_{0}}\Vert_{L^{2}(Q_{0})}   \lesssim r^{1-\frac{2}{p}}\Vert f\Vert_{L^{p}(Q_{0})}+r\Vert\nabla
f\Vert_{L^{2}(Q_{0})}
\end{equation*}
for all $r \in (0, 1)$. 
Optimizing the right-hand side by equating both terms, i.e. selecting $r =\Big(  \frac{\Vert\nabla f\Vert_{L^{2}(Q_{0})}}{\Vert f\Vert_{L^{p}(Q_{0}%
)}} \Big)^{-\frac{p}{2}},$ we find
\[
\Vert f-f_{Q_{0}}\Vert_{L^{2}(Q_0)}\lesssim\Vert\nabla f\Vert_{L^{2}(Q_0)}^{1-\frac
{p}{2}}\Vert f\Vert_{L^{p}(Q_0)}^{\frac{p}{2}}.
\]
\end{remark}

\begin{proof}
[Proof of Lemma \ref{LemmaInt}]We will use the sparse characterization of
$L^{2}$ (cf. Theorem \ref{CorSparLeb} and \cite{DMComptes}):
\begin{equation}
\Vert f-f_{Q_{0}}\Vert_{L^{2}(Q_{0})}\approx\sup_{\mathcal{Q}=(Q_{i})_{i\in
I}\in S(Q_{0})}\left\{  \sum_{i\in I}\bigg(\frac{1}{|Q_{i}|}\int_{Q_{i}%
}|f-f_{Q_{i}}|\bigg)^{2}|Q_{i}|\right\}  ^{1/2}\label{v3.1}%
\end{equation}
and
\begin{equation}
\Vert f\Vert_{L^{2}(Q_{0})}\approx\sup_{\mathcal{Q} \in
S(Q_{0})}\left\{  \sum_{i\in I}\bigg(\frac{1}{|Q_{i}|}\int_{Q_{i}%
}|f|\bigg)^{2}|Q_{i}|\right\}  ^{1/2}.\label{v3.1new2}%
\end{equation}

To estimate the left-hand side of \eqref{IntIneqS} we use (\ref{v3.1}). Let
$f\in S_{\Psi}(Q_0)$ and  $\mathcal{Q} \in S(Q_{0}).$ Then,
for $M\in\mathbb{N}_{0}$ we have
\begin{equation}
\sum_{i\in I}\bigg(\frac{1}{|Q_{i}|}\int_{Q_{i}}|f-f_{Q_{i}}|\bigg)^{2}|Q_{i}|
  =J_{1}(M)+J_{2}(M),\label{v3.2new}%
\end{equation}
where%
\[
J_{1}(M):=\sum_{k=0}^{M}\sum_{Q_{i}\in \mathbb{D}_{k}(\mathcal{Q}%
)}\bigg(\frac{1}{|Q_{i}|}\int_{Q_{i}}|f-f_{Q_{i}}|\bigg)^{2}|Q_{i}|,
\]
\[
J_{2}(M):=\sum_{k=M+1}^{\infty}\sum_{Q_{i}\in \mathbb{D}%
_{k}(\mathcal{Q})}\bigg(\frac{1}{|Q_{i}|}\int_{Q_{i}}|f-f_{Q_{i}}|\bigg)^{2}|Q_{i}|.
\]

We estimate $J_{1}(M)$ and $J_{2}(M)$. Since $f\in S_{\Psi}(Q_0)$, we find
\begin{align}
J_{1} (M)&  \lesssim\sum_{k=0}^{M}\sum_{Q_{i}\in \mathbb{D}%
_{k}(\mathcal{Q})}\bigg(\frac{1}{|Q_{i}|}\int_{Q_{i}}|f|\bigg)^{2}|Q_{i}| \nonumber\\
&  \approx\sum_{k=0}^{M}2^{2k}\sum_{Q_{i}\in \mathbb{D}%
_{k}(\mathcal{Q})}\bigg(\int_{Q_{i}}|f|\bigg)^{2}\nonumber\\
&  \leq\sum_{k=0}^{M}2^{2k}s_{k+1}(f)^{2}\nonumber\\
&  \lesssim \Vert f\Vert_{S_{\Psi}(Q_{0})}^{2}\,\sum_{k=1}^{M}2^{2k}%
\Psi(k)^{2}\nonumber\\
&  \lesssim\Vert f\Vert_{S_{\Psi}(Q_{0})}^{2}\,2^{2M}\Psi
(M)^{2} \label{v3.4}%
\end{align}
where we have used Definition \ref{DefAdm}(i) in the last estimate.

To estimate $J_{2}(M)$ we will make use of the classical Poincar\'{e}
inequality:
\begin{equation}
\int_{Q}|f-f_{Q}|\lesssim\ell(Q)\int_{Q}|\nabla f|.\label{v3.3}%
\end{equation}
Then, by \eqref{v3.3} and \eqref{v3.1new2} applied to $|\nabla f|$,
\begin{align}
J_{2}(M) &  \lesssim\sum_{k=M+1}^{\infty}\sum_{Q_{i} \in 
\mathbb{D}_{k}(\mathcal{Q})}\bigg(\frac{\ell(Q_{i})}{|Q_{i}|}\int_{Q_{i}}|\nabla
f|\bigg)^{2}|Q_{i}|\nonumber\\
&  \approx\sum_{k=M+1}^{\infty}2^{-k2}\sum_{Q_{i}\in 
\mathbb{D}_{k}(\mathcal{Q})}\bigg(\frac{1}{|Q_{i}|}\int_{Q_{i}}|\nabla
f|\bigg)^{2}|Q_{i}|\nonumber\\
&  \lesssim\Vert\nabla f\Vert_{L^{2}(Q_{0})}^{2}\,\sum_{k=M+1}^{\infty}%
2^{-k2} \quad \text{(by H\"{o}lder's inequality)}\nonumber\\
&  \approx\Vert\nabla f\Vert_{L^{2}(Q_{0})}^{2}\,2^{-M2}.\label{v3.5}%
\end{align}

Inserting the estimates \eqref{v3.4} and \eqref{v3.5} into \eqref{v3.2new}, we
achieve
\[
\left\{  \sum_{i\in I}\bigg(\frac{1}{|Q_{i}|}\int_{Q_{i}}|f-f_{Q_{i}%
}|\bigg)^{2}|Q_{i}|\right\}  ^{1/2}\lesssim\Vert f\Vert_{S_{\Psi
}(Q_{0})}2^{M}\Psi(M)+\Vert\nabla f\Vert_{L^{2}(Q_{0})}\,2^{-M}.
\]
Since this bound is independent of the sparse family $\mathcal{Q}$, we arrive
at (cf. \eqref{v3.1})
\[
\Vert f-f_{Q_{0}}\Vert_{L^{2}(Q_{0})}\lesssim\Vert f\Vert_{S_{\Psi
}(Q_{0})}2^{M}\Psi(M)+\Vert\nabla f\Vert_{L^{2}(Q_{0})}\,2^{-M}.
\]
Since $\Psi$ is decreasing, the previous estimate can be
expressed as
\[
\Vert f-f_{Q_{0}}\Vert_{L^{2}(Q_{0})}\lesssim\frac{\Psi(-\log r)}{r}\,\Vert
f\Vert_{S_{\Psi}(Q_{0})}+r\Vert\nabla f\Vert_{L^{2}(Q_{0})}%
\]
for all $r\in(0,1).$
\end{proof}

We are now ready to present the proof of Theorem \ref{ThmCon}. The
strategy of proof is inspired by \cite{Lan}, we have replaced the role played
there by structure functions with our decays of sparse indices. We provide
full details for the sake of completeness.

\begin{proof}
[Proof of Theorem \ref{ThmCon}] Let $\{u^\varepsilon\}_{\varepsilon > 0}$ be a physical realization of $u$ and let $\{\omega^{\varepsilon}\}_{\varepsilon > 0}$ be the related vorticities. By assumption, there exists an admissible decay $\Psi$ such that  
\begin{equation}\label{109}
M:=\sup_{\varepsilon>0}\Vert\omega^{\varepsilon}\Vert
_{C(0,T;S_{\Psi}(\T^2))}<\infty.
\end{equation}
Furthermore,  $\omega^{\varepsilon}$ satisfies the transport
equation:
\[
\omega_{t}^{\varepsilon}+u^{\varepsilon}\cdot\nabla\omega^{\varepsilon}=\varepsilon\Delta\omega^{\varepsilon}%
\]
and $\text{div }u^{\varepsilon}=0$. Multiplying both sides of the previous equation by
$\omega^{\varepsilon}$ and integrating on $\mathbb{T}^{2}$ yields
\[
\frac{d}{dt}\,\Vert\omega^{\varepsilon}\Vert_{L^{2}(\T^2)}^{2}=-2\varepsilon\Vert\nabla\omega^{\varepsilon
}\Vert_{L^{2}(\T^2)}^{2}.
\]
Consequently, for any $\delta\in(0,T)$ and $t\in(\delta,T)$,
\begin{equation}
\Vert\omega^{\varepsilon}(t)\Vert_{L^{2}(\T^2)}^{2}=\Vert\omega^{\varepsilon}(\delta)\Vert_{L^{2}(\T^2)
}^{2}-2\varepsilon\int_{\delta}^{t}\Vert\nabla\omega^{\varepsilon}(s)\Vert_{L^{2}(\T^2)}%
^{2}\,ds.\label{0.9}%
\end{equation}

According to Lemma \ref{LemmaInt} (with $Q_{0}=\mathbb{T}^{2}$ and\footnote{Note
that $\omega^\varepsilon$ has mean zero, i.e., $\omega^\varepsilon_{\mathbb{T}^{2}}=0$.}
$f=\omega^{\varepsilon}=\omega^{\varepsilon}(\cdot,t), \, t\in(0,T)$) and \eqref{109}, there exists a universal constant $C > 0$ such that
\[
\Vert\omega^{\varepsilon}\Vert_{L^{2}(\T^2)}^{2}\leq C\frac{\Psi(-\log r)^{2}}{r^{2}}%
M^{2}+Cr^{2}\Vert\nabla\omega^{\varepsilon}\Vert_{L^{2}(\T^2)}^{2}, \qquad \forall r \in (0, 1). 
\]
Integrating
\begin{equation}
\int_{\delta}^{t}\Vert\omega^{\varepsilon}(s)\Vert_{L^{2}(\T^2)}^{2}\,ds 
  \leq CTM^{2}\frac{\Psi(-\log r)^{2}}{r^{2}}+Cr^{2}\int_{\delta}^{t}%
\Vert\nabla\omega^{\varepsilon}(s)\Vert_{L^{2}(\T^2)}^{2}\, ds\label{v3.1.2}%
\end{equation}
for $r \in (0, 1)$. 
 In fact, letting $\Psi(t)=\Psi(0)$ for $t<0$,
\eqref{v3.1.2} with $r\geq1$ follows immediately from Poincar\'{e}'s
inequality\footnote{By the Poincar\'e inequality $\|\omega^\varepsilon\|_{L^2(\T^2)} \lesssim \|\nabla \omega^\varepsilon\|_{L^2(\T^2)}$, we have, for $r \geq 1$, \begin{align*}\int_\delta^t \|\omega^\varepsilon(s)\|^2_{L^2(\T^2)} \, ds \lesssim \int_t^\delta \|\nabla \omega^\varepsilon(s)\|^2_{L^2(\T^2)} \, ds \leq r^2 \int_t^\delta \|\nabla \omega^\varepsilon(s)\|^2_{L^2(\T^2)} \, ds. \end{align*}}. Next we optimize the right-hand side of \eqref{v3.1.2}. Indeed,
setting
\begin{equation*}
r_{0}:=\log\frac{(\int_{\delta}^{t}\Vert\nabla\omega^{\varepsilon}(s)\Vert_{L^{2}(\T^2)}%
^{2}\,ds)^{1/4}}{\Psi(0)^{1/2}}
\end{equation*}
and
\[
r=\frac{\Psi(r_{0})^{1/2}}{(\int_{\delta}^{t}\Vert\nabla\omega^{\varepsilon}%
(s)\Vert_{L^{2}(\T^2)}^{2}\,ds)^{1/4}}.
\]
Note that $-\log r\geq r_{0}$ (since $\Psi$ is decreasing) and thus
$\Psi(-\log r)\leq\Psi(r_{0})$. Accordingly, it follows from \eqref{v3.1.2}
that
\[
\bigg(\int_{\delta}^{t}\Vert\omega^{\varepsilon}(s)\Vert_{L^{2}(\T^2)}^{2}\,ds\bigg)^{2}\leq
C^{2}(TM^{2}+1)^{2}\Psi\left(  \log\frac{(\int_{\delta}^{t}\Vert\nabla
\omega^{\varepsilon}(s)\Vert_{L^{2}(\T^2)}^{2}\,ds)^{1/4}}{\Psi(0)^{1/2}}\right)  ^{2}%
\int_{\delta}^{t}\Vert\nabla\omega^{\varepsilon}(s)\Vert_{L^2(\T^2)}^{2}\,ds.
\]
Setting $x_{\varepsilon}=x_{\varepsilon}(t)=\varepsilon\int_{\delta}^{t}\Vert\omega^{\varepsilon}%
(s)\Vert_{L^{2}(\T^2)}^{2}\,ds$ and $y_{\varepsilon}=y_{\varepsilon}(t)=\int_{\delta}^{t}\Vert
\nabla\omega^{\varepsilon}(s)\Vert^{2}_{L^2(\T^2)}\,ds$, the previous estimate can be rewritten
as
\begin{equation}
\Big(\frac{x_{\varepsilon}}{\varepsilon}\Big)^{2}\leq f(y_{\varepsilon})\label{v3.1.4}%
\end{equation}
where $f(y)=C^{2}(TM^{2}+1)^{2}y\Psi\big(\log\frac{y^{1/4}}{\Psi(0)^{1/2}%
}\big)^{2}$. The function $f$ satisfies
\[
\sup_{y>0}\frac{f(y)}{y}=C^{2}(TM^{2}+1)^{2}\sup_{y>0}\Psi\bigg(\log
\frac{y^{1/4}}{\Psi(0)^{1/2}}\bigg)^{2}=C^{2}(TM^{2}+1)^{2}\Psi(0)^{2}<\infty,
\]
and (recall that $\lim_{y\rightarrow\infty}\Psi(y)=0$)
\[
\limsup_{y\rightarrow\infty}\frac{f(y)}{y}=C^{2}(TM^{2}+1)^{2}\limsup
_{y\rightarrow\infty}\Psi\bigg(\log\frac{y^{1/4}}{\Psi(0)^{1/2}}\bigg)^{2}=0.
\]
In addition $f(0)=0$ (note that $\lim_{y\rightarrow-\infty}\Psi(y)=\Psi(0)$).
Hence  \cite[Lemma C.1]{Lan} guarantees the
existence of a strictly increasing function $F$ with $F(y)\geq f(y)$ such that
the corresponding inverse function of $F$ can be expressed as $F^{-1}%
(x)=\sigma(\sqrt{x})x$, where $\sigma$ is a continuous increasing function
with $\sigma(\sqrt{x})\geq\sigma_{0}>0$ and $\lim_{x\rightarrow\infty}%
\sigma(x)=\infty$. From \eqref{v3.1.4}, we have
\[
\Big(\frac{x_{\varepsilon}}{\varepsilon}\Big)^{2}\leq F(y_{\varepsilon}),
\]
and thus
\[
\sigma\Big(\frac{x_{\varepsilon}}{\varepsilon}\Big)\Big(\frac{x_{\varepsilon}}{\varepsilon}\Big)^{2}%
=F^{-1}\Big(\Big(\frac{x_{\varepsilon}}{\varepsilon}\Big)^{2}\Big)\leq y_{\varepsilon}%
\]
or equivalently
\begin{equation}
-\varepsilon^{2}\int_{\delta}^{t}\Vert\nabla\omega^{\varepsilon}(s)\Vert_{L^2(\T^2)}^{2}\,ds\leq
-\sigma\Big(\frac{x_{\varepsilon}}{\varepsilon}\Big)x_{\varepsilon}^{2}.\label{v3.1.5}%
\end{equation}

Note that \eqref{0.9} can be rewritten as
\[
\frac{d}{dt}\,x_{\varepsilon}=\varepsilon\Vert\omega^{\varepsilon}(\delta)\Vert_{L^{2}(\T^2)}^{2}-2\varepsilon
^{2}\int_{\delta}^{t}\Vert\nabla\omega^{\varepsilon}(s)\Vert_{L^{2}(\T^2)}^{2}\,ds
\]
and then, by \eqref{v3.1.5} and well-known a priori $L^{2}$%
-estimates\footnote{Recall that $u_{0}\in L^{2}(\T^2)$, cf. Definition \ref{DefVV}.} for Navier-Stokes solutions (cf. \cite[Lemma A.2]{Lan}),
\begin{equation}
\frac{d}{dt}\,x_{\varepsilon}\leq\varepsilon\Vert\omega^{\varepsilon}(\delta)\Vert_{L^{2}(\T^2)}^{2}%
-2\sigma\Big(\frac{x_{\varepsilon}}{\varepsilon}\Big)x_{\varepsilon}^{2}\leq\frac{\Vert u_{0}%
\Vert_{L^{2}(\T^2)}^{2}}{\delta}-2\sigma\Big(\frac{x_{\varepsilon}}{\varepsilon}\Big)x_{\varepsilon}^{2}.
\label{v3.1.6}%
\end{equation}

Next we show that
\begin{equation}
\lim_{\varepsilon\rightarrow0}\varepsilon\int_{\delta}^{t}\Vert\omega^{\varepsilon}(s)\Vert_{L^{2}(\T^2)%
}^{2}\,ds=0 \label{v3.1.7}%
\end{equation}
uniformly with respect to $t\in\lbrack\delta,T]$. For $\eta>0$ arbitrary,
consider the set
\[
A_{\eta,t}:=\{\varepsilon>0:x_{\varepsilon}(t)\geq\eta\}.
\]
Assume first that $0\in\bar{A}_{\eta,t}$. In particular, there exists
$\{\varepsilon_{l}\}_{l\in\mathbb{N}}\subset A_{\eta,t}$ with $\lim_{l\rightarrow
\infty}\varepsilon_{l}=0$. Since $\sigma$ is increasing, $\sigma(x_{\varepsilon_{l}}%
(t)/\varepsilon_{l})\geq\sigma(\eta/\varepsilon_{l})$ and (cf. \eqref{v3.1.6})
\[
\frac{d}{dt}\,x_{\varepsilon_{l}}\leq\frac{\Vert u_{0}\Vert_{L^{2}(\T^2)}^{2}}{\delta
}-2\sigma\Big(\frac{\eta}{\varepsilon_{l}}\Big)\eta^{2}.
\]
Observe that $\lim_{l\rightarrow\infty}\sigma(\eta/\varepsilon_{l})=\infty$,
which yields that $x_{\varepsilon_{l}}(t)$ is decreasing with respect to $t$ whenever
$l\geq l_{0}=l_{0}(\eta,\sigma,\Vert u_{0}\Vert_{L^{2}(\T^2)},\delta)$. Since
$x_{\varepsilon_{l}}(\delta)=0$ and $x_{\varepsilon_{l}}(t)\geq0$, we conclude that
$x_{\varepsilon_{l}}(t)=0$ for all $t>\delta$ and $l\geq l_{0}$. Therefore there
exists $\varepsilon_{0}=\varepsilon_{0}(\eta,\sigma,\Vert u_{0}\Vert_{L^{2}(\T^2)},\delta)>0$
such that
\begin{equation}
x_{\varepsilon}(t)\leq\eta \qquad\text{if}\qquad\varepsilon\leq\varepsilon_{0}. \label{v3.1.8}%
\end{equation}
On the other hand, if $0\not \in \bar{A}_{\varepsilon,t}$ then \eqref{v3.1.8}
holds trivially. Either way, we have shown that \eqref{v3.1.7} is fulfilled.

Recall the energy formula for $2$D Navier-Stokes solutions
\[
\frac{d}{dt}\,\Vert u^{\varepsilon}\Vert_{L^{2}(\T^2)}^{2}=-2\varepsilon\Vert\omega^{\varepsilon}%
\Vert_{L^{2}(\T^2)}^{2}.
\]
Integrating over $[\delta,t]$:
\[
\Vert u^{\varepsilon}(t)\Vert_{L^{2}(\T^2)}^{2}-\Vert u^{\varepsilon}(\delta)\Vert_{L^{2}(\T^2)}^{2}%
=-2\varepsilon\int_{\delta}^{t}\Vert\omega^{\varepsilon}(s)\Vert_{L^{2}(\T^2)}^{2}\,ds.
\]
In light of \eqref{v3.1.7}, we get (uniformly in $t$)
\begin{equation}
\lim_{\varepsilon\rightarrow0}\Vert u^{\varepsilon}(t)\Vert_{L^{2}(\T^2)}^{2}-\Vert u^{\varepsilon}%
(\delta)\Vert_{L^{2}(\T^2)}^{2}=0. \label{v3.1.9}%
\end{equation}

Since $\{u^\varepsilon\}_{\varepsilon > 0}$ is sparse stable, by virtue of Theorem \ref{teo:exhaust}, one can find a sequence of $\varepsilon$'s with
$\varepsilon\rightarrow0$ such that
\[
\lim_{\varepsilon\rightarrow0}\Vert u^{\varepsilon}(t)\Vert_{L^{2}(\T^2)}^{2}=\Vert u(t)\Vert
_{L^{2}(\T^2)}^{2}%
\]
for all $t\in(0,T)$. This combined with \eqref{v3.1.9} lead to
\begin{equation}
\Vert u(t)\Vert_{L^{2}(\T^2)}=\Vert u(\delta)\Vert_{L^{2}(\T^2)},\qquad\text{for any}\quad
t\in(\delta,T). \label{v3.1.10}%
\end{equation}

On the other hand, since $\Vert u(t)\Vert_{L^{2}(\T^2)}$ is right-continuous at
$t=0$, given any $\eta>0$ one can find $\delta>0$ (depending on $\eta$) such that
\[
0\leq\Vert u(t)\Vert_{L^{2}(\T^2)}-\Vert u_{0}\Vert_{L^{2}(\T^2)}\leq\eta
,\qquad\text{for all}\quad t\in(0,\delta].
\]
Since \eqref{v3.1.10} also holds, we have
\[
0\leq\Vert u(t)\Vert_{L^{2}(\T^2)}-\Vert u_{0}\Vert_{L^{2}(\T^2)}\leq\eta
,\qquad\text{for all}\quad t\in(0,T].
\]
Clearly this shows that $u$ is conservative, i.e., $\Vert u(t)\Vert_{L^{2}(\T^2)
}=\Vert u_{0}\Vert_{L^{2}(\T^2)}, \, t\in\lbrack0,T]$.
\end{proof}

%
%
%
%

\end{document}